\documentclass[12pt]{amsart}
\usepackage[T1]{fontenc}
\usepackage[dvipsnames, table]{xcolor}
\usepackage{mathtools,amssymb,amsthm,tikz,tikz-cd,multirow,stmaryrd}
\SetSymbolFont{stmry}{bold}{U}{stmry}{m}{n} 
\usepackage{fullpage}
\usepackage{etoolbox}
\usepackage{extarrows}
\usepackage[shortlabels]{enumitem}
\usepackage{makecell}
\usepackage{tensor}
\usepackage{tabu}
\usepackage{hyperref}
\usepackage{tabularray}
\usepackage[most]{tcolorbox}

\usepackage{mathscinet} 

\usepackage{ifthen}

\usepackage{mathdots}

\usepackage{caption}
\usepackage{subcaption}
\usepackage[protrusion=true,expansion=true]{microtype}
\usepackage[capitalize, noabbrev]{cleveref} 
\crefname{equation}{}{}
\crefname{case}{case}{cases}
\Crefname{case}{Case}{Cases}

\setlist[itemize]{label={$\vcenter{\hbox{\tiny$\bullet$}}$}, leftmargin=2\parindent}
\setlist[enumerate,1]{label={\upshape(\roman*)}, leftmargin=2\parindent}

\DeclareMathAlphabet{\mathbbold}{U}{bbold}{m}{n}  
\newcommand{\id}{\operatorname{id}}

\DeclareRobustCommand{\SkipTocEntry}[5]{}

\renewcommand\leq{\leqslant}
\renewcommand\geq{\geqslant}
\newcommand\iso{\cong}

\DeclareMathOperator{\wt}{wt}
\DeclareMathOperator{\GTP}{GTP}
\DeclareMathOperator{\SSYT}{SSYT}
\DeclareMathOperator{\res}{res}
\DeclareMathOperator{\GL}{GL}
\DeclareMathOperator{\End}{End}

\newcommand{\Sch}{\mathrm{Sch}} 
\newcommand{\tab}{\mathrm{tab}} 
\newcommand{\TG}{\mathrm{T}_\Gamma} 
\newcommand{\TD}{\mathrm{T}_\Delta} 
\newcommand{\sh}{\vartheta} 
\newcommand{\M}{\mathbf{S}} 
\newcommand{\m}{\mathbf{s}} 
\newcommand{\mDG}{\mathfrak{p}^{\Delta}_{\Gamma}} 
\newcommand{\mGD}{\mathfrak{p}^{\Gamma}_{\Delta}} 

\newcommand{\Left}{\textnormal{\texttt{L}}}
\newcommand{\Right}{\textnormal{\texttt{R}}}
\newcommand{\Tr}{T_\Right}
\newcommand{\Tl}{T_\Left}
\newcommand{\Rrr}{R^\Right_\Right}
\newcommand{\Rll}{R^\Left_\Left}
\newcommand{\Rrl}{R^\Right_\Left}
\newcommand{\Rlr}{R^\Left_\Right}
\newcommand{\Vl}{V_\Left}
\newcommand{\Vr}{V_\Right}

\newcommand{\Sr}{\mathfrak{S}^\Right}

\newcommand{\SlN}[1]{\mathfrak{S}^{\Left, #1}}
\newcommand{\Zr}{Z^\Right}

\newcommand{\ZlN}[1]{Z^{\Left, #1}}

\newcommand{\Td}{T_\Delta}
\newcommand{\Tg}{T_\Gamma}
\newcommand{\Rdd}{R^\Delta_\Delta}
\newcommand{\Rgg}{R^\Gamma_\Gamma}
\newcommand{\Rdg}{R^\Delta_\Gamma}
\newcommand{\Rgd}{R^\Gamma_\Delta}

\newcommand{\rcw}[2][35]{\rotatebox[origin=c]{-#1}{$#2$}}
\newcommand{\rccw}[2][35]{\rotatebox[origin=c]{#1}{$#2$}}

\usetikzlibrary{decorations,decorations.pathreplacing,calligraphy,positioning,calc,tikzmark,fit,matrix,arrows.meta}
\tikzstyle{spin}=[circle, draw, fill=white, minimum size=16pt, inner sep=1pt, outer sep=0pt]
\tikzstyle{dot}=[fill=black, circle, inner sep=0pt, minimum size=4pt]

\tikzstyle{gamma}=[double distance=0.8pt]
\tikzstyle{delta}=[ultra thick]
\tikzstyle{halo}=[circle, fill=white, inner sep=0pt]
\tikzstyle{paths}=[colored/.style = {ultra thick}, plus/.style = {}, unknown/.style = {dashed, ultra thin}]
\tikzstyle{train}=[paths, scale=0.75, baseline={([yshift=-0.5ex]current  bounding  box.center)}]

\hypersetup{%
    pdftitle={The Schützenberger involution and colored lattice models},
    pdfauthor={Henrik P. A. Gustafsson, Carl Westerlund},
    linktocpage=true,           
    colorlinks=true,            
    linkcolor=blue!75!black,    
    citecolor=green!50!black,   
    filecolor=magenta,          
    urlcolor=blue!75!black      
}

\let\oldtocsection=\tocsection
\let\oldtocsubsection=\tocsubsection
\let\oldtocsubsubsection=\tocsubsubsection
\renewcommand{\tocsection}[2]{\bf\oldtocsection{#1}{#2}}
\renewcommand{\tocsubsection}[2]{\hspace{1.8em}\oldtocsubsection{#1}{#2}}
\renewcommand{\tocsubsubsection}[2]{\hspace{3.6em}\oldtocsubsubsection{#1}{#2}}

\newcommand{\cb}{\cellcolor{blue!15} \color{black}}
\newcommand{\hc}{|[highlighted entry]|}
\newcommand{\emptycol}{\hphantom{\hspace{0.8em} \centralLabel \hspace{0.8em}}}  

\tikzstyle{sipattern}=[
  baseline=-1mm,
  every left delimiter/.style={
    xshift=1ex,
  },
  every right delimiter/.style={
    xshift=-1ex,
  },
  every matrix/.style={
    matrix of nodes,
    nodes in empty cells,
    nodes={anchor=center},
    left delimiter=\{,
    right delimiter=\},
    every node/.style={
      text height=1.5ex,
      text depth=0,
    },
  },
  highlighted entry/.style={
    fill=blue!15,
  },
  arrow/.style={
    -Stealth,
    thick,
  },
  left arrow/.style={
    arrow,
    out=10,
    in=-20,
    looseness=1.7,
  },
  right arrow/.style={
    arrow,
    out=-10+180,
    in=20+180,
    looseness=1.7,
  },
]

\definecolor{green}{RGB}{0,180,0}

\newtheorem{theorem}{Theorem}[section]
\newtheorem{lemma}[theorem]{Lemma}
\newtheorem{proposition}[theorem]{Proposition}

\newtheorem{corollary}[theorem]{Corollary}

\newtheorem{maintheorem}{Theorem}

\theoremstyle{definition}
\newtheorem{remark}[theorem]{Remark}

\newtheorem{algorithm}[theorem]{Algorithm}

\newtheorem{example}[theorem]{Example}
\tcolorboxenvironment{example}{
skin=enhancedmiddle,breakable,left=5mm,
frame hidden,
boxsep=0pt,
before skip=10pt,after skip=10pt,
borderline west={2pt}{0pt}{gray!50},
colback=gray!7,
fontlower=\footnotesize,
fontupper=\footnotesize}

\numberwithin{equation}{section}

\tikzset{
  pics/fusedvertex/.style n args={4}{
    code={
      \draw[line width=1.2pt] (-0.7,0) node[left]{#1} -- (0.7,0) node[right]{#3};
      \draw[line width=4.5pt] (0,0.6) node[above]{#2} -- (0,-0.6) node[below]{#4};
    },
  },
}

\begin{document}

\title{The Schützenberger involution and colored lattice models}

\author{Henrik P. A. Gustafsson}
\address{Department of Mathematics and Mathematical Statistics, Umeå University, SE-901 87 Umeå, Sweden}
\email{henrik.gustafsson@umu.se}

\author{Carl Westerlund}
\address{Department of Mathematics and Mathematical Statistics, Umeå University, SE-901 87 Umeå, Sweden}
\email{carl.westerlund@umu.se}




\begin{abstract}
  Colored lattice models can be used to describe many different types of special functions of interest in both algebraic combinatorics and representation theory, for example Schur polynomials, nonsymmetric Macdonald polynomials, and characters and Whittaker functions for representations of $p$-adic groups.
  A notable example is the metaplectic ice model of which there are actually two different variants: a Gamma and a Delta variant.
  These variants differ in key aspects but surprisingly produce equal partition functions, which are weighted sums over admissible configurations, and this equality is called the Gamma-Delta duality.
  The duality was used to prove the analytic continuation of certain multiple Dirichlet series and is highly non-trivial, especially since the number of configurations on each side of the equality can differ.

  In this paper we generalize the Gamma-Delta duality by constructing a new family of solvable, colored lattice models and prove that they are dual to existing lattice models in the literature.
  Besides the above metaplectic case we also get a new lattice model and Gamma-Delta type duality for (non-metaplectic) Iwahori Whittaker functions together with a crystal limit for Demazure atoms for Cartan type $A$.
  The equality of partition functions is shown using Yang--Baxter equations involving both variants of lattice models simultaneously.

  For the crystal lattice model we show that the duality refines to a weight-respecting bijection of states given by the Schützenberger involution on the associated Gelfand--Tsetlin patterns or semistandard Young tableaux.
  We also show how the individual steps exchanging two rows in the proof of the duality for the partition functions refines to Berenstein--Kirillov or, equivalently, Bender--Knuth involutions.
\end{abstract}

\maketitle
\begin{samepage}
\tableofcontents
\end{samepage}

\section{Introduction}
\label{sec:introduction}
\subsection{Summary} In this paper we construct a new family of solvable, colored lattice models and prove that they are dual to existing models in the literature.
This new family of dualities generalize the so-called $\Gamma$-$\Delta$ duality of \cite{BBF:orange,BBB} which was used to prove analytic continuation for Weyl group multiple Dirichlet series in \cite{BBF:gauss}.
Specifically, we show that our newly constructed family is dual to the family of lattice models in~\cite{BBBG:duality} which includes the metaplectic ice model of \cite{BBCFG, BBB} and the Iwahori ice model of \cite{BBBG:Iwahori}.
These specializations describe different kinds of Whittaker functions in $p$-adic representation theory, but the representation theoretic origins of the lattice models will not be important for the arguments in this paper; we mention them mainly as motivation.
The lattice models are colored versions of six-vertex ice models whose states are described by $r$ paths through a finite two dimensional square lattice with $r$ rows, where each path is assigned a color from some fixed, ordered palette of size $m$.
The fact that they are solvable means that they satisfy Yang--Baxter equations, which give functional equations for the partition functions.
In our case these functional equations allow us to relate different partition functions using Demazure-type operators. 

By \emph{dual} lattice models we mean that there is a bijection of boundary data such that the corresponding partition functions are equal.
The partition functions are weighted sums over the admissible path configurations with the given boundary data and are functions in complex row parameters $z_1, z_2, \ldots, z_r$ that also depend on some other parameters that we specify below. 
The weight of a configuration is called a \emph{Boltzmann weight} and is a product over vertices determined by local edge configurations around each vertex using the above parameters as specified by the tables we present below.
The duality generalizes the $\Gamma$-$\Delta$ duality for the metaplectic ice models constructed in \cite{BBCFG, BBB} and does not, in general, refine to a bijection of individual states; the duality generally depends on intricate cancellations between states within each partition function and the two partition functions may not even involve the same number of states.

We will describe the distinguishing features of the duality which are retained in our generalization below, but let us first discuss the significance and development of the original $\Gamma$-$\Delta$ duality. The metaplectic ice models of \cite{BBCFG, BBB} were constructed to compute local $p$-adic contributions, called $p$-parts, to Weyl group multiple Dirichlet series.
These $p$-parts were originally computed as weighted sums of certain Gelfand--Tsetlin patterns, or equivalently semistandard Young tableaux, and the original $\Gamma$-$\Delta$ duality was conjectured in \cite{BBF:gauss} as an equality of two differently weighted sums of subsets of Gelfand--Tsetlin patterns related by the Schützenberger involution.
These two different ways of weighting the sums come from two particular choices of reduced words representing the long element $w_0$ of $S_r$ called $\Gamma$ and $\Delta$, namely $(\underbracket[0.4pt][1pt]{1\vphantom{,}},\underbracket[0.4pt][1pt]{2,1},\underbracket[0.4pt][1pt]{3,2,1},\ldots,\underbracket[0.4pt][1pt]{r,r-1,\ldots,1})$ and $(\underbracket[0.4pt][1pt]{r\vphantom{,}},\underbracket[0.4pt][1pt]{r-1,r},\underbracket[0.4pt][1pt]{r-2,r-1,r},\ldots,\underbracket[0.4pt][1pt]{1,2,\ldots,r})$ respectively.

It was shown in~\cite{BBF:gauss} that the (then still conjectured) $\Gamma$-$\Delta$ duality implies fundamental functional equations and an analytic continuation for the multiple Dirichlet series.
An intricate combinatorial proof of the duality, based on a series of different equivalent restatements, was later achieved in the book~\cite{BBF:orange}.

In~\cite{BBCFG} the two differently weighted sums of Gelfand--Tsetlin patterns were reinterpreted in terms of partition functions for two lattice models also called $\Gamma$ and $\Delta$.
The proof of the $\Gamma$-$\Delta$ duality, which was now reformulated as an equality of two partition functions, could then be greatly simplified in~\cite{BBB} using this lattice model perspective and, in particular, a Yang--Baxter equation mixing the two lattice model variants.

In this paper we generalize the $\Gamma$-$\Delta$ duality in the following way.
We construct a new continuous family of solvable, colored lattice models and show that each member is dual to a member of the family of lattice models in~\cite{BBBG:duality} which interpolates between the six-vertex Iwahori ice model of \cite{BBBG:Iwahori} and the metaplectic $\Delta$-ice model of \cite{BBCFG, BBB} using Drinfeld twists. 
Thus, the new family includes the existing metaplectic $\Gamma$-ice model as well as a new, dual Iwahori ice model.

Both families of lattice models consist of configurations, or states, of colored paths through a two-dimensional grid with $r$ rows and sufficiently many columns.
However, an important difference is that in the new family the paths are only moving down and to the left, while in the previously known family they are only moving down and to the right.
See \cref{fig:example_fused_left-moving_state,fig:example_fused_right-moving_state} for an example.
In both cases there are $r$ paths entering at the top boundary, no paths on the bottom boundary and $r$ paths exiting on the left or right boundary respectively.
The two families are not simply mirror images of each other; indeed the two partition functions may involve a different number of states as mentioned above.
Furthermore, the duality leaves the top boundary unchanged while the horizontal, outgoing boundary is not only changed from the left boundary to the right boundary but the colors are also reversed when read from top to bottom.
Lastly, the row parameters $z_1, \ldots z_r$ are also reversed under the duality.

\begin{figure}[htbp]
  \newcommand{\scl}{0.8} 
  \begin{subfigure}[t]{0.33\textwidth}
    \centering
    \begin{tikzpicture}[scale=\scl]
      \foreach \x in {0,...,4} {
        \draw (\x, -0.5) -- (\x, 2.5);
      }
      \foreach \y/\i in {0/3,1/2,2/1} {
        \draw (-0.5, \y) -- (4.5, \y);
      }

      \newcommand{\lineWidth}{1}
      \pgfmathsetmacro{\pathShift}{\lineWidth/(10*\scl)}
      \draw[blue, line width=\lineWidth mm] (4,2.5) -- (4,2) -- (2,2) -- (2,1) -- (1,1) -- (1,0) -- (-0.5,0);
      \draw[green, line width=\lineWidth mm] (2,2.5) -- (2,2) -- (-0.5,2);
      \draw[red, line width=\lineWidth mm] (0,2.5) -- (0,1) -- (-0.5,1);
    \end{tikzpicture}
    \caption{An example state for a left-moving model.}
    \label{fig:example_fused_left-moving_state}
  \end{subfigure}%
  \begin{subfigure}[t]{0.33\textwidth}
    \centering
    \begin{tikzpicture}[scale=\scl]
      \foreach \x in {0,...,4} {
        \draw (\x, -0.5) -- (\x, 2.5);
      }
      \foreach \y/\i in {0/3,1/2,2/1} {
        \draw (-0.5, \y) -- (4.5, \y);
      }

      \newcommand{\lineWidth}{1}
      \pgfmathsetmacro{\pathShift}{\lineWidth/(10*\scl)}
      \draw[blue, line width=\lineWidth mm] (4,2.5) -- (4,2) -- (4.5,2);
      \draw[green, line width=\lineWidth mm] (2,2.5) -- (2,0) -- (3,0) -- (4.5,0);
      \draw[red, line width=\lineWidth mm] (0,2.5) -- (0,2) -- (2+\pathShift,2) -- (2+\pathShift,1) -- (3,1) -- (4.5,1);
    \end{tikzpicture}
    \caption{An example state for a right-moving model.}
    \label{fig:example_fused_right-moving_state}
  \end{subfigure}%
  \begin{subfigure}[t]{0.33\textwidth}
    \centering
    \begin{tikzpicture}[scale=\scl]
      \foreach \x in {0,...,4} {
        \draw (\x, -0.5) -- (\x, 2.5);
      }
      \foreach \y/\i in {0/3,1/2,2/1} {
        \draw (-0.5, \y) -- (4.5, \y);
      }

      \newcommand{\lineWidth}{1}
      \pgfmathsetmacro{\pathShift}{\lineWidth/(10*\scl)}
      \draw[blue, line width=\lineWidth mm] (4,2.5) -- (4,2) -- (4.5,2);
      \draw[green, line width=\lineWidth mm] (2,2.5) -- (2,1) -- (-0.5,1);
      \draw[red, line width=\lineWidth mm] (0,2.5) -- (0,2) -- (3,2) -- (3,0) -- (4.5,0);
    \end{tikzpicture}
    \caption{An example state for a mixed model.}
    \label{fig:example_fused_mixed_state}
  \end{subfigure}%
  \caption{Examples of colored states.}
  \label{fig:examples_left_and_right_moving}
  \vspace{-1em}
\end{figure}

In more detail, there is a family of right-moving solvable lattice models $\Right(\Phi,\mathfrak{X})$ parametrized by $\Phi$ and $\mathfrak{X}$ (see \cref{sec:fusion}) and a deformation parameter $q$ that was originally constructed in \cite{BBBG:duality} and includes the Iwahori and metaplectic specializations mentioned above.
For a member $\Right(\Phi,\mathfrak{X})$ of this lattice model family, denote by~$\mathfrak{S}^{\Right(\Phi,\mathfrak{X})}_{\mu, \sigma}$ the system of states with top boundary conditions specified by $\mu \in (\mathbb{Z}_{\geq 0})^r$ and right boundary colors specified from the top down by an $r$-tuple~$\sigma$ of colors (this will be described more precisely in \cref{sec:fusion}).
Let $Z(\mathfrak{S}^{\Right(\Phi,\mathfrak{X})}_{\mu, \sigma})(\mathbf{z})$ be the corresponding partition function with row parameters $\mathbf{z} = (z_1,\ldots, z_r) \in (\mathbb{C}^\times)^r$ from the top down.
Note that it is only the Boltzmann weights in the partition function that depend on $\Phi$ and $\mathfrak{X}$; the set of admissible states does not.

We have in this paper constructed a new family of left-moving solvable lattice models $\Left(\Phi, \mathfrak{X})$ with similarly defined systems $\mathfrak{S}^{\Left(\Phi,\mathfrak{X})}_{\mu, \sigma}$ but where $\sigma$ now denotes the colors on the left boundary (still from the top down).
The generalized $\Gamma$-$\Delta$ duality, which we in this paper call the left-right duality (for reasons explained below), then states:
\begin{maintheorem}
  \label{mainthm:duality}
  For any right-moving lattice model $\Right(\Phi,\mathfrak{X})$ there is a left-moving lattice model $\Left(\Phi, \mathfrak{X})$ (and vice versa) such that
    \begin{equation*}
      Z(\mathfrak{S}^{\Right(\Phi,\mathfrak{X})}_{\mu, \sigma})(\mathbf{z}) = Z(\mathfrak{S}^{\Left(\Phi,\mathfrak{X})}_{\mu, w_0 \sigma})(w_0\mathbf{z})
    \end{equation*}
    for all $\mu \in (\mathbb{Z}_{\geq0})^r$ and color $r$-tuples $\sigma$, where $w_0$ is the long Weyl element in the symmetric group $S_r$.
\end{maintheorem}
\noindent The proof, and a more detailed statement, is given by \cref{thm:left-right_duality} in \cref{sec:left-right_duality}.

We have mentioned that there is in general no bijection between the states $\mathfrak{S}^{\Right(\Phi,\mathfrak{X})}_{\mu, \sigma}$ and $\mathfrak{S}^{\Left(\Phi,\mathfrak{X})}_{\mu, w_0 \sigma}$ since these sets may have different cardinality (an example of this can be seen in \cref{fig:left-right_duality_example}).
However, there is an important limit where the lattice models degenerate into five-vertex models and in this special case there is a bijection between the states that preserves the weight of each state (up to a reversal of the spectral parameters $z_1, \dots, z_r$ as in \cref{mainthm:duality}).
In other words, the left-right duality refines in this limit to a state-by-state bijection which we now describe.

\begin{figure}
  \newcommand{\scl}{0.58}
  \begin{tikzpicture}[scale=\scl]
    \newcommand{\lineWidth}{0.7}
    \pgfmathsetmacro{\pathShift}{\lineWidth/(10*\scl)}
    \newcommand{\stateXShift}{4.8}
    \newcommand{\stateYShift}{5}
    \begin{scope}
      \begin{scope}
        \foreach \x in {0,...,2} {
          \draw (\x, -0.5) -- (\x, 2);
        }
        \foreach \y in {0,...,2} {
          \draw (-0.5, \y) -- (2.5, \y);
        }
        \foreach \i in {1,...,3} {
          \node at (-1, 3-\i) {\footnotesize$z_\i$};
        }

        \draw[blue, line width=\lineWidth mm] (2, 2.5) -- (2, 0) -- (-0.5, 0);
        \draw[green, line width=\lineWidth mm] (0, 2.5) -- (0, 2) -- (-0.5, 2);
        \draw[red, line width=\lineWidth mm] (1, 2.5) -- (1, 1) -- (-0.5, 1);

        \node at (1, -1) {\footnotesize$-v^3 z_1^2 z_2$};
      \end{scope}

      \begin{scope}[xshift=\stateXShift cm]
        \foreach \x in {0,...,2} {
          \draw (\x, -0.5) -- (\x, 2);
        }
        \foreach \y/\i in {0/3,1/2,2/1} {
          \draw (-0.5, \y) -- (2.5, \y);
        }

        \draw[blue, line width=\lineWidth mm] (2, 2.5) -- (2, 1) -- (1, 1) -- (1, 0) -- (-0.5, 0);
        \draw[green, line width=\lineWidth mm] (0, 2.5) -- (0, 2) -- (-0.5, 2);
        \draw[red, line width=\lineWidth mm] (1, 2.5) -- (1, 1) -- (-0.5, 1);

        \node at (1, -1) {\footnotesize$(1-v)v^2 z_1^2 z_3$};
      \end{scope}

      \begin{scope}[xshift=2*\stateXShift cm]
        \foreach \x in {0,...,2} {
          \draw (\x, -0.5) -- (\x, 2);
        }
        \foreach \y/\i in {0/3,1/2,2/1} {
          \draw (-0.5, \y) -- (2.5, \y);
        }

        \draw[blue, line width=\lineWidth mm] (2, 2.5) -- (2, 2) -- (0, 2) -- (0, 0) -- (-0.5, 0);
        \draw[green, line width=\lineWidth mm] (0, 2.5) -- (0, 2) -- (-0.5, 2);
        \draw[red, line width=\lineWidth mm] (1, 2.5) -- (1, 1) -- (-0.5, 1);

        \node at (1, -1) {\footnotesize$(1-v)v z_2 z_3^2$};
      \end{scope}

      \begin{scope}[xshift=3*\stateXShift cm]
        \foreach \x in {0,...,2} {
          \draw (\x, -0.5) -- (\x, 2);
        }
        \foreach \y/\i in {0/3,1/2,2/1} {
          \draw (-0.5, \y) -- (2.5, \y);
        }

        \draw[blue, line width=\lineWidth mm] (2, 2.5) -- (2, 2) -- (1, 2) -- (1, 0) -- (-0.5, 0);
        \draw[green, line width=\lineWidth mm] (0, 2.5) -- (0, 2) -- (-0.5, 2);
        \draw[red, line width=\lineWidth mm] (1-\pathShift, 2.5) -- (1-\pathShift, 1) -- (-0.5, 1);

        \node at (1, -1) {\footnotesize$(1-v)v^2 z_1 z_2 z_3$};
      \end{scope}
    \end{scope}

    \begin{scope}[yshift=-\stateYShift cm]
      \begin{scope}
        \foreach \x in {0,...,2} {
          \draw (\x, -0.5) -- (\x, 2);
        }
        \foreach \y/\i in {0/1,1/2,2/3} {
          \draw (-0.5, \y) -- (2.5, \y);
        }
        \foreach \i in {1,...,3} {
          \node at (-1, \i-1) {\footnotesize$z_\i$};
        }

        \draw[blue, line width=\lineWidth mm] (2, 2.5) -- (2, 2) -- (2.5, 2);
        \draw[green, line width=\lineWidth mm] (0, 2.5) -- (0, 0) -- (2.5, 0);
        \draw[red, line width=\lineWidth mm] (1, 2.5) -- (1, 1) -- (2.5, 1);

        \node at (1, -1) {\footnotesize$-v^3 z_1^2 z_2$};
      \end{scope}

      \begin{scope}[xshift=\stateXShift cm]
        \foreach \x in {0,...,2} {
          \draw (\x, -0.5) -- (\x, 2);
        }
        \foreach \y/\i in {0/1,1/2,2/3} {
          \draw (-0.5, \y) -- (2.5, \y);
        }

        \draw[blue, line width=\lineWidth mm] (2, 2.5) -- (2, 2) -- (2.5, 2);
        \draw[green, line width=\lineWidth mm] (0, 2.5) -- (0, 0) -- (2.5, 0);
        \draw[red, line width=\lineWidth mm] (1, 2.5) -- (1, 2) -- (2, 2) -- (2, 1) -- (2.5, 1);

        \node at (1, -1) {\footnotesize$(1-v)v^2 z_1^2 z_3$};
      \end{scope}

      \begin{scope}[xshift=2*\stateXShift cm]
        \foreach \x in {0,...,2} {
          \draw (\x, -0.5) -- (\x, 2);
        }
        \foreach \y/\i in {0/1,1/2,2/3} {
          \draw (-0.5, \y) -- (2.5, \y);
        }

        \draw[blue, line width=\lineWidth mm] (2, 2.5) -- (2, 2) -- (2.5, 2);
        \draw[green, line width=\lineWidth mm] (0, 2.5) -- (0, 2) -- (2, 2) -- (2, 0) -- (2.5, 0);
        \draw[red, line width=\lineWidth mm] (1, 2.5) -- (1, 1) -- (2.5, 1);

        \node at (1, -1) {\footnotesize$(1-v)v z_2 z_3^2$};
      \end{scope}

      \begin{scope}[xshift=3*\stateXShift cm]
        \foreach \x in {0,...,2} {
          \draw (\x, -0.5) -- (\x, 2);
        }
        \foreach \y/\i in {0/1,1/2,2/3} {
          \draw (-0.5, \y) -- (2.5, \y);
        }

        \draw[blue, line width=\lineWidth mm] (2, 2.5) -- (2, 2) -- (2.5, 2);
        \draw[green, line width=\lineWidth mm] (0, 2.5) -- (0, 2) -- (1, 2) -- (1, 0) -- (2.5, 0);
        \draw[red, line width=\lineWidth mm] (1+\pathShift, 2.5) -- (1+\pathShift, 1) -- (2.5, 1);

        \node at (1, -1) {\footnotesize$(1-v)v z_1 z_2 z_3$};
      \end{scope}

      \begin{scope}[xshift=4*\stateXShift cm - 0.5cm]
        \foreach \x in {0,...,2} {
          \draw (\x, -0.5) -- (\x, 2);
        }
        \foreach \y/\i in {0/1,1/2,2/3} {
          \draw (-0.5, \y) -- (2.5, \y);
        }

        \draw[blue, line width=\lineWidth mm] (2, 2.5) -- (2, 2) -- (2.5, 2);
        \draw[green, line width=\lineWidth mm] (0, 2.5) -- (0, 1) -- (1, 1) -- (1, 0) -- (2.5, 0);
        \draw[red, line width=\lineWidth mm] (1, 2.5) -- (1, 2) -- (2, 2) -- (2, 1) -- (2.5, 1);

        \node at (1.3, -1) {\footnotesize$-(1-v)^2 v z_1 z_2 z_3$};
      \end{scope}
    \end{scope}
    \pgfmathsetmacro{\stateXSeparation}{\stateXShift-2-2*0.5}
    \pgfmathsetmacro{\leftPadding}{0.5+\stateXSeparation/2}
    \begin{scope}
      \begin{scope}
        \draw (-\leftPadding-0.3, -1.7-\stateYShift) rectangle (\stateXShift-\leftPadding, 3);
      \end{scope}
      \begin{scope}[xshift=\stateXShift cm]
        \draw (-\leftPadding, -1.7-\stateYShift) rectangle (\stateXShift-\leftPadding, 3);
      \end{scope}
      \begin{scope}[xshift=2*\stateXShift cm]
        \draw (-\leftPadding, -1.7-\stateYShift) rectangle (\stateXShift-\leftPadding, 3);
      \end{scope}
      \begin{scope}[xshift=3*\stateXShift cm]
        \draw (-\leftPadding, -1.7-\stateYShift) rectangle (\stateXShift-\leftPadding + \stateXShift, 3);
      \end{scope}
    \end{scope}
  \end{tikzpicture}
  \caption{An example of the left-right duality in the Iwahori case. The top boundary conditions are the same, but on the left/right boundary they have been reversed. Note that the number of left-moving states is not the same as the number of right-moving states, so the duality cannot in general be refined by a weight-preserving bijection of states. However, often it can be partially refined by grouping the states into ``packets'' such that there is a weight-preserving bijection between the packets. In this case it can be refined into four packets, three of which are singletons on both sides, and one which is a singleton on the left-moving side but a pair of states on the right-moving side.}
  \label{fig:left-right_duality_example}
\end{figure}

If we specialize to the right-moving Iwahori lattice model $\Right(\text{Iwahori})$ and the new left-moving dual lattice model $\Left(\text{Iwahori})$ constructed in this paper and take the crystal limit $v = 1/q^2 \to 0$ the lattice models become five-vertex models which we denote by $\Right(\text{crystal})$ and $\Left(\text{crystal})$ because of their connection to Kashiwara--Nakashima crystals.
The model $\Right(\text{crystal})$ was constructed in \cite{BBBG:demazure} and used to compute Demazure atoms and characters which will be introduced in more detail in \cref{sec:intro-crystal}.
We have collected our crystal limit results in this separate subsection for clarity, but one of the main results is summarized as follows.
\begin{maintheorem}
  \label{mainthm:crystal}
  There is a bijection $\Sch : \mathfrak{S}^{\Right(\mathrm{crystal})}_{\mu, \sigma} \xlongrightarrow{\sim} \mathfrak{S}^{\Left(\mathrm{crystal})}_{\mu, w_0 \sigma}$ given by the Schützenberger involution on the associated semistandard Young tableaux such that the weights of the states~$\mathfrak{s} \in \mathfrak{S}^{\Right(\mathrm{crystal})}_{\mu, \sigma}$ satisfy
    \begin{equation*}
      \wt(\Sch(\mathfrak{s}))(\mathbf{z}) = \wt(\mathfrak{s})(w_0\mathbf{z}).
    \end{equation*}
\end{maintheorem}
\noindent The proof, and a more detailed statement, is given by \cref{thm:ti-boundary-conditions} and \cref{cor:Schützenberger_on_states} in \cref{sec:crystal}.

More generally, it is also possible to consider systems with any combination of left-moving and right-moving rows (both for the general six-vertex model families and the crystal limit).
We call these \emph{mixed models} and we show that they all compute the same partition functions as the two special cases above where all rows are left-moving or all rows are right-moving, with appropriate mappings of boundary conditions.
\cref{fig:examples_left_and_right_moving} contains examples of states for a left-moving model, a right-moving model and a mixed model where the rows are right-, left- and then right-moving.

Furthermore, we prove in \cref{thm:YBE} that the two families form a \emph{system of solvable lattice models}, meaning that they satisfy all Yang--Baxter equations involving only left-moving rows, only right-moving rows and all possible mixtures of both (for both the so called RRR and RTT variants).
The proof is self-contained and recovers the non-mixed Yang--Baxter equations for the known right-moving Iwahori ice model originally proved in~\cite{BBBG:Iwahori} as well as the mixed and non-mixed Yang--Baxter equations and both variants of the metaplectic ice models proved in~\cite{BBB, BBBGray}.
As will be further discussed in the next section the latter was done before the reinterpretation of the models in terms of colored paths which greatly simplifies the descriptions of both the states and their weights.
The mixed Yang--Baxter equations are used in a standard \emph{train argument} which swaps two rows of different types in the lattice models which forms an important step in proving that the left- and right-moving families are dual.

\subsection{Background and motivation}
Although originating from statistical mechanics, lattice models have become a valuable tool in several mathematical fields including:
\begin{itemize}
  \item \textbf{algebraic combinatorics} in \cite{Kuperberg, BBF, BMN, Wheeler-ZinnJustin, BSW, BBBG:demazure, Buciumas-Scrimshaw, CGKM:LLT, Borodin-Wheeler, ABW} with topics ranging from alternating signs matrices, Schur polynomials, Lascoux--Schützenberger key polynomials, LLT polynomials, and both symmetric and nonsymmetric Macdonald polynomials.
  \item \textbf{algebraic geometry} in \cite{Knutson-Miller, BBL, BFHTW, Buciumas-Scrimshaw:Grothendieck} focusing on Schubert calculus and Bott--Samelson varieties.
  \item \textbf{representation theory} in \cite{Hamel-King,BBCFG,BBB,Ivanov,BBBG:hamiltonians,BBBG:Iwahori,BBBG:duality,BBBG:metahori} including Demazure characters and atoms, Whittaker functions of $p$-adic groups and their metaplectic covers.
\end{itemize}

Several of the lattice models listed above are lattice models described by colored paths.
An early example of this are the Perk--Schultz lattice models \cite{Perk-Schultz} where each edge is assigned exactly one color (in contrast to this paper) and present a large collection of solvable models including left- and right-moving six-vertex models.
The families presented in this paper are not included in that collection, but share an important feature for how two crossing colored paths are weighted as discussed further in \cref{sec:quantum-groups}.

Previously in the introduction we have highlighted two important colored lattice models: First, the Iwahori lattice models which computes all values of a basis of (non-metaplectic) Iwahori Whittaker functions on $\GL_r(F)$ where $F$ is a non-archimedean local field.
Second, the metaplectic ice model which computes $p$-parts of Weyl group multiple Dirichlet series or, equivalently, spherical Whittaker functions on metaplectic $n$-covers of $\GL_r(F)$.

The metaplectic ice models were not originally described in \cite{BBCFG, BBB} as lattice models of colored paths, but instead of paths where each edge was decorated by a \emph{charge} in $\mathbb{Z}/n\mathbb{Z}$ which varied along the path.
It was later recast in~\cite{BBBG:duality,BBBG:metahori} in terms of colors (or so called supercolors).

If we specialize to $n = 1$ the $\Gamma$ and $\Delta$ metaplectic ice models reduce to uncolored $\Gamma$ and $\Delta$ lattice models described by non-distinguishable paths and were studied in \cite{BBF}.
By \emph{uncolored} we mean that the paths are not distinguishable, and they are technically of the same single color with a palette size of $n=1$.
The associated partition functions compute spherical Whittaker functions on $\GL_r(F)$ which are Schur polynomials times a deformed Weyl denominator factor.
The states can be described by Gelfand--Tsetlin patterns listing the column numbers of occupied vertical edges on each row and in this description the partition function becomes the Tokuyama formula \cite{Tokuyama} which interpolates between the Weyl character formula and the combinatorial description of Schur polynomials in terms of semistandard Young tableaux.
In fact, \cite{BBF} gave an alternative proof of a generalization of the Tokuyama formula from \cite{Hamel-King}.

If we specialize the number of colors in the right-moving Iwahori ice model of \cite{BBBG:Iwahori} to $m=1$ it also reduces to the uncolored $\Gamma$ lattice model of \cite{BBF}.
Thus, we have two different colored models reducing to the same uncolored model: the left-moving metaplectic $\Gamma$ ice model and the right-moving Iwahori lattice model.
This may at first seem counterintuitive since the paths in the original lattice models move in different directions, but when specializing to the uncolored model the notion of left- and right-moving becomes interchangeable.
This is because the edges can be in one of two configurations in the uncolored model (occupied or unoccupied denoted by $\ominus$ and $\oplus$) and left- and right-moving configurations are interchangeable by inverting each horizontal edge.

This means that the right-moving family contains both $\Delta$ metaplectic ice and an Iwahori lattice model that should therefore be labeled $\Gamma$ Iwahori ice (see \cref{fig:specializations}).
The newly constructed left-moving Iwahori ice model should in contrast be called the $\Delta$ Iwahori ice model and is in the same family as the left-moving $\Gamma$ metaplectic ice model.
Since the $\Gamma$ and $\Delta$ labels are not the same for all family members, we will instead refer to the two families as left- and right-moving, and only use $\Gamma$ and $\Delta$ in the two special cases of Iwahori and metaplectic ice.

\begin{figure}[htpb]
  \centering
  \begin{tikzpicture}[scale=1.25]
    \node[label={above:Left-moving family [This paper]}] (E) at (4,0) {};
    \node[label={below:Right-moving family \cite{BBBG:duality}}] (F) at (4,-3) {};
    \draw[|-|] (0,0) node[below=1mm] (A) {$\Gamma$ metaplectic ice \cite{BBB}} -- (8,0) node[below=1mm] (B) {$\Delta$ Iwahori ice [This paper]};
    \draw[|-|] (0,-3) node[above=1mm] (C) {$\Delta$ metaplectic ice \cite{BBB}} -- (8,-3) node[above=1mm] (D) {$\Gamma$ Iwahori ice \cite{BBBG:Iwahori}};
    \draw[Implies-Implies, double equal sign distance] (A) -- (C) node[midway, label={right:$\Gamma$-$\Delta$ duality \cite{BBB}}]{};
    \draw[Implies-Implies, double equal sign distance] (B) -- (D) node[midway, label={[right=0.75em,align=left]$\Gamma$-$\Delta$ duality\\{}[This paper]}]{};
    \draw[Implies-Implies, double distance = 5pt, shorten <= 5pt, shorten >= 5pt] (E) -- (F) node[midway, label={[right=0.75em, align=left]Left-right duality\\{}[This paper]}]{};
  \end{tikzpicture}
  \caption{The left- and right-moving families of lattice models together with their specializations, naming conventions and references.}
  \label{fig:specializations}
\end{figure}

To further explain how the different lattice models relate it is useful to take the uncolored lattice models of \cite{BBF} and their computation of spherical (non-metaplectic) Whittaker functions as a starting point.
We have seen two different ways of refining these models by introducing colors giving rise to (i) the Iwahori lattice model with $r$ colors and (ii) the metaplectic ice models with $n$ colors. On the representation theory side this corresponds to refining the spherical Whittaker functions on $\GL_r(F)$ to (i) Iwahori Whittaker functions or (ii) spherical Whittaker functions on metaplectic $n$-covers of $\GL_r(F)$. 
 
The refinement of spherical Whittaker functions into a sum of less symmetrical Iwahori Whittaker functions is very beneficial in representation theory because of their richer structure.
On the lattice model side this structure gives rise to Demazure type operators which relate different partition functions. The colorization of the uncolored model was first done in the five-vertex case in~\cite{BBBG:demazure} to compute Demazure atoms and characters.

In~\cite{BBBG:duality} Gustafsson and collaborators combined these into a family of right-moving lattice models interpolating between the Iwahori and the $\Delta$ metaplectic ice models through Drinfeld twists\footnote{See \cref{sec:quantum-groups}.} of the underlying affine supersymmetric quantum group $U_q(\widehat{\mathfrak{gl}}(m|1))$, where $m$ is the number of colors.
There the family was used to prove that the two seemingly very different generalizations of Whittaker functions described above are essentially equal.
We call this identity the \emph{Iwahori-metaplectic correspondence} and the $\Gamma$-$\Delta$ duality was used in the proof in~\cite{BBBG:duality}.
The lattice model interpretation also led to a conjectured parahoric-metaplectic correspondence that was later proved in~\cite{Venkateswaran} using representations of the double affine Hecke algebra.

The two generalization (i) and (ii) above can also be taken simultaneously to obtain metaplectic Iwahori Whittaker functions, which was done in~\cite{BBBG:metahori} where a lattice model was constructed combining two sets of colored paths (one left-moving and one right-moving) overlapping on vertical edges.
The right-moving Iwahori ice model and the left-moving metaplectic ice model are special cases of this combined model.

This interconnectedness between the different lattice models motivated the current paper.
In particular, the existence of the two variants $\Gamma$ and $\Delta$ for the metaplectic ice model, which are left- and right-moving respectively, suggested that there should be a left-moving variant of Iwahori ice as well, and in fact a whole family of left-moving lattice models dual to the right-moving family as illustrated in \cref{fig:specializations}.

An important application of our results is the resolution of an issue raised in~\cite{Buciumas-Scrimshaw}.
They construct the crystal limits of the Iwahori left- and right-moving lattice models and combine them in alternating rows to compute Demazure atoms and characters for Cartan type B and C.
They showed that these lattice models satisfy Yang--Baxter equations for mixed row types for three out of four possible row type combinations.
We resolve the obstruction they mentioned for finding an $R$-matrix satisfying the fourth Yang--Baxter equation (coming from internal color loops) as explained in \cref{rem:quasi-solvable}, meaning that we obtain all four mixed Yang--Baxter equations for the general family and therefore also in the crystal case.

Our results also provide a novel suggestion for addressing an open question that was raised in \cite{BBF:orange} for the original $\Gamma$-$\Delta$ duality for metaplectic ice.
There it was empirically observed that even though the duality does not refine to a state-by-state bijection in general, the equality arranges itself into small packets of states which can be matched on both sides.
However, it was unclear how these packets should be defined in general.
As a step towards identifying these packets, we propose that the packets should be invariant for the entire family.
Our state-by-state duality for the crystal limit in the Iwahori specialization then provides an important special case.
For more details see \cref{rem:packets}.

We are currently working on a project which uses the new left-moving Iwahori model constructed in this paper (together with the existing right-moving version) to construct a lattice model for Iwahori Whittaker functions of Cartan types B and C by alternating left- and right-moving rows.
We also plan to use these constructions to prove a generalized version of the Iwahori-metaplectic correspondence from \cite{BBBG:duality} for Whittaker functions which are both metaplectic and for Iwahori fixed vectors.

\subsection{Crystal specialization}
\label{sec:intro-crystal}
In this subsection we have collected our results for the crystal models. 
These models are obtained from taking the $v =1/q^2 \to 0$ limit of both the left- and right-moving Iwahori lattice models.
The states are in bijection with Gelfand--Tsetlin patterns and therefore also semistandard Young tableaux (SSYT).
It turns out that the left-right duality, which is an equality of two sums, can be refined to a term-by-term equality for the crystal models, and that the pairing of terms is given by the Schützenberger involution on the corresponding semistandard Young tableaux (\cref{cor:Schützenberger_on_states}).
Note that this refinement cannot be done for general $v$ since the number of terms in each partition function is not always equal (see \cref{fig:left-right_duality_example} for an example).

In \cite{BBBG:demazure}, Gustafsson and collaborators showed that the partition functions of the right-moving crystal model are \emph{Demazure atoms} of Cartan type $A$, which are polynomials in $z_1, \dots, z_r$ and certain building blocks of Demazure characters which in turn are generalizations of Schur polynomials.
We will focus only on Cartan type $A$ in this paper.

Demazure atoms were introduced in \cite{LS90} by Lascoux and Schützenberger where they were called \emph{standard bases} and Demazure characters were called \emph{key polynomials}. T
The Demazure characters and atoms can also be lifted to the Kashiwara--Nakashima crystal $\mathcal{B}_\lambda$ which can be realized on the set of semistandard Young tableaux with shape $\lambda$ \cite{Kashiwara-Nakashima,Kashiwara:Demazure,BumpSchilling}.
In this crystal perspective Demazure atoms (and therefore also characters) are certain subsets of tableaux $\mathcal{B}_\lambda^\circ(w) \subset \mathcal{B}_\lambda$ enumerated by Weyl elements $w \in S_r$ and summing over their weights as $\sum_{T \in \mathcal{B}_\lambda^\circ(w)}\mathbf{z}^{\wt(T)}$ recovers the associated polynomial.

In each Demazure atom there is a distinguished tableau called a \emph{key tableau} which was used to obtain combinatorial formulas for Demazure characters and atoms in \cite{LS90}.
Several algorithms have been developed for computing the key tableau in the same Demazure atom as a given tableau; see for example \cite{LS90,ReinerShimozono,LenartUnified,MasonRSK,MasonAtoms,AvalKeys,WillisThesis,WillisKey,Seaborn,ProctorWillis,WillisFilling,BBBG:demazure}.
As explained in \cite[Section~1.1]{BBBG:demazure} it is enough to compute the word $w$ for the Demazure atom the given tableau belongs to, since the key tableau is easily determined from this.
For other colored lattice models related to Demazure atoms, Demazure characters and crystals, including for other Cartan types, see also~\cite{Buciumas-Scrimshaw, Yang:closed, JNS}.

It was shown in~\cite{BBBG:demazure} that not only do the right-moving crystal partition functions give the polynomial Demazure atoms, but that there is a bijection between the corresponding system of states for given boundary conditions and the tableaux $\mathcal{B}_\lambda^\circ(w)$, and this bijection involves the Schützenberger involution on tableaux.
Since we show in the current paper that there is a state-by-state duality between the left- and right-moving crystal models given by the Schützenberger involution (\cref{mainthm:crystal}), we thus get a more natural and direct bijection between crystal Demazure atoms and \emph{left-moving} crystal lattice model systems as a result.
We also obtain an alternative proof of Algorithm~2 in~\cite{BBBG:demazure} following directly from our state-by-state duality instead of the crystal manipulations used in~\cite{BBBG:demazure}.

More precisely, let $\lambda = (\lambda_1, \ldots, \lambda_r) \in \mathbb{Z}^r$, where $\lambda_1 \geq \cdots \geq \lambda_r \geq 0$, and let $\rho = (r-1, r-2, \ldots, 1, 0)$.
In \cref{sec:crystal} we give bijections between $\GTP_\lambda$, the set of Gelfand--Tsetlin patterns with top row $\lambda$, and both $\mathfrak{S}_{\lambda + \rho}^\Gamma$ and $\mathfrak{S}_{\lambda + \rho}^\Delta$, the sets of states of the $\Gamma$ and $\Delta$ crystal models with top boundary paths at columns $\lambda + \rho$ with some fixed colors.
This allows us to transfer the Schützenberger involution to a bijection between $\mathfrak{S}_{\lambda + \rho}^\Gamma$ and $\mathfrak{S}_{\lambda + \rho}^\Delta$.
We also show that all of these bijections interact nicely with the weights of the various objects (for the precise definitions and statements, see \cref{sec:crystal}).
This can be summarized in the following commutative diagram, where $\SSYT(\lambda)$ is the set of semistandard Young tableaux with shape $\lambda$.
\begin{equation}
  \label{eq:GammaDelta-commutative-diagram}
  \begin{tikzcd}
    \mathfrak{S}_{\lambda + \rho}^\Gamma  \arrow[r, leftrightarrow, dashed] \arrow[d, leftrightarrow, dashed, "\text{Sch}"] 
    & \GTP_\lambda                        \arrow[r, leftrightarrow] \arrow[d, leftrightarrow, "\text{Sch}"] 
    & \SSYT(\lambda)                      \arrow[d, leftrightarrow, "\text{Sch}"] 
    \\
    \mathfrak{S}_{\lambda + \rho}^\Delta \arrow[r, leftrightarrow, dashed] 
    & \GTP_\lambda \arrow[r, leftrightarrow]
    & \SSYT(\lambda)   
  \end{tikzcd}
\end{equation}

This gives a state-by-state refinement of the left-right duality which was originally a statement about partition functions, i.e.~sums over states.
The general left-right duality for partition functions is proven in this paper by applying the Yang--Baxter equation to swap pairs of adjacent rows in systems consisting of a mix of left-moving and right-moving rows. 
This only proves the duality for partition functions, but we show in \cref{sec:crystal} that for the crystal models each step in the proof refines to a state-by-state counterpart.
Indeed, we show that the row swaps can be refined to the level of states by Berenstein--Kirillov involutions $t_i$ on Gelfand--Tsetlin patterns introduced in~\cite{KirillovBerenstein} or, equivalently, Bender--Knuth involutions $\operatorname{BK}_i$ on SSYT introduced in~\cite{BenderKnuth}. 
Furthermore, the way the left-right duality reduces to individual row swaps exactly mirrors the factorization of the Schützenberger involution into Berenstein--Kirillov involutions.

The above commutative diagram can therefore be decomposed into a tower of commutative diagrams of the following kind:
\begin{equation}
  \begin{tikzcd}
    \mathfrak{S}_{\lambda + \rho}^\Theta  \arrow[r, leftrightarrow, dashed] \arrow[d, leftrightarrow, dashed, "t_{r-i}"] 
    & \GTP_\lambda                        \arrow[r, leftrightarrow] \arrow[d, leftrightarrow, "t_{r-i}"] 
    & \SSYT(\lambda)                      \arrow[d, leftrightarrow, "\operatorname{BK}_{r-i}"] 
    \\
    \mathfrak{S}_{\lambda + \rho}^{s_{i}\Theta} \arrow[r, leftrightarrow, dashed] 
    & \GTP_\lambda \arrow[r, leftrightarrow]
    & \SSYT(\lambda)   
  \end{tikzcd}
\end{equation}
Here $\mathfrak{S}^\Theta_{\lambda + \rho}$ denotes the set of states of the lattice model with a mix of left- and right-moving rows given by $\Theta = (\Theta_1, \Theta_2, \dots, \Theta_r)$, such that $\Theta_i \neq \Theta_{i+1}$, and top boundary conditions determined by $\lambda + \rho$. 

As an intermediate step in the proof we define subsets of Gelfand--Tsetlin patterns where pairs of consecutive rows are either so-called \emph{left-strict} or \emph{right-strict} corresponding to row types of the lattice model (although somewhat confusingly a left-moving row in the lattice model corresponds to a \emph{right-strict} row pair in the Gelfand--Tsetlin pattern).
In \cref{lemma:ti-GTPTheta_commutes} we show the combinatorial result that Berenstein--Kirillov involutions act naturally on these subsets by swapping two consecutive distinct row types.

In the last part of the paper we show that the boundary conditions are transformed as expected under these row swaps by Berenstein--Kirillov involutions, thus completing the proof that the involutions refine the corresponding equality for partition functions. 
As the first part of the proof we show that the colors of the horizontally exiting paths (on either the left or right for each row depending on $\Theta_i$) read from the top to bottom can be determined from the associated Gelfand--Tsetlin pattern using an action of the Coxeter monoid corresponding to the symmetric group~$S_r$. 
This first part generalizes a similar statement in~\cite{BBBG:demazure} to states with a mix of left- and right-moving rows.

\begin{remark}
After the first version of this manuscript appeared on arXiv, we became aware of the work of \cite{XXZ:hybrid} which presents a collection of lattice models computing Demazure characters similar to the hybrid pipe dream models by \cite{KnutsonUdell}.
After a series of transformations these appear to be equivalent to our (mixed) crystal models although with different boundary conditions giving rise to Demazure characters instead of Demazure atoms.
Ongoing work carried out together with Ben Brubaker and Daniel Bump suggests that the crystal model with these alternative boundary conditions is related to a variant of the crystal model known as the \emph{closed} crystal model that was introduced in \cite{Yang:closed}.
Similar variants were also discussed in~\cite{Buciumas-Scrimshaw} for Cartan types B and C. 

Given these rewritings, \cite{XXZ:hybrid} also proves that there is a state-by-state left-right duality for the crystal models (i.e.\ in the limit $v\to 0$).
Similar to the proof presented in \cref{sec:crystal} this paper, the bijection factorizes into two kinds of simpler bijections: one swapping two adjacent rows and the other changing the row-type of the last row. 
The main difference in the two approaches is how the former operation is described; in~\cite[Section~3]{XXZ:hybrid} it is described diagrammatically, case by case for different subpatterns, on the other hand we describe the states by Gelfand--Tsetlin patterns and the row-swapping bijection as a Berenstein--Kirillov involution on the corresponding Gelfand--Tsetlin patterns.
As a corollary in our paper we get that the full left-right bijection becomes the Schützenberger involution, as well as a new interpretation of the key tableau algorithms from \cite{BBBG:demazure} as mentioned above.
\end{remark}

\subsection{Outline}
We explain how the families of lattice models are constructed in \cref{sec:families_of_lattice_models} and how specializations of them relate to models previously defined in the literature. 
In \cref{sec:YBE} we introduce the different types of Yang--Baxter equations that are used in the paper and prove that the families of lattice models satisfy these equations in \cref{thm:YBE}.
In the same section we also prove the duality of the left- and right-moving families, that is, show that their partition functions agree in \cref{thm:left-right_duality}.
Because of color loops, the fourth $R$-matrix mentioned above needs to be treated differently from the other cases.
Although we do not need this particular $R$-matrix for the rest of our results, we show how the Yang--Baxter equations involving it can be reduced to the other cases without it in \cref{sec:YBE,appendix:proof_of_R-matrix_inverse}.
Lastly, we consider the crystal limit of the Iwahori models in \cref{sec:crystal} where we first show that there is a bijection between mixed states and Gelfand--Tsetlin patterns, define the Berenstein--Kirillov involutions and describe how they give rise to the Schützenberger involution.
We show how the Berenstein--Kirillov involutions give a bijection of states of different mixed row types and, in \cref{thm:ti-state-weights}, how the weight of the state is transformed.
In \cref{thm:ti-boundary-conditions} we then prove that the boundary colors of the state are transformed as expected using a combinatorial argument based on patterns of elements in a Coxeter monoid.

\subsection{Acknowledgments}
We would like to thank the participants of the \emph{Conference on 
Solvable Lattice Models, Number Theory and Combinatorics} in Dublin 2024 for their comments on an early presentation of the results in this paper.
We are particularly thankful to Ben Brubaker and Daniel Bump as well as Valentin Buciumas and Travis Scrimshaw for helpful discussions about their work \cite{BBF:orange} and \cite{Buciumas-Scrimshaw}.

\section{The families of lattice models}
\label{sec:families_of_lattice_models}

\subsection{Setup}
\label{sec:introduction_to_lattice_models}

There are two equivalent ways to define the Iwahori ice model related by a process called \emph{fusion}.
The model introduced below is called the \emph{fused} model, and in \cref{sec:fusion} we introduce the \emph{unfused} model, which is often easier to work with.
The relationship between the models is also described in \cref{sec:fusion}.

Consider a rectangular grid with $r$ rows and $N$ columns, which we will think of as a graph with the crossings as vertices.
There are edges at the boundary of the grid that are only adjacent to a single vertex, so this is a slight generalization of the usual definition of a graph.
The edges at the boundary of the grid will be called \emph{boundary edges} while the remaining edges will be known as \emph{internal edges}.

For each edge, there is a set of possible labels called \emph{spins}, and a labeling of all the boundary edges will be referred to as a set of \emph{boundary conditions}.
A \emph{vertex configuration} is a choice of spins for the four edges surrounding the vertex. 
The lattice models we consider assign a so-called \emph{Boltzmann weight} to each vertex configuration which is a rational function, or more precisely an element in the fraction field of a complex polynomial ring in some indeterminates associated to lattice model.

We say that a vertex configuration is \emph{admissible} if its Boltzmann weight is non-zero. 
In this paper we mainly consider lattice models with six types of admissible vertex configurations (in the \emph{unfused} description, see \cref{sec:fusion}), which is why they are sometimes called six-vertex models, and the Boltzmann weights for vertices at row $i$ will depend on a row parameter $z_i$.
The rows are numbered $1, \ldots, r$ from top to bottom.

A \emph{state} $\mathfrak{s}$ of the model is a labeling of all the edges and we say that the state is \emph{admissible} if every vertex configuration of the state is admissible.
Given some specific conditions on the states we define the corresponding \emph{system} $\mathfrak{S}$ of the lattice model to be the set of all states satisfying these conditions.
Unless otherwise noted, we will consider systems consisting of all admissible states with some given boundary conditions.

The \emph{partition function} $Z(\mathfrak{S})$ of a system $\mathfrak{S}$ is a weighted sum over the states 
\begin{equation*}
  Z(\mathfrak{S}) \coloneq \sum_{\mathfrak{s}\in\mathfrak{S}} \wt(\mathfrak{s}),
\end{equation*}
where the \emph{Boltzmann weight} $\wt(\mathfrak{s})$ of the state $\mathfrak{s}$ is the product of the Boltzmann weights of all its vertices.
Note that non-admissible states have vanishing Boltzmann weights, and so even if a system is enlarged to include non-admissible states, the partition function would not be affected.

In this paper, the partition functions will be polynomials in $z_1^{\pm1}, \dots, z_r^{\pm1}$ with coefficients depending on, for example, a parameter $q$.
We will consider models whose states can be described by a set of $r$ colored paths going from the top boundary of the grid to either its left or right boundary.

We fix an ordered palette of $m$ colors $\mathcal{P} = \mathcal{P}_m = \{c_1, \dots, c_m\}$ where $c_1 < \dots < c_m$, and will assign a color to each path according to the following rules.
The set of possible spins (labelings) of an edge, called a \emph{spin set} is different for horizontal and vertical edges: the possible spins for the horizontal edges are either a color $c_1, \dots, c_m$ or no color at all, while for the vertical edges the possible spins are instead subsets of $\mathcal{P}$.
In other words, there can be at most one path on each horizontal edge, while on the vertical edges the paths can overlap assuming that they have distinct colors.
That the vertical edges can only carry at most one path of each color is the reason why the models are called \emph{fermionic} in contrast to \emph{bosonic} models (such as \cite{BumpNaprienko}) where vertical edges can carry any number of paths for each color.
It is not necessary that the number of colors $m$ equals the number of rows $r$, even though this is often considered for applications in $p$-adic representation theory such as in \cite{BBBG:Iwahori,BBBG:demazure}.

The models we will consider are divided into two types: \emph{left-moving} and \emph{right-moving} models.
Right-moving models are characterized by the fact that the colored paths start at the top boundary and move down and to the \emph{right} through the grid until they get to the \emph{right} boundary.
Left-moving models, on the other hand, are characterized by the fact that the colored paths start at the top boundary and move down and to the \emph{left} through the grid until they get to the \emph{left} boundary.
See \cref{fig:example_fused_right-moving_state,fig:example_fused_left-moving_state} for examples of states of these models.

The admissible vertex configurations and their weights for the right- and left-moving models can be seen in \cref{tab:T-weights_fused_right_family} and \cref{tab:T-weights_fused_left_family} respectively.
In the next subsection these will be decomposed into simpler pieces.

\begin{table}
  \caption{The fused vertex configurations $\Tr$ for the family of right-moving models together with the corresponding Boltzmann weights, where $c_i$ and $c_j$ are any distinct colors.
  Here $\Sigma$ denotes a subset of the palette, and $\Sigma^{+}_c$ is defined to be $\Sigma \cup \{c\}$ for any color $c$ and when this is used it is assumed that $c \notin \Sigma$.}
  \label{tab:T-weights_fused_right_family}

  \newcommand{\centerVertically}[1]{\ensuremath{\vcenter{\hbox{#1}}}}
  \begin{tblr}{hlines, vlines, columns={c}, cells={valign=b}}
    \SetRow{rowsep=4pt}
    \SetCell[c=3]{}
    $\boldsymbol{\Tr}$

    \\ \hline

    \begin{tikzpicture}[baseline]
      \pic at (0,0) {fusedvertex={$+$}{$\Sigma$}{$+$}{$\Sigma$}};
    \end{tikzpicture}
    &
    \begin{tikzpicture}[baseline]
      \pic at (0,0) {fusedvertex={$c_i$}{$\Sigma$}{$c_i$}{$\Sigma$}};
    \end{tikzpicture}
    &
    \begin{tikzpicture}[baseline]
      \pic at (0,0) {fusedvertex={$+$}{$\Sigma^{+}_{c_i}$}{$c_i$}{$\Sigma$}};
    \end{tikzpicture}

    \\

    $\Phi^{|\Sigma|}$
    &
    \centerVertically{$\displaystyle \Phi^{|\Sigma|} z \prod_{c_k \in \Sigma} \mathfrak{X}_{i,k}$}
    &
    \centerVertically{$\displaystyle \Phi^{|\Sigma|}\prod_{\substack{c_k \in \Sigma\\k>i}} \mathfrak{X}_{i,k}$}

    \\

    \begin{tikzpicture}[baseline]
      \pic at (0,0) {fusedvertex={$c_i$}{$\Sigma$}{$+$}{$\Sigma^{+}_{c_i}$}};
    \end{tikzpicture}
    &
    \SetCell[c=2]{}
    \begin{tikzpicture}[baseline]
      \pic at (0,0) {fusedvertex={$c_i$}{$\Sigma^{+}_{c_j}$}{$c_j$}{$\Sigma^{+}_{c_i}$}};
    \end{tikzpicture}
    &

    \\

    $\displaystyle \Phi^{|\Sigma|+1} (1-q^2)z \prod_{\substack{c_k \in \Sigma\\k<i}} \mathfrak{X}_{i,k}$
    &
    \SetCell[c=2]{}
    $\begin{cases}
      \displaystyle \Phi^{|\Sigma|+1} (1-q^2)z \prod_{\substack{c_k \in \Sigma\\k<i}} \mathfrak{X}_{i,k} \prod_{\substack{c_k \in \Sigma\\k>j}} \mathfrak{X}_{j,k} & \text{if } i < j\\
      0 & \text{if } i > j
    \end{cases}$
    &
  \end{tblr}
\end{table}

\begin{table}
  \caption{The fused vertex configurations $\Tl$ for the family of left-moving models together with the corresponding Boltzmann weights, where $c_i$ and $c_j$ are any distinct colors.
  Here $\Sigma$ denotes a subset of the palette, and $\Sigma^{+}_c$ is defined to be $\Sigma \cup \{c\}$ for any color $c$ and when this is used it is assumed that $c \notin \Sigma$.}
  \label{tab:T-weights_fused_left_family}

  \newcommand{\centerVertically}[1]{\ensuremath{\vcenter{\hbox{#1}}}}
  \begin{tblr}{hlines, vlines, columns={c}, cells={valign=b}}
    \SetRow{rowsep=4pt}
    \SetCell[c=3]{}
    $z\boldsymbol{\Tl}$

    \\ \hline

    \begin{tikzpicture}[baseline]
      \pic at (0,0) {fusedvertex={$+$}{$\Sigma$}{$+$}{$\Sigma$}};
    \end{tikzpicture}
    &
    \begin{tikzpicture}[baseline]
      \pic at (0,0) {fusedvertex={$c_i$}{$\Sigma$}{$c_i$}{$\Sigma$}};
    \end{tikzpicture}
    &
    \begin{tikzpicture}[baseline]
      \pic at (0,0) {fusedvertex={$c_i$}{$\Sigma^{+}_{c_i}$}{$+$}{$\Sigma$}};
    \end{tikzpicture}

    \\

    $\Phi^{|\Sigma|} z$
    &
    \centerVertically{$\displaystyle \Phi^{|\Sigma|} \prod_{c_k \in \Sigma} \mathfrak{X}_{k,i}$}
    &
    \centerVertically{$\displaystyle \Phi^{|\Sigma|} \prod_{\substack{c_k \in \Sigma\\k<i}} \mathfrak{X}_{k,i}$}

    \\

    \begin{tikzpicture}[baseline]
      \pic at (0,0) {fusedvertex={$+$}{$\Sigma$}{$c_i$}{$\Sigma^{+}_{c_i}$}};
    \end{tikzpicture}
    &
    \SetCell[c=2]{}
    \begin{tikzpicture}[baseline]
      \pic at (0,0) {fusedvertex={$c_j$}{$\Sigma^{+}_{c_j}$}{$c_i$}{$\Sigma^{+}_{c_i}$}};
    \end{tikzpicture}
    &

    \\

    $\displaystyle \Phi^{|\Sigma|+1} (1-q^2) z \prod_{\substack{c_k \in \Sigma\\k>i}} \mathfrak{X}_{k,i}$
    &
    \SetCell[c=2]{}
    $\begin{cases}
      0 & \text{if } i < j \\
      \displaystyle \Phi^{|\Sigma|+1} (1-q^2) \prod_{\substack{c_k \in \Sigma\\k>i}} \mathfrak{X}_{k,i} \prod_{\substack{c_k \in \Sigma\\k<j}} \mathfrak{X}_{k,j} & \text{if } i > j
    \end{cases}$
    &
  \end{tblr}
\end{table}

For all the models we consider in this paper the admissible vertex configurations have a property we call \emph{color conservation}, which means that if we consider two of the edges of a vertex as \emph{inputs} and the other two as \emph{outputs} then the multiset of input colors is equal to the multiset of output colors. 
For the left- and right-moving models we assign the input and output edges according to how the paths are allowed to move, i.e.~the bottom and right edges are outputs for right-moving models and the bottom and left edges are outputs for the left-moving models, with the remaining edges being inputs. 

Because of this color conservation, for an admissible state of the left-moving model the colors of the paths on the left boundary must be a permutation of the colors of the paths on the top boundary, and similarly for the right-moving model.
The boundary conditions for a left- or right-moving system can therefore be described by the filled-in columns on the top boundary along with their colors, as well as the permutation of these colors that appear on the left or right boundary, respectively.

\subsection{Fusion and boundary conditions}
\label{sec:fusion}
There is another description of these models where the admissible vertex configurations and their weights are simpler, and which is often easier to work with.
Instead of using a subset of the palette $\mathcal{P}_m$ to describe the spin of a vertical edge the alternative description uses a list of $m$ binary elements specifying whether the corresponding colors are included or not.
To do this we split each column into $m$ new columns, one for each color of the palette, and restrict each new column to only carry paths of this color (or no path at all).
The $m$ new columns corresponding to a single column in the original model will be called a \emph{block}.
This is called the \emph{unfused} model and we will first describe it for the right-moving model, and then see how it relates to the fused version described above.
The process is illustrated in \cref{fig:example_unfused_state_right-moving}.

To reiterate, the unfused model has a grid with $r$ rows and $Nm$ columns ($m$ times more than before) and we split the columns into $N$ blocks of $m$ columns.
In each block the columns are assigned a column color, from $c_1$ on the left to $c_m$ on the right.
We will also number the columns from \emph{right to left}, starting at $0$.
Each edge can carry at most one path, and the vertical edges are furthermore limited to paths of the corresponding column color.
The admissible vertex configurations are described in \cref{tab:T_right_weights} together with their Boltzmann weights depending on the row parameter $z$ and parameters $q$ and $\Phi$.
It also depends on parameters $\mathfrak{X}_{i,j}$ for $1 \leq i,j \leq r$ which satisfy
\begin{equation}
  \label{eq:X-condition}
  \begin{cases}
    \mathfrak{X}_{i,i} = -q^2 \\
    \mathfrak{X}_{i,j} \mathfrak{X}_{j,i} = q^2 \text{ for } i \neq j.
  \end{cases}
\end{equation}
If $\mathcal{F} = \mathbb{C}(q,\Phi,\mathfrak{X}_{1,1},\mathfrak{X}_{1,2},\dots)$, where the $\mathfrak{X}_{i,j}$ are those with $i<j$, then the Boltzmann weights at row $k$ can be viewed as elements of the field $\mathcal{F}(z_k)$ of rational functions in $z_k$.

We will now describe how the boundary conditions for the unfused model is given by a list of colors $\sigma = (\sigma_1, \dots, \sigma_r) \in \mathcal{P}^r$ and a list of column numbers $\mu = (\mu_1, \dots, \mu_r) \in (\mathbb{Z}_{\geq 0})^r$ with $\mu_1 > \mu_2 > \dots > \mu_r$.
Here we assume that $Nm \geq \mu_1$, which we can make sure is true by choosing $N$ to be sufficiently large.
Firstly, there are no paths on the left and bottom boundaries.
Secondly, for the top boundary there is a path on the edges with column numbers $\mu_1, \dots, \mu_r$ of the corresponding column color (remember that vertical edges are only allowed to carry a path of a specific color), which explains why the assumption $Nm \geq \mu_1$ is required.
Finally, for the right boundary there is a path on every edge, and the color of the path on the $i^\text{th}$ row is $\sigma_i$.
Because of color conservation for admissible states, $\sigma$ needs to be a permutation of the top boundary colors which are determined by the entries of $\mu$ modulo $m$.

\begin{figure}[htbp]
  \newcommand{\scl}{0.8}
  \begin{tikzpicture}[baseline={(equalline)}, scale=\scl]
    \coordinate (equalline) at ($ (0,1) - (0,0.5ex) $);
    \begin{scope}[shift={(-0.75,0)}]
    \foreach \x/\Color in {0/red, 1/green, 2/blue, 3/red, 4/green, 5/blue} {
      \pgfmathtruncatemacro{\i}{mod(\x, 3)+1}
      \pgfmathtruncatemacro{\colnum}{5-\x}
      \draw (\x, -0.5) -- (\x, 2.5) node [above, \Color] {$c_\i$} node [above=1.25em] {\colnum};
    }
    \foreach \y in {0,...,2} {
      \pgfmathtruncatemacro{\i}{3-\y}
      \draw (-0.5, \y) node[left] {$z_{\i}$} -- (5.5, \y) node[right] {$\sigma_{\i}$};
    }

    \node[above=1.25em, anchor=south east] at (0,2.5) {column$\quad$};

    \draw[ultra thick, decorate, decoration = {calligraphic brace}] (0,4) --  (2,4) node[midway,above=0.5em] {block 1};
    \draw[ultra thick, decorate, decoration = {calligraphic brace}] (3,4) --  (5,4) node[midway,above=0.5em] {block 0};

    \draw[blue, ultra thick] (5,2.5) -- (5,2) -- (5.5,2);
    \draw[green, ultra thick] (1,2.5) -- (1,1) -- (4,1) -- (5.5,1);
    \draw[red, ultra thick] (0,2.5) -- (0,2) -- (3,2) -- (3,0) -- (3,0) -- (5.5,0);
    \node at (2.5,-1) {unfused model};

    \end{scope}
    
    \draw[Stealth-Stealth] (6,1) -- (8,1) node[midway,above] {fusion};
      
    \begin{scope}[shift={(9,0)}]
    \draw (0, -0.5) -- (0, 2.5) node [above=1.25em] {1};
    \draw (1, -0.5) -- (1, 2.5) node [above=1.25em] {0};

    \foreach \y in {0,...,2} {
      \draw (-0.5, \y) -- (1.5, \y);
    }

    \node[above=1.25em, anchor=south east] at (0,2.5) {block$\quad$};
    \node at (0.5,-1) {fused model};
    
    \newcommand{\lineWidth}{1}
    \pgfmathsetmacro{\pathShift}{\lineWidth/(10*\scl)}
    \draw[blue, line width=\lineWidth mm] (1,2.5) -- (1,2) -- (1.5,2);
    \draw[green, line width=\lineWidth mm] (0,2.5) -- (0,1) -- (1.5,1);
    \draw[red, line width=\lineWidth mm] (0-\pathShift,2.5) -- (0-\pathShift,2) -- (1,2) -- (1,0) -- (1,0) -- (1.5,0);
    \end{scope}
  \end{tikzpicture}
  \caption{A state for the unfused right-moving model with boundary data given by $\mu = (5,4,0)$ and $\sigma = (c_3, c_2, c_1)$ and its corresponding fused description.}
  \label{fig:example_unfused_state_right-moving}
\end{figure}

See \cref{fig:example_unfused_state_right-moving} for an example of a state for this model.
Note that the partition function does not depend on $N$ as long as $Nm \geq \mu_1$, because every vertex to the left of column $\mu_1$ will be an $\texttt{a}_1$ vertex, which has Boltzmann weight $1$.
We will denote by $\Sr_{\mu,\sigma}$ the system with these boundary conditions and usually abbreviate the partition function $Z(\Sr_{\mu,\sigma})(\mathbf{z})$ to $\Zr_{\mu,\sigma}(\mathbf{z})$, where $\mathbf{z} = (z_1, z_2, \dots, z_r)$.

\begin{table}[tbp]
  \caption{The unfused vertex configurations $\Tr$ for the family of right-moving models together with the corresponding Boltzmann weights.
  Here, $c_j$ is the column color and $c_i$ is the color of a path.}
  \label{tab:T_right_weights}  
  \begin{tabular}{|c|c|c|c|c|c|}
    \hline
    \multicolumn{6}{|c|}{$\boldsymbol{\Tr} \vphantom{\displaystyle\int}$} \\
    \hline
    \hline
    $\texttt{a}_1$  &  $\texttt{a}_2$  &  $\texttt{b}_1$  &  $\texttt{b}_2$  &  $\texttt{c}_1$  &  $\texttt{c}_2$
    \\ \hline
    \begin{tikzpicture}
      \draw (0,-1) -- (0,1) node[label=above:$c_j$] {};
      \draw (-1,0) -- (1,0);
    \end{tikzpicture}
    &
    \begin{tikzpicture}
      \draw [blue, ultra thick] (0,-1) -- (0,1) node[label=above:$c_j$] {};
      \draw [red, ultra thick] (-1,0) node[label={[label distance=-0.5em] left:$c_i$}] {} -- (1,0);
    \end{tikzpicture}
    &
    \begin{tikzpicture}
      \draw (-1,0) -- (1,0);
      \draw [red, ultra thick] (0,-1) -- (0,1) node[label=above:$c_j$] {};
    \end{tikzpicture}
    &
    \begin{tikzpicture}
      \draw (0,-1) -- (0,1) node[label=above:$c_j$] {};
      \draw [red, ultra thick] (-1,0) node[label={[label distance=-0.5em] left:$c_i$}] {} -- (1,0);
    \end{tikzpicture}
    &
    \begin{tikzpicture}
      \draw (1,0) -- (0,0) -- (0,1) node[label={[red]above:$c_j$}] {};
      \draw [red, ultra thick] (-1,0) -- (0,0) -- (0,-1);
    \end{tikzpicture}
    &
    \begin{tikzpicture}
      \draw (-1,0) -- (0,0) -- (0,-1);
      \draw [red, ultra thick] (1,0) -- (0,0) -- (0,1) node[label=above:$c_j$] {};
    \end{tikzpicture}
    \\ \hline
    $\displaystyle 1$

    & $\displaystyle \Phi \mathfrak{X}_{i,j}
    \begin{cases}
      z & i = j \\
      1 & i \neq j
    \end{cases}$

    & $\Phi$

    & $\displaystyle \begin{cases}
      z & i = j \\
      1 & i \neq j
    \end{cases}$

    & $\Phi(1-q^2)z$

    & $\displaystyle 1$
    \\ \hline
  \end{tabular}
\end{table}
\begin{table}[tbp]
  \vspace{1em}
  \caption{The unfused vertex configurations $\Tl$ for the family of left-moving models together with the corresponding Boltzmann weights.
  Here, $c_j$ is the column color and $c_i$ is the color of a path.}
  \label{tab:T_left_weights}
  \begin{tabular}{|c|c|c|c|c|c|}
    \hline
    \multicolumn{6}{|c|}{$\boldsymbol{\Tl}\vphantom{\displaystyle\int}$} \\
    \hline
    \hline
    $\texttt{a}_1$  &  $\texttt{a}_2$  &  $\texttt{b}_1$  &  $\texttt{b}_2$  &  $\texttt{c}_1$  &  $\texttt{c}_2$
    \\ \hline
    \begin{tikzpicture}
      \draw (0,-1) -- (0,1) node[label=above:$c_j$] {};
      \draw (-1,0) -- (1,0);
    \end{tikzpicture}
    &
    \begin{tikzpicture}
      \draw [blue, ultra thick] (0,-1) -- (0,1) node[label=above:$c_j$] {};
      \draw [red, ultra thick] (-1,0) -- (1,0) node[label={[label distance=-0.5em] right:$c_i$}] {};
    \end{tikzpicture}
    &
    \begin{tikzpicture}
      \draw (-1,0) -- (1,0);
      \draw [red, ultra thick] (0,-1) -- (0,1) node[label=above:$c_j$] {};
    \end{tikzpicture}
    &
    \begin{tikzpicture}
      \draw (0,-1) -- (0,1) node[label=above:$c_j$] {};
      \draw [red, ultra thick] (-1,0) -- (1,0) node[label={[label distance=-0.5em] right:$c_i$}] {};
    \end{tikzpicture}
    &
    \begin{tikzpicture}
      \draw (-1,0) -- (0,0) -- (0,1) node[label={[red]above:$c_j$}] {};
      \draw [red, ultra thick] (1,0) -- (0,0) -- (0,-1);
    \end{tikzpicture}
    &
    \begin{tikzpicture}
      \draw (1,0) -- (0,0) -- (0,-1);
      \draw [red, ultra thick] (-1,0) -- (0,0) -- (0,1) node[label=above:$c_j$] {};
    \end{tikzpicture}
    \\ \hline
    $\displaystyle 1$

    & $\displaystyle \Phi \mathfrak{X}_{j,i}
    \begin{cases}
      z^{-1} & i = j \\
      1      & i \neq j
    \end{cases}$

    & $\Phi$

    & $\displaystyle \begin{cases}
      z^{-1} & i = j \\
      1      & i \neq j
    \end{cases}$

    & $\Phi(1-q^2)$

    & $\displaystyle z^{-1}$
    \\ \hline
  \end{tabular} 
\end{table}

Given an unfused model, we can find the admissible vertex configurations and the Boltzmann weights of the corresponding fused model by considering a one-row one-block grid, which corresponds to a single vertex in the fused model.
Any given vertex configuration for the fused model gives us boundary conditions of the unfused one-row one-block grid, i.e.~in the following equality we are given the fused edge spins $A$, $B$, $C$ and $D$, which determine the unfused edge spins as $A$, $B_i$, $C$ and $D_i$ by setting $B_i$ to carry a path if there is a path of color $c_i$ in $B$ and similarly for $D_i$ and $D$.

Given boundary conditions of the unfused one-row one-block grid, there is at most one way to color the internal edges because of color conservation.
If there is an admissible unfused state then we say that the fused vertex configuration we started with was admissible, and its weight is equal to that of the unfused state, i.e.~the product of the weights of the $m$ unfused vertices.
\begin{equation*}
  \wt\left(
    \begin{tikzpicture}[baseline=-1mm, scale=1]
      \draw (-0.9,0) node[spin] {$A$} -- (0.9,0) node[spin] {$C$};
      \draw (0,-0.9) node[spin] {$D$} -- (0,0.9) node[spin] {$B$};
    \end{tikzpicture}
  \right)
  =
  \wt \left(
    \begin{tikzpicture}[baseline=-1mm, scale=1]
      \draw (-0.9,0) node[spin] {$A$} -- (2.5,0);
      \draw (3,0) node {$\cdots$};
      \draw (3.5,0) -- (4.9,0) node[spin]{$C$};
      \foreach \x in {1,...,3} {
        \draw (\x-1,-0.9) node[spin] {$D_\x$} -- (\x-1,0.9) node[label=above:$c_\x$, spin] {$B_\x$};
      }
      \draw (4,-0.9) node[spin] {$D_m$} -- (4,0.9) node[label=above:$c_m$, spin] {$B_m$};
    \end{tikzpicture}
  \right)
\end{equation*}
If there is no admissible unfused state, then the fused vertex configuration is not admissible and the Boltzmann weight is set to zero.

The unfused description of the left-moving model is defined similarly and the fusion process is the same.
Specifically, the spin sets are the same as for the unfused Delta model and the $m$ column colors of the unfused columns in a block are still $c_1, \ldots, c_m$ from left to right. 
The admissible vertex configurations and their Boltzmann weights are listed in \cref{tab:T_left_weights}.
The boundary conditions for the top boundary are determined by $\mu$ in the same way as for the right-moving model, and the left boundary colors are given by $\sigma$ from top to bottom in the same way as the right boundary colors for the right-moving model.

It is not difficult to see that with these boundary conditions, $\texttt{b}_2$ is the only vertex configuration that can appear to the left of column $\mu_1$.
Since the Boltzmann weight of the $\texttt{b}_2$ vertex configuration in \cref{tab:T_left_weights} is not identically $1$, the partition function of the left-moving model actually depends on the number of blocks $N$, in contrast to the right-moving model.
We will therefore denote by $\SlN{N}_{\mu,\sigma}$ the system with the boundary conditions above, and by $\SlN{N}_\mu$ the system where we drop the condition on the colors on the left-hand side.
Furthermore, we will usually abbreviate the partition functions $Z(\SlN{N}_{\mu,\sigma})(\mathbf{z})$ and $Z(\SlN{N}_\mu)(\mathbf{z})$ to $\ZlN{N}_{\mu,\sigma}(\mathbf{z})$ and $\ZlN{N}_\mu(\mathbf{z})$, respectively.
In both of these cases, if $N$ is increased by $1$, then the partition function is multiplied by $\mathbf{z}^{(1,\ldots,1)} = z_1 z_2 \cdots z_r$.

The dependence on $N$ can be avoided by multiplying each unfused weight in~\cref{tab:T_left_weights} by $z^{1/m}$ since there are $N m$ columns, but we choose not to do this to avoid unnecessary roots.
A consequence of this is that we will need to include an extra factor $\mathbf{z}^N = (z_1 \cdots z_r)^N$ in the left-right duality as detailed in \cref{thm:left-right_duality} which was left out of \cref{mainthm:duality} for simplicity.
However, this is not an issue for the fused model where the weights would simply be multiplied by $z$, and these are the weights that can be found in \cref{tab:T-weights_fused_left_family}, where they are therefore labeled $z\Tl$.

\subsection{Specializations}
\label{sec:specializations}
We are interested in two specializations of these models which are related to two different types of Whittaker functions, namely metaplectic (spherical) Whittaker functions and (non-metaplectic) Iwahori Whittaker functions.
The two specializations are:
\begin{align}
  &\textbf{Metaplectic specialization:} & \Phi &= 1 & \mathfrak{X}_{i,j} &= g(j-i) & q^2 &= v \label{eq:metaplectic-specialization} \\
  &\textbf{Iwahori specialization: } & \Phi &= -v & \mathfrak{X}_{i,j} &= \left\{\begin{smallmatrix*}[l] -1 & & i < j \\ -1/v & & i \geq j \end{smallmatrix*}\right. & q^2 &= 1/v. \label{eq:Iwahori-specialization}
\end{align}
Here $g(a)$ is a Gauss sum which is a particular sum over roots of unity ubiquitous in number theory and which satisfies~\eqref{eq:X-condition}.
The reparametrization of $q$ in terms of $v$ is both due to historical reasons and because its connection to the underlying non-archimedean field, namely $v$ is the reciprocal of the residue field cardinality.
This means that the connection between the residue field cardinality and the $q$ appearing in the lattice model is different for the two different specializations.

We will call the corresponding lattice models left- and right-moving metaplectic ice and Iwahori ice, respectively, and their weights are shown in~\cref{tab:T_right_specializations,tab:T_left_specializations}.
There we have also included another specialization called the crystal limit which will be discussed in \cref{sec:crystal-models}.

\begin{table}[p]
  \caption{Specializations of the unfused right-moving model $\Tr$.}
  \label{tab:T_right_specializations}
  \newcommand{\scl}{0.78}
  \begin{tblr}{hlines, vlines, columns={c}, row{4,5,6} = {ht=1.5cm,valign=m}}
    \SetRow{ht=1cm}\SetCell[c=7]{c} $\boldsymbol{\Tr}$ \textbf{(specializations, unfused)} \\ \hline
    & $\texttt{a}_1$  &  $\texttt{a}_2$  &  $\texttt{b}_1$  &  $\texttt{b}_2$  &  $\texttt{c}_1$  &  $\texttt{c}_2$
    \\
    &
    \begin{tikzpicture}[scale=\scl]
      \draw (0,-1) -- (0,1) node[label=above:$c_j$] {};
      \draw (-1,0) -- (1,0);
    \end{tikzpicture}
    &
    \begin{tikzpicture}[scale=\scl]
      \draw [blue, ultra thick] (0,-1) -- (0,1) node[label=above:$c_j$] {};
      \draw [red, ultra thick] (-1,0) node[label={[label distance=-0.5em] left:$c_i$}] {} -- (1,0);
    \end{tikzpicture}
    &
    \begin{tikzpicture}[scale=\scl]
      \draw (-1,0) -- (1,0);
      \draw [red, ultra thick] (0,-1) -- (0,1) node[label=above:$c_j$] {};
    \end{tikzpicture}
    &
    \begin{tikzpicture}[scale=\scl]
      \draw (0,-1) -- (0,1) node[label=above:$c_j$] {};
      \draw [red, ultra thick] (-1,0) node[label={[label distance=-0.5em] left:$c_i$}] {} -- (1,0);
    \end{tikzpicture}
    &
    \begin{tikzpicture}[scale=\scl]
      \draw (1,0) -- (0,0) -- (0,1) node[label={[red]above:$c_j$}] {};
      \draw [red, ultra thick] (-1,0) -- (0,0) -- (0,-1);
    \end{tikzpicture}
    &
    \begin{tikzpicture}[scale=\scl]
      \draw (-1,0) -- (0,0) -- (0,-1);
      \draw [red, ultra thick] (1,0) -- (0,0) -- (0,1) node[label=above:$c_j$] {};
    \end{tikzpicture}
    \\
    
    {\footnotesize Metaplectic \\ $\Delta$}
    
    & $1$
    & $g(j-i)
    \displaystyle\left\{\begin{smallmatrix*}[l]
        z &&& i = j \\
        1 &&& i \neq j
    \end{smallmatrix*}\right.$
    & $1$
    &  $\displaystyle \left\{\begin{smallmatrix*}[l]
        z &&& i = j \\
        1 &&& i \neq j
    \end{smallmatrix*}\right.$
    & $(1-v)z$
    & 1
    
    \\
    
    {\footnotesize Iwahori \\ $\Gamma$}

    & $\displaystyle 1$

    & $\displaystyle\left\{\begin{smallmatrix*}[l]
        v &&& i < j \\
        z &&& i = j \\
        1 &&& i > j
    \end{smallmatrix*}\right.$

    & $\displaystyle -v$

    & $\displaystyle\left\{\begin{smallmatrix*}[l]
        z &&& i = j \\
        1 &&& i \neq j
    \end{smallmatrix*}\right.$

    & $\displaystyle (1-v)z$

    & $\displaystyle 1$
    \\ 

    {\centering\footnotesize Crystal \\ $\Gamma$}

    & $1$

    & $\displaystyle\left\{\begin{smallmatrix*}[l]
        0 &&& i < j \\
        z &&& i = j \\
        1 &&& i > j
    \end{smallmatrix*}\right.$

    & $\displaystyle 0$

    & $\displaystyle\left\{\begin{smallmatrix*}[l]
        z &&& i = j \\
        1 &&& i \neq j
    \end{smallmatrix*}\right.$

    & $\displaystyle z$

    & $\displaystyle 1$
    \\
  \end{tblr}
\end{table}
\begin{table}[p]
  \caption{Specializations of the unfused left-moving model $\Tl$.}
  \label{tab:T_left_specializations}
  \newcommand{\scl}{0.78}
  \begin{tblr}{hlines, vlines, columns={c}, row{4,5,6} = {ht=1.5cm,valign=m}}
    \SetRow{ht=1cm}\SetCell[c=7]{c}$\boldsymbol{\Tl}$ \textbf{(specializations, unfused)} \\ \hline
    &  $\texttt{a}_1$  &  $\texttt{a}_2$  &  $\texttt{b}_1$  &  $\texttt{b}_2$  &  $\texttt{c}_1$  &  $\texttt{c}_2$
    \\ 
    &
    \begin{tikzpicture}[scale=\scl]
      \draw (0,-1) -- (0,1) node[label=above:$c_j$] {};
      \draw (-1,0) -- (1,0);
    \end{tikzpicture}
    &
    \begin{tikzpicture}[scale=\scl]
      \draw [blue, ultra thick] (0,-1) -- (0,1) node[label=above:$c_j$] {};
      \draw [red, ultra thick] (-1,0) -- (1,0) node[label={[label distance=-0.5em] right:$c_i$}] {};
    \end{tikzpicture}
    &
    \begin{tikzpicture}[scale=\scl]
      \draw (-1,0) -- (1,0);
      \draw [red, ultra thick] (0,-1) -- (0,1) node[label=above:$c_j$] {};
    \end{tikzpicture}
    &
    \begin{tikzpicture}[scale=\scl]
      \draw (0,-1) -- (0,1) node[label=above:$c_j$] {};
      \draw [red, ultra thick] (-1,0) -- (1,0) node[label={[label distance=-0.5em] right:$c_i$}] {};
    \end{tikzpicture}
    &
    \begin{tikzpicture}[scale=\scl]
      \draw (-1,0) -- (0,0) -- (0,1) node[label={[red]above:$c_j$}] {};
      \draw [red, ultra thick] (1,0) -- (0,0) -- (0,-1);
    \end{tikzpicture}
    &
    \begin{tikzpicture}[scale=\scl]
      \draw (1,0) -- (0,0) -- (0,-1);
      \draw [red, ultra thick] (-1,0) -- (0,0) -- (0,1) node[label=above:$c_j$] {};
    \end{tikzpicture}
    \\
    
    {\footnotesize Metaplectic \\ $\Gamma$}

    & $1$
    & $g(i-j)
    \displaystyle\left\{\begin{smallmatrix*}[l]
        z^{-1} && i = j \\
        1      && i \neq j
    \end{smallmatrix*}\right.$
    & $1$
    & $\displaystyle\left\{\begin{smallmatrix*}[l]
        z^{-1} &&& i = j \\
        1      &&& i \neq j
    \end{smallmatrix*}\right.$
    & $1-v$
    & $z^{-1}$

    \\

    {\footnotesize Iwahori \\ $\Delta$}

    & $\displaystyle 1$

    & $\displaystyle\left\{\begin{smallmatrix*}[l]
        1   &&& i < j \\
        z^{-1} &&& i = j \\
        v   &&& i > j
    \end{smallmatrix*}\right.$

    & $\displaystyle -v$

    & $\displaystyle\left\{\begin{smallmatrix*}[l]
        z^{-1} &&& i = j \\
        1      &&& i \neq j
    \end{smallmatrix*}\right.$

    & $\displaystyle 1-v$

    & $\displaystyle z^{-1}$
    \\

    {\footnotesize Crystal \\ $\Delta$}
    
    & $\displaystyle 1$

    & $\displaystyle\left\{\begin{smallmatrix*}[l]
        1      &&& i < j \\
        z^{-1} &&& i = j \\
        0      &&& i > j
    \end{smallmatrix*}\right.$

    & $\displaystyle 0$

    & $\displaystyle\left\{\begin{smallmatrix*}[l]
      z^{-1} & i = j \\
      1      & i \neq j
    \end{smallmatrix*}\right.$

    & $\displaystyle 1$

    & $\displaystyle z^{-1}$
    \\
  \end{tblr}
\end{table}

Let us now describe how these specializations relate to existing lattice models in the literature.
See also \cref{fig:specializations}.
\begin{itemize}
  \item The right-moving Iwahori ice model is exactly the one defined in~\cite{BBBG:Iwahori} for computing values of Iwahori Whittaker functions of unramified principal series of $\GL_r(F)$ where $F$ is an non-archimedean local field with residue field cardinality $1/v$.
  \item The left-moving Iwahori ice model is new and the construction of this model was a motivation for this paper. The reason for this will be explained in the next paragraph.
  \item The right-moving metaplectic ice model is equivalent to the metaplectic $\Delta'$-model of \cite{BBBG:duality} by a reparametrization of row parameters and an exchange of colors with supercolors (as explained therein).
  The $\Delta'$-model is in turn related to the $\Delta$-model of the same paper by a change of basis for the spin sets.
  The partition function computes values of spherical Whittaker functions on metaplectic $n$-covers of $\GL_r(F)$.
  \item The left-moving metaplectic ice model is equivalent to the metaplectic $\Gamma$-model of~\cite{BBBG:duality} after reparametrizing the row parameters, multiplying the $T$-weights by $z$, applying the same change of basis and exchanging colors with supercolors as above.
\end{itemize}
As mentioned in the introduction, the $\Delta$ and $\Gamma$ metaplectic models actually give the same partition functions (up to reversing $\mathbf{z}$) and were originally defined in \cite{BBB, BBCFG}.
The fact that the $\Gamma$ and $\Delta$ variants give the same partition functions is non-trivial and was proven in an intricate combinatorial argument in \cite{BBF:orange} and later using Yang--Baxter equations in \cite{BBB,BBBGray}.
The identity, which is called the $\Gamma$-$\Delta$ duality, is not a state-by-state identity but requires a complicated interplay between the terms of the partition function coming from differently sized sets of states on the two sides.
The duality was used in~\cite{BBBG:duality} to prove an unexpected and surprising equality of non-metaplectic Iwahori Whittaker functions and metaplectic spherical Whittaker functions.

We should note that the change of basis described above to go from the $\Delta$-model to the $\Delta'$-model introduces an extra overall factor of $\mathbf{z}^\sigma$ (where a color is identified with its index) for the partition function, and similarly for the corresponding $\Gamma$-models.
However, this does not change the $\Gamma$-$\Delta$ equality since both $\mathbf{z}$ and $\sigma$ are reversed leaving the same factor.

One important goal of this paper was to construct a dual, left-moving model to the existing right-moving Iwahori model mirroring the above story for the metaplectic specialization.
In \cref{sec:left-right_duality} we will actually show that the whole left- and right-moving families described in this section are dual to each other, giving the same partition functions as detailed in~\cref{thm:left-right_duality}.

\begin{remark}
  \label{rem:naming-convention}
  Since the metaplectic specialization of the left-moving model is called a $\Gamma$ model and the right-moving model a $\Delta$ model, it may be tempting to classify the whole left- and right-moving family as $\Gamma$ and $\Delta$ models, respectively.
  Indeed, we used this terminology in early talks about this work.
  However, we thank Daniel Bump for pointing out to us that the right-moving Iwahori ice model reduces to a $\Gamma$ model when the palette size $m = 1$.
  This corresponds to a non-metaplectic spherical Whittaker function (and is thus a common point of origin for Iwahori ice and metaplectic ice).
  Thus, the right-moving Iwahori ice model should instead be classified as a $\Gamma$ model instead of a $\Delta$ model for all $m$.
  Its left-moving dual reduces to a $\Delta$ model for $m=1$ and will therefore be called a $\Delta$ model for all $m$.
  See \cref{fig:specializations}.
  As we will see in the next subsection this naming convention for the Iwahori specialization also agrees with the models in the $v\to 0$ limit studied in~\cite{Buciumas-Scrimshaw}.
  
  Since the $\Gamma$-$\Delta$ classification is therefore not an invariant for the continuous families of lattice models we will in this paper mostly use the terminology left- and right-moving and will use the terminology $\Gamma$ and $\Delta$ only for the Iwahori or the metaplectic specializations.
\end{remark}

In the next subsection we discus the $v\to 0$ limit of the Iwahori ice models where the left-right duality does reduce to a state-by-state equality.

\subsection{Crystal models}
\label{sec:crystal-models}
In the Iwahori specialization of both the left-moving and right-moving families we will now consider the limit $v \to 0$, known as the \emph{crystal limit} since it is connected to crystal bases of the quantum group defined by \cite{Kashiwara, Lusztig} (see also \cite{BumpSchilling}).
For convenience the weights can be found in \cref{tab:T_right_specializations,tab:T_left_specializations}.
Note that in the crystal limit the $\texttt{b}_1$ weight becomes zero and thus the $\texttt{b}_1$ vertex configuration is no longer admissible.
For the $\texttt{a}_2$ configuration the weight also vanishes in the case $i<j$, which means that paths of distinct colors are only allowed to cross in one direction, and then those two paths cannot cross again.

\begin{lemma}
  \label{lemma:crystal_models_at_most_one_color}
  Suppose we have an admissible vertex configuration of the crystal limit $v \to 0$ of the Iwahori specialization of the \emph{fused} left- or right-moving model (i.e.~$\Delta$ and $\Gamma$, respectively).
  If the bottom edge carries at most one colored path, then the same is true for the top edge.
\end{lemma}
\begin{proof}
  We will start with the right-moving case.
  Recall that by color conservation, the multiset of colors on the bottom and right edges is the same as the multiset of colors on the top and left edges.
  Since the horizontal edges can carry at most one color, this means that the top edge can carry at most one color more than the bottom edge.
  
  The statement then immediately holds in the case where the bottom edge has no colors at all.
  So suppose that the bottom edge has a single color $C$, and furthermore that the top edge carries two colors $D$ and $E$ which must therefore be distinct because the vertical edges are fermionic.
  If neither $D$ nor $E$ is equal to $C$, then we would get a contradiction from color conservation and the fact that the horizontal edges carry at most a single color.
  So we may assume that $E = C$.
  
  Recall that the column colors are increasing from left to right.
  If $D < C$ then we have the following situation
  \begin{equation}
    \begin{tikzpicture}[baseline, scale=0.5]
      \draw [blue, ultra thick] (1.5,0) -- (0,0) -- (0,1) node[label=above:$D\;$] {};
      \draw [red, ultra thick] (0.5,-1) -- (0.5,1) node[label=above:$\;C$] {};
    \end{tikzpicture}
  \end{equation}
  which has weight $0$ according to the first case of the $\texttt{a}_2$ weight in \cref{tab:T_right_specializations}.
  On the other hand, if $C < D$ we have
  \begin{equation}
    \begin{tikzpicture}[baseline, scale=0.5]
      \draw [blue, ultra thick] (2,0) -- (1,0) -- (1,1) node[label=above:$\;D$] {};
      \draw [red, ultra thick] (0.5,-1) -- (0.5,1) node[label=above:$C\;$] {};
    \end{tikzpicture}
  \end{equation}
  which has weight $0$ according to the $\texttt{b}_1$ weight.
  So in both cases we reach a contradiction.
  Thus we have shown that the top edge can carry at most one color which finishes the proof for the right-moving case.

  The proof for the left-moving case is analogous.
\end{proof}

\begin{table}
  \caption{The vertex configurations and weights of the fused Gamma Iwahori crystal model with at most one color on the vertical edges.}
  \label{tab:T_fused_gamma_Iwahori_crystal_weights}
  \newcommand{\scl}{0.8}
  \begin{tabular}{|c|c|c|c|c|c|c|c|}
    \hline
    \multicolumn{7}{|c|}{$\boldsymbol{T_\Right} = \boldsymbol{T_\Gamma}\vphantom{\displaystyle\int}$ \textbf{(fused crystal limit)}} \\
    \hline
    \hline
    $\texttt{a}_1$  &  $\texttt{a}_2$  &  $\texttt{a}_2'$  &  $\texttt{b}_1$  &  $\texttt{b}_2$  &  $\texttt{c}_1$  &  $\texttt{c}_2$
    \\ \hline
    \begin{tikzpicture}[scale=\scl]
      \draw (0,-1) -- (0,1);
      \draw (-1,0) -- (1,0);
    \end{tikzpicture}
    &
    \begin{tikzpicture}[scale=\scl]
      \draw [blue, ultra thick] (0,-1) -- (0,1) node[label=above:$c_j$] {};
      \draw [red, ultra thick] (-1,0) node[label={[label distance=-0.5em] left:$c_i$}] {} -- (1,0);
    \end{tikzpicture}
    &
    \begin{tikzpicture}[scale=\scl]
      \draw [blue, ultra thick] (1,0) -- (0,0) -- (0,1) node[label=above:$c_j$] {};
      \draw [red, ultra thick] (-1,0) node[label={[label distance=-0.5em] left:$c_i$}] {} -- (0,0) -- (0,-1);
    \end{tikzpicture}
    &
    \begin{tikzpicture}[scale=\scl]
      \draw (-1,0) -- (1,0);
      \draw [red, ultra thick] (0,-1) -- (0,1) node[label=above:$c_i$] {};
    \end{tikzpicture}
    &
    \begin{tikzpicture}[scale=\scl]
      \draw (0,-1) -- (0,1);
      \draw [red, ultra thick] (-1,0) node[label={[label distance=-0.5em] left:$c_i$}] {} -- (1,0);
    \end{tikzpicture}
    &
    \begin{tikzpicture}[scale=\scl]
      \draw (1,0) -- (0,0) -- (0,1);
      \draw [red, ultra thick] (-1,0) node[label={[label distance=-0.5em]left:$c_i$}] {} -- (0,0) -- (0,-1);
    \end{tikzpicture}
    &
    \begin{tikzpicture}[scale=\scl]
      \draw (-1,0) -- (0,0) -- (0,-1);
      \draw [red, ultra thick] (1,0) -- (0,0) -- (0,1) node[label=above:$c_i$] {};
    \end{tikzpicture}
    \\ \hline

    $\displaystyle 1$

    & $\displaystyle
    \begin{cases}
      0 & i < j \\
      z & i \geq j
    \end{cases}$

    & $\displaystyle
    \begin{cases}
      z & i < j \\
      0 & i > j
    \end{cases}$

    & $\displaystyle 0$

    & $\displaystyle z$

    & $\displaystyle z$

    & $\displaystyle 1$
    \\ \hline
  \end{tabular}
\end{table}
\begin{table}
  \caption{The vertex configurations and weights obtained by fusing the Delta Iwahori crystal model, restricting to at most one color on the vertical edges and then multiplying each weight by~$z$.
  The multiplication by $z$ is convenient since it removes the dependence on the number of columns for the partition function.}
  \label{tab:T_fused_delta_Iwahori_crystal_weights}
  \newcommand{\scl}{0.8}
  \begin{tabular}{|c|c|c|c|c|c|c|}
    \hline
    \multicolumn{7}{|c|}{$z \boldsymbol{T_\Left} = z \boldsymbol{T_\Delta}\vphantom{\displaystyle\int}$ \textbf{(fused crystal limit)}} \\
    \hline
    \hline
    $\texttt{a}_1$  &  $\texttt{a}_2$  &  $\texttt{a}_2'$  &  $\texttt{b}_1$  &  $\texttt{b}_2$  &  $\texttt{c}_1$  &  $\texttt{c}_2$
    \\ \hline
    \begin{tikzpicture}[scale=\scl]
      \draw (0,-1) -- (0,1);
      \draw (-1,0) -- (1,0);
    \end{tikzpicture}
    &
    \begin{tikzpicture}[scale=\scl]
      \draw [blue, ultra thick] (0,-1) -- (0,1) node[label=above:$c_j$] {};
      \draw [red, ultra thick] (-1,0) -- (1,0) node[label={[label distance=-0.5em] right:$c_i$}] {};
    \end{tikzpicture}
    &
    \begin{tikzpicture}[scale=\scl]
      \draw [blue, ultra thick] (-1,0) -- (0,0) -- (0,1) node[label=above:$c_j$] {};
      \draw [red, ultra thick] (0,-1) -- (0,0) -- (1,0) node[label={[label distance=-0.5em] right:$c_i$}] {};
    \end{tikzpicture}
    &
    \begin{tikzpicture}[scale=\scl]
      \draw (-1,0) -- (1,0);
      \draw [red, ultra thick] (0,-1) -- (0,1) node[label=above:$c_i$] {};
    \end{tikzpicture}
    &
    \begin{tikzpicture}[scale=\scl]
      \draw (0,-1) -- (0,1);
      \draw [red, ultra thick] (-1,0) -- (1,0) node[label={[label distance=-0.5em] right:$c_i$}] {};
    \end{tikzpicture}
    &
    \begin{tikzpicture}[scale=\scl]
      \draw (-1,0) -- (0,0) -- (0,1);
      \draw [red, ultra thick] (0,-1) -- (0,0) -- (1,0) node[label={[label distance=-0.5em] right:$c_i$}] {};
    \end{tikzpicture}
    &
    \begin{tikzpicture}[scale=\scl]
      \draw (1,0) -- (0,0) -- (0,-1);
      \draw [red, ultra thick] (-1,0) -- (0,0) -- (0,1) node[label=above:$c_i$] {};
    \end{tikzpicture}
    \\ \hline

    $\displaystyle z$

    & $\displaystyle
    \begin{cases}
      1 & i \leq j \\
      0      & i > j
    \end{cases}$

    & $\displaystyle
    \begin{cases}
      0      & i < j \\
      1 & i > j
    \end{cases}$

    & $\displaystyle 0$

    & $\displaystyle 1$

    & $\displaystyle z$

    & $\displaystyle 1$
    \\ \hline
  \end{tabular} 
\end{table}

Recall that in our usual boundary conditions for a system the edges along the bottom boundary carry no paths; in other words they are unoccupied.
By arguing inductively upwards from the bottom row the above lemma then implies that the only vertex configurations that can appear in an admissible state are those with at most one color on the vertical paths.
Since we are always going to impose boundary conditions of this kind, we can restrict the crystal models to these admissible vertex configurations without affecting the partition functions, and from now on, it is always this restricted set of admissible vertex configurations that we will use for crystal models.

Following \cref{sec:fusion}, we can compute the Boltzmann weights of the fused admissible vertex configurations with at most one color on the vertical paths, and these can be found in \cref{tab:T_fused_gamma_Iwahori_crystal_weights,tab:T_fused_delta_Iwahori_crystal_weights}.
The weights of the latter have been multiplied by $z$ which conveniently removes the partition function's dependence on the number of columns $N$, since any extra $\texttt{b}_2$ configurations added on the left side of the grid would only contribute a factor of $1$.
We make this change \emph{after} the fusion process instead of making an analogous change for the unfused weights for the whole family to avoid factors of $z^{1/m}$.

As an example, \cref{fig:a2'_weight} illustrates the weight computation for one case of the $\texttt{a}_2'$ vertex configuration.

\begin{figure}
  \begin{tikzpicture}
    \newcommand{\lineWidth}{0.6}
    \newcommand{\bdry}{0.8}
    \newcommand{\columnColorsHeight}{1.3}
    \draw (-1.7, \columnColorsHeight) node {column colors:};
    \draw[red] (1, \columnColorsHeight) node {$c_i$};
    \draw[blue] (3, \columnColorsHeight) node {$c_j$};
    \newcommand{\weightHeight}{-1.4}
    \draw (-1.7, \weightHeight) node {weights:};
    \draw (0, \weightHeight) node {1};
    \draw (1, \weightHeight) node {$z$};
    \draw (2, \weightHeight) node {1};
    \draw (3, \weightHeight) node {1};
    \draw (4, \weightHeight) node {1};
    \begin{scope}
      \foreach \x in {0,...,4} {
        \draw (\x, -\bdry) -- (\x, \bdry);
      }
      \draw (-1, 0) -- (5, 0);

      \draw[blue, line width=\lineWidth mm] (3,\bdry) -- (3,0) -- (5,0);
      \draw[red, line width=\lineWidth mm] (-1,0) -- (1,0) -- (1,-\bdry);
    \end{scope}
    \draw[-Latex, thick] (5.5,0) -- (6.2,0);
    \begin{scope}[xshift=8cm]
      \draw [blue, line width=\lineWidth mm] (\bdry,0) -- (0,0) -- (0,\bdry) node[label=above:$c_j$] {};
      \draw [red, line width=\lineWidth mm] (-\bdry,0) node[label={[label distance=-0.5em] left:$c_i$}] {} -- (0,0) -- (0,-\bdry);
      \draw (0, \weightHeight) node {$z$};
    \end{scope}
  \end{tikzpicture} 
  \caption{Fusion of an $\texttt{a}_2'$ vertex of the left-moving Iwahori crystal model with $i < j$.}
  \label{fig:a2'_weight}
\end{figure}

\section{Solvability, Yang--Baxter equations and the duality}
\label{sec:YBE}
In this section we will show that the families of left- and right-moving models introduced in \cref{sec:families_of_lattice_models} satisfy various Yang--Baxter equations.
The general Yang--Baxter equation with respect to three finite-dimensional vector spaces (or quantum group modules as will be discussed further in \cref{sec:quantum-groups}) $U$, $V$ and $W$ as well as three linear maps $f : U \otimes V \to V \otimes U$, $g: U \otimes W \to W \otimes U$ and $h : V \otimes W \to W \otimes V$ is the following functional equation 
\begin{equation}
  \label{eq:YBE}
  \llbracket f, g, h \rrbracket \coloneq (\id_W \otimes f)(g \otimes \id_V)(\id_U \otimes h) - (h \otimes \id_U)(\id_V \otimes g)(f \otimes \id_W) = 0
\end{equation}
as maps $U \otimes V \otimes W \to W \otimes V \otimes U$ where we have suppressed the compositions and where $\id_A$ denotes the identity map on $A$.
This definition of the Yang--Baxter commutator $\llbracket f, g, h \rrbracket$ is related the one used in, for example, \cite{BBF} which considers endomorphisms of $U\otimes V$ instead of maps $U \otimes V \to V \otimes U$, via composition with the linear map $\tau : U \otimes V \to V \otimes U$ defined on pure tensors by $a \otimes b \mapsto b \otimes a$.

We now describe how this Yang--Baxter equation is connected to lattice models in general.
The spaces $U$, $V$ and $W$ are free vector spaces on spin sets $\Sigma_U$, $\Sigma_V$, $\Sigma_W$ associated to different types of edges of the lattice model.
The linear maps presented above on these vector spaces are interpreted as vertices of the lattice models with edges carrying the associated spin sets. 
More precisely, if $\Sigma_U$ are the horizontal spins and $\Sigma_V$ the vertical spins of a standard vertex~\tikz[scale=0.18, thick]{\draw (-1,0) -- (1,0); \draw (0,-1) -- (0,1); \node[dot,minimum size=3pt] at (0,0){};}, the associated linear map $f : U \otimes V \to V \otimes U$ is given by 
\begin{equation}
  f(a \otimes b) = \sum_{\substack{c \in \Sigma_U \\ d \in \Sigma_V}} \; \wt \! \left( 
  \begin{tikzpicture}[baseline=-1mm, scale=0.3]
    \draw (-1,0) node[left] {$a$} -- (0,0) node[dot] {} --  (1,0) node[right] {$c$};
    \draw (0,1) node[above] {$b$} -- (0,-1) node[below] {$d$}; 
  \end{tikzpicture}
  \right) d \otimes c  \qquad a \in \Sigma_U, b \in \Sigma_V.
\end{equation}
In other words, the Boltzmann weights are the matrix coefficients of $f$.

Recall from \cref{sec:families_of_lattice_models} that we have both left-moving ($\Left$) and right-moving ($\Right$) lattice models.
Although the spin sets for the horizontal edges are the same for the left- and right-moving models, it is beneficial to differentiate between the vector spaces for the horizontal edges as $V_\Left$ and $V_\Right$, which are vector spaces over the field $\mathcal{F} = \mathbb{C}(q,\Phi,\mathfrak{X}_{1,1},\mathfrak{X}_{1,2},\dots)$ introduced in \cref{sec:fusion}.
Recall that their spin sets are $\Sigma_{V_\Left} = \Sigma_{V_\Right} = \mathcal{P} \cup \{\emptyset\}$.
For both models the vector space for the vertical edges is the free $\mathcal{F}$-vector space on the spin set which is the power set of the palette $\mathcal{P}$.

More generally, we will consider a so-called \emph{parametrized} Yang--Baxter equation where we introduce formal row parameters $z_i$ by replacing the horizontal vector space $V_X$ on row $i$ with $V_X(z_i) \coloneq V_X \otimes \mathcal{F}(z_i)$.
Unless otherwise specified, tensor products in this section will be over $\mathcal{F}$.

We will need two types of Yang--Baxter equations.
In the first one, called the $RTT$-equation, $W$ corresponds to a vertical edge, and $U = V_X(z_1)$ and $V = V_Y(z_2)$ correspond to horizontal edges of type $X$ and $Y$ (with $X,Y \in \{\Left, \Right\}$) where we have chosen rows $1$ and $2$ for notational simplicity; the same can be done for any rows $i$ and $j$.
We then consider maps $\tau g(z_1) \in \End_{\mathcal{F}(z_1)}(V_X(z_1) \otimes W) \iso \End_\mathcal{F}(V_X \otimes W) \otimes \mathcal{F}(z_1)$ and $\tau h(z_2) \in \End_{\mathcal{F}}(V_Y \otimes W) \otimes \mathcal{F}(z_2)$ associated to standard vertices \tikz[scale=0.18, thick]{\draw (-1,0) -- (1,0); \draw (0,-1) -- (0,1); \node[dot,minimum size=3pt] at (0,0){};} at rows $1$ and $2$ in the lattice model.

We also consider a map
\begin{equation*}
  \tau f(z_1,z_2) \in \End_{\mathcal{F}(z_1,z_2)}\Bigl(\bigl(V_X(z_1) \otimes V_Y(z_2)\bigr) \otimes_{\mathcal{F}(z_1) \otimes \mathcal{F}(z_2)} \mathcal{F}(z_1,z_2)\Bigr)
  \iso \End_\mathcal{F}(V_X \otimes V_Y) \otimes \mathcal{F}(z_1,z_2) 
\end{equation*}
which involves only horizontal edges for two different rows and can be seen as swapping these rows\footnote{Note the extension from $\mathcal{F}(z_1) \otimes \mathcal{F}(z_2)$ to the fraction field $\mathcal{F}(z_1,z_2)$.}.
In the lattice model we associate this map with a new type of vertex \tikz[scale=0.18, thick]{\draw (-1, -1) -- (1,1); \draw (1,-1) -- (-1,1); \node[dot,minimum size=3pt] at (0,0){};} connecting four horizontal edges.
It is customary to call this new type of vertex an $R$-vertex, and the standard vertex above a $T$-vertex, with the corresponding maps denoted $R$ and $T$ which explains the name $RTT$-equation.
We use the following labeling convention where we in the Boltzmann weights have included the row parameters 
\begin{align}
  T_X(z_i)(a \otimes b) &= \sum_{\substack{c \in \Sigma_{V_X} \\ d \in \Sigma_W}} \; \wt \! \left( 
  \begin{tikzpicture}[baseline=-1mm, scale=0.3]
    \draw (-1,0) node[left] {$a$} -- (0,0) node[dot] {} --  (1,0) node[right,label=right:$\scriptstyle z_i$] {$c$};
    \draw (0,1) node[above] {$b$} -- (0,-1) node[below] {$d$}; 
  \end{tikzpicture}
\right) d \otimes c  &  a \in \Sigma_{V_X}, b \in \Sigma_W
\\
R^X_Y(z_1, z_2)(a \otimes b) &= \sum_{\substack{c \in \Sigma_{V_X} \\ d \in \Sigma_{V_Y}}}
  \wt \! \left( 
  \begin{tikzpicture}[baseline=-1mm, scale=0.3]
    \draw (-1,-1) node[left,label=left:$\scriptstyle z_1$] {$a$} -- (0,0) node[dot] {} --  (1,1) node[right,label=right:$\scriptstyle z_1$] {$c$};
    \draw (-1,1) node[left,label=left:$\scriptstyle z_2$] {$b$} -- (1,-1) node[right,label=right:$\scriptstyle z_2$] {$d$}; 
  \end{tikzpicture}
\right) d \otimes c  &  a \in \Sigma_{V_X}, b \in \Sigma_{V_Y}
\end{align}
where the weights are elements of $\mathcal{F}(z_i)$ and $\mathcal{F}(z_1,z_2)$ respectively.

The associated Yang--Baxter equation $\llbracket R^X_Y(z_1,z_2), T_X(z_1), T_Y(z_2) \rrbracket = 0$ over $\mathcal{F}(z_1,z_2)$ is called a \emph{parametrized} Yang--Baxter equation, but we will often drop this adjective and sometimes also the $z_1$ and $z_2$ arguments which can be recovered by keeping track of the horizontal spaces for the different rows.

The lattice model description of this parametrized Yang--Baxter equation states that the partition functions for the following systems with fixed boundary conditions are equal.
\begin{equation}
\begin{gathered}
  \label{eq:RTT-picture}
  \llbracket R^X_Y(z_1,z_2), T_X(z_1), T_Y(z_2) \rrbracket = 0\colon\hfill\\
  \begin{tikzpicture}[baseline=-1mm, scale=0.75]
    \draw (-1, 1) node[left,label=left:$\scriptstyle z_2$]{$b$} -- (1,-1) node[label={[blue, label distance=-1mm]below:$j$}]{} -- (3,-1) node[right,label=right:$\scriptstyle z_2$]{$e$};
    \draw (-1,-1) node[left,label=left:$\scriptstyle z_1$]{$a$} -- (1, 1) node[label={[blue, label distance=-1mm]above:$i$}]{} -- (3, 1) node[right,label=right:$\scriptstyle z_1$]{$d$};
    \draw (2,2) node[label=above:$c$]{} -- (2,0) node[label={[blue, label distance=-1mm]left:$k$}]{} -- (2,-2) node[label=below:$f$]{};
    \node[label=left:$R^X_Y$, dot] at (0,0) {};
    \node[label=north east:$T_X$, dot] at (2,1) {};
    \node[label=north east:$T_Y$, dot] at (2,-1) {};
  \end{tikzpicture}
  \quad
  =
  \quad
  \begin{tikzpicture}[baseline=-1mm,xscale=-0.75,yscale=-0.75]
    \draw (-1, 1) node[right,label=right:$\scriptstyle z_2$]{$e$} -- (1,-1) node[label={[blue, label distance=-1mm]above:$l$}]{} -- (3,-1) node[left,label=left:$\scriptstyle z_2$]{$b$};
    \draw (-1,-1) node[right,label=right:$\scriptstyle z_1$]{$d$} -- (1, 1) node[label={[blue, label distance=-1mm]below:$m$}]{} -- (3, 1) node[left,label=left:$\scriptstyle z_1$]{$a$};
    \draw (2,2) node[label=below:$f$]{} -- (2,0) node[label={[blue, label distance=-1mm]right:$n$}]{} -- (2,-2) node[label=above:$c$]{};
    \node[label=right:$R^X_Y$, dot] at (0,0) {};
    \node[label=north west:$T_X$, dot] at (2,1) {};
    \node[label=north west:$T_Y$, dot] at (2,-1) {};
  \end{tikzpicture}
\end{gathered}
\end{equation}
The internal edges $i$, $j$, $k$, $l$, $m$ and $n$ are summed over while the boundary edges are kept fixed.
These $RTT$-equations can be used to obtain functional equations for partition functions by a repeated use of \eqref{eq:RTT-picture} in what is called the \emph{train argument}; see for example the proof of \cref{lemma:mixed_model_exchange_rows}.

The other type of Yang--Baxter equation we will consider consists of three $R$-vertices and is therefore called the $RRR$-equation.
Here $U = V_X(z_1)$, $V = V_Y(z_2)$ and $W = V_Z(z_3)$ all correspond to horizontal edges of row types $X,Y,Z \in \{\Left,\Right\}$ with row parameters $z_1$, $z_2$ and~$z_3$.
The associated Yang--Baxter equation is $\llbracket R^X_Y(z_1,z_2), R^X_Z(z_1,z_3), R^Y_Z(z_2,z_3) \rrbracket = 0$ over $\mathcal{F}(z_1,z_2,z_3)$ and its lattice model description states that the partition functions for the following systems with fixed boundary conditions are equal.

\begin{equation}
  \begin{gathered}
  \label{eq:RRR-picture}
  \llbracket R^X_Y(z_1,z_2), R^X_Z(z_1,z_3), R^Y_Z(z_2,z_3) \rrbracket = 0\colon\hfill\\[0.5em]
  \begin{tikzpicture}[baseline, scale=0.75, shift={(0,-1)}]
    \draw (-1,-1) node[left,label=left:$\scriptstyle z_1$]{$a$} -- (1,1) node[label={[blue,label distance=-3mm]south east:$i$}]{} -- (3,3) node[right,label=right:$\scriptstyle z_1$]{$d$};
    \draw (-1,1) node[left,label=left:$\scriptstyle z_2$]{$b$} -- (1,-1) -- (2,-1) node[label={[blue,label distance=-2mm]above:$j$}]{} -- (3,-1) -- (5,1) node[right,label=right:$\scriptstyle z_2$]{$e$};
    \draw (1,3) node[left,label=left:$\scriptstyle z_3$]{$c$} -- (3,1) node[label={[blue,label distance=-3mm]south west:$k$}]{} -- (5,-1) node[right,label=right:$\scriptstyle z_3$]{$f$};
    \node[label=left:$R^X_Y$, dot] at (0,0) {}; 
    \node[label=left:$R^X_Z$, dot] at (2,2) {}; 
    \node[label=right:$R^Y_Z$, dot] at (4,0) {}; 
  \end{tikzpicture}
  \quad
  =
  \quad
  \begin{tikzpicture}[baseline,xscale=0.75,yscale=-0.75, shift={(0,-1)}]
    \draw (-1,-1) node[left,label=left:$\scriptstyle z_3$]{$c$} -- (1,1) node[label={[blue,label distance=-3mm]north east:$n$}]{} -- (3,3) node[right,label=right:$\scriptstyle z_3$]{$f$};
    \draw (-1,1) node[left,label=left:$\scriptstyle z_2$]{$b$} -- (1,-1) -- (2,-1) node[label={[blue,label distance=-2mm]below:$l$}]{} -- (3,-1) -- (5,1) node[right,label=right:$\scriptstyle z_2$]{$e$};
    \draw (1,3) node[left,label=left:$\scriptstyle z_1$]{$a$} -- (3,1) node[label={[blue,label distance=-3mm]north west:$m$}]{} -- (5,-1) node[right,label=right:$\scriptstyle z_1$]{$d$};
    \node[label=left:$R^Y_Z$, dot] at (0,0) {}; 
    \node[label=left:$R^X_Z$, dot] at (2,2) {}; 
    \node[label=right:$R^X_Y$, dot] at (4,0) {};
  \end{tikzpicture}
\end{gathered}
\end{equation}
As before, the internal edges $i$, $j$, $k$, $l$, $m$ and $n$ are summed over while the boundary edges are kept fixed.
Note how this picture is just a slight rearrangement of the one in~\eqref{eq:RTT-picture}.

We say that the left- and right-moving lattice models form a \emph{system of solvable lattice models} if they satisfy all of the above Yang--Baxter equations, i.e.\ \eqref{eq:RTT-picture} and \eqref{eq:RRR-picture} for all $X$, $Y$ and $Z$.
The name is inspired by the related Yang--Baxter systems discussed in~\cite{BBG:vertical,BBBGray} and Section 9 of the preprint of~\cite{BBF} referencing \cite{Hlavaty:1997, Hlavaty:1995, Vladimirov:1993}.
This is a stronger statement than the solvability of either the left- or right-moving models (the latter of which was shown in~\cite{BBBG:duality}) which means that \eqref{eq:RTT-picture} and \eqref{eq:RRR-picture} are satisfied for $X = Y = Z = \Left$ or $X = Y = Z = \Right$ respectively.
Indeed, when $X$, $Y$ and $Z$ are not all the same, we also have Yang--Baxter equations involving combinations of left- and right-moving rows and mixed $R$-matrices that interchange them.

The fusion process described in \cref{sec:fusion} for constructing lattice models was developed in \cite{BBBG:Iwahori} and Lemma~5.4 therein shows that if the unfused model satisfies \emph{auxiliary} or \emph{unfused} Yang--Baxter equations then this implies that the fused models satisfy the standard Yang--Baxter equations.
These auxiliary Yang--Baxter equations are very similar to those in~\eqref{eq:RTT-picture} and~\eqref{eq:RRR-picture} except that the vertices now carry a vertex color, all being the same color $c_k$ for some $k$ except for the $R$-vertex in the right-hand side of~\eqref{eq:RTT-picture} which is $c_{k+1}$ (where $c_{m+1} = c_1$).
The fused $R$-vertex is then the unfused $R$-vertex with $k=1$.
Since the auxiliary Yang--Baxter equations imply the standard Yang--Baxter equations we will often drop the term \emph{auxiliary} for unfused models.

We present the different (unfused) $R$-vertices and their weights for our models in \cref{tab:R-vertices,tab:R-vertices-mixed} and the following theorem states that they, together with the different $T$-vertices presented before satisfy all the Yang--Baxter equations of \eqref{eq:RTT-picture} and \eqref{eq:RRR-picture}.

\begin{table}[htpb]
  \centering
  \caption{Configurations and Boltzmann weights for the unfused $\Rll$ and $\Rrr$ attaching to the left of two vertically aligned $T$-vertices with vertex color $c_k$ and row parameters $z_1$ above $z_2$.
  The row parameters are swapped by the $R$-vertices so that $z_2$ is above $z_1$ at the $R$-vertices' left edges as shown in the first column.
  Here $i \neq j$ and the weights on the first row are equal for $\Rll$ and $\Rrr$ after substituting $\Phi$ with $1/(q^2\Phi)$, thus exchanging the second and third weights, while for the second row they are equal after swapping $z_1$ and $z_2$.}
  \label{tab:R-vertices}
  \vspace{-0.5em}
  \begin{equation*}
    \begin{array}{|c|c|c|c|c|c|}
      \hline
      \multicolumn{6}{|c|}{\boldsymbol{\Rll} \textbf{ and } \boldsymbol{\Rrr} \vphantom{\displaystyle\int}}
      \\\hline\hline
      \begin{tikzpicture}[baseline, scale=0.75]
        \node[anchor=south] at (-1,-1) {$z_1$};
        \node[anchor=north] at (-1,1) {$z_2$};
      \end{tikzpicture}
      &
      \begin{tikzpicture}[baseline, scale=0.75]
        \draw (-1,-1) -- (1,1);
        \draw (-1,1) -- (1,-1);
        \node[label=above:$c_k$] (0,0) {};
      \end{tikzpicture}  
      &
      \begin{tikzpicture}[baseline, scale=0.75]
        \draw (-1,1) -- (1,-1);
        \draw[ultra thick, red] (-1,-1) -- (1,1)  node[label=below:$c_i$]{};
        \node[label=above:$c_k$] (0,0) {};
      \end{tikzpicture} 
      &
      \begin{tikzpicture}[baseline, scale=0.75]
        \draw (-1,-1) -- (1,1);
        \draw[ultra thick, red] (-1,1) -- (1,-1) node[label=above:$c_i$]{};
        \node[label=above:$c_k$] (0,0) {};
      \end{tikzpicture}
      &
      \begin{tikzpicture}[baseline, scale=0.75]
        \draw[ultra thick, red] (-1,-1) -- (1,1) node[label=below:$c_i$]{};
        \draw[ultra thick, red] (-1,1) -- (1,-1) node[label=above:$c_i$]{};
        \node[label=above:$c_k$] (0,0) {};
      \end{tikzpicture}
      &
      \begin{tikzpicture}[baseline, scale=0.75]
        \draw[ultra thick, blue] (-1,1) -- (1,-1) node[label=above:$c_j$]{};
        \draw[ultra thick, red] (-1,-1) -- (1,1) node[label=below:$c_i$]{};
        \node[label=above:$c_k$] (0,0) {};
      \end{tikzpicture}

      \\\hline
      
      \Rll\vphantom{\Bigl)}
      & z_1 - q^2 z_2 
      & q^2\Phi(z_1-z_2) 
      & (z_1 - z_2)/\Phi 
      & z_2 - q^2 z_1
      & \mathfrak{X}_{i,j} (z_1-z_2)

      \\\hline

      \Rrr\vphantom{\Bigl)}
      & \multicolumn{5}{c|}{\Phi \longrightarrow 1/(q^2\Phi)}       
      
      \\\hline\hline

      \begin{tikzpicture}[baseline, scale=0.75]
        \node[anchor=south] at (-1,-1) {$z_1$};
        \node[anchor=north] at (-1,1) {$z_2$};
      \end{tikzpicture}
      &
      \begin{tikzpicture}[baseline, scale=0.75]
        \draw (-1,-1) -- (0,0) -- (1,-1);
        \draw[ultra thick, red] (-1,1) -- (0,0) --  (1,1) node[label=below:$c_i$]{};
        \node[label=above:$c_k$] (0,0) {};
      \end{tikzpicture}
      &
      \begin{tikzpicture}[baseline, scale=0.75]
        \draw (-1,1) -- (0,0) --  (1,1);
        \draw[ultra thick, red] (-1,-1) -- (0,0) -- (1,-1) node[label=above:$c_i$]{};
        \node[label=above:$c_k$] (0,0) {};
      \end{tikzpicture}  
      & \multicolumn{3}{c|}{
      \begin{tikzpicture}[baseline, scale=0.75]
        \draw[ultra thick, blue] (-1,1) -- (0,0) --  (1,1) node[label=below:$c_j$]{};
        \draw[ultra thick, red] (-1,-1) -- (0,0) -- (1,-1) node[label=above:$c_i$]{};
        \node[label=above:$c_k$] (0,0) {};
      \end{tikzpicture} 
      }
      \\\hline
      \Rll\vphantom{\Bigl)} & (1-q^2)z_1 & (1-q^2)z_2 & 
      \multicolumn{3}{c|}{
        (1-q^2) \cdot
        \left\{\begin{smallmatrix*}[l]
          z_2 & \text{if $i < j < k$ or $j < k \leq i$ or $k \leq i < j$}\\
          z_1 & \text{if $j < i < k$ or $i < k \leq j$ or $k \leq j < i$}
        \end{smallmatrix*}\right.
      }
      \\\hline
      \Rrr\vphantom{\Bigl)}
      &
      \multicolumn{5}{c|}{z_1 \longleftrightarrow z_2}
      \\\hline
    \end{array}
  \end{equation*}  
\end{table}


\begin{table}[htpb]
  \centering
  \caption{Configurations and Boltzmann weights for the unfused $\Rlr$ and $\Rrl$ attaching to the left of two vertically aligned $T$-vertices with vertex color $c_k$ and row parameters $z_1$ above~$z_2$.
  The row parameters are swapped by the $R$-vertices so that $z_2$ is above $z_1$ at the $R$-vertices' left edges as shown in the first column.
  Here $i \neq j$ and the conditions for the last weights are the same as in \cref{tab:T_left_specializations}.}
  \label{tab:R-vertices-mixed}
  \vspace{-0.5em}

  \makebox[\textwidth][c]{
  \begin{tblr}{width=1.05\textwidth, hlines, vlines, colspec={cccccX[c]}, columns={mode=math}}
    \SetCell[c=6]{c}{\boldsymbol{\Rlr} \textbf{ and } \boldsymbol{\Rrl}\vphantom{\displaystyle\int}} 
      \\\hline
      \begin{tikzpicture}[baseline, scale=0.75]
        \node[anchor=south] at (-1,-1) {$z_1$};
        \node[anchor=north] at (-1,1) {$z_2$};
      \end{tikzpicture}
      &
      \begin{tikzpicture}[baseline, scale=0.75]
        \draw (-1,-1) -- (1,1);
        \draw (-1,1) -- (1,-1);
        \node[label=above:$c_k$] (0,0) {};
      \end{tikzpicture}  
      &
      \begin{tikzpicture}[baseline, scale=0.75]
        \draw (-1,1) -- (1,-1);
        \draw[ultra thick, red] (-1,-1) -- (1,1)  node[label=below:$c_i$]{};
        \node[label=above:$c_k$] (0,0) {};
      \end{tikzpicture} 
      &
      \begin{tikzpicture}[baseline, scale=0.75]
        \draw (-1,-1) -- (1,1);
        \draw[ultra thick, red] (-1,1) -- (1,-1) node[label=above:$c_i$]{};
        \node[label=above:$c_k$] (0,0) {};
      \end{tikzpicture}
      &
      \begin{tikzpicture}[baseline, scale=0.75]
        \draw[ultra thick, red] (-1,-1) -- (1,1) node[label=below:$c_i$]{};
        \draw[ultra thick, red] (-1,1) -- (1,-1) node[label=above:$c_i$]{};
        \node[label=above:$c_k$] (0,0) {};
      \end{tikzpicture}
      &
      \begin{tikzpicture}[baseline, scale=0.75]
        \draw[ultra thick, blue] (-1,1) -- (1,-1) node[label=above:$c_j$]{};
        \draw[ultra thick, red] (-1,-1) -- (1,1) node[label=below:$c_i$]{};
        \node[label=above:$c_k$] (0,0) {};
      \end{tikzpicture}
      \\
      \Rlr \vphantom{\Bigl)}
      & z_1 - z_2 
      & \Phi (z_1- q^2 z_2) 
      & \Phi(z_1- q^2 z_2) 
      & \Phi^2(q^4 z_2 - z_1) 
      & \Phi^2 \mathfrak{X}_{j,i} (z_1-q^2z_2)
      \\
      \Rrl \vphantom{\Bigl)}
      & z_2 - q^{2m+2} z_1
      & (z_2 - q^{2m}z_1)/\Phi
      & (z_2 - q^{2m}z_1)/\Phi
      & (q^{2m-2}z_1 - z_2)/\Phi^2
      & \frac{\mathfrak{X}_{j,i}}{(q\Phi)^2} (z_2 - q^{2m}z_1) 
  \end{tblr}
}
\\[0.35mm]
\makebox[\textwidth][c]{
  \begin{tblr}{width=1.05\textwidth, hlines, vlines, colspec={cccX[c]}, columns={mode=math}}
      \begin{tikzpicture}[baseline, scale=0.75]
        \node[anchor=south] at (-1,-1) {$z_1$};
        \node[anchor=north] at (-1,1) {$z_2$};
      \end{tikzpicture}
      &
      \begin{tikzpicture}[baseline, scale=0.75]
        \draw (-1,-1) -- (0,0) --  (-1,1);
        \draw[ultra thick, red] (1,1) -- (0,0) -- (1,-1) node[label=above:$c_i$]{};
        \node[label=above:$c_k$] (0,0) {};
      \end{tikzpicture} 
      &
      \begin{tikzpicture}[baseline, scale=0.75]
        \draw (1,-1) -- (0,0) -- (1,1);
        \draw[ultra thick, red] (-1,1) -- (0,0) --  (-1,-1) node[label=above:$c_i$]{};
        \node[label=above:$c_k$] (0,0) {};
      \end{tikzpicture}
      & 
      \begin{tikzpicture}[baseline, scale=0.75]
        \draw[ultra thick, blue] (1,1) -- (0,0) --  (1,-1) node[label=above:$c_j$]{};
        \draw[ultra thick, red] (-1,1) -- (0,0) -- (-1,-1) node[label=above:$c_i$]{};
        \node[label=above:$c_k$] (0,0) {};
      \end{tikzpicture} 
      \\
      \Rlr \vphantom{\Bigl)} 
      & \Phi(1-q^2)z_1 
      & \Phi(1-q^2)z_2 
      & 
        -\Phi^2(1-q^2) \cdot
        \left\{\begin{smallmatrix*}[l]
          q^2 z_2 & \text{if $i<j<k$, \ldots}\\
          z_1 & \text{if $j<i<k$, \ldots}
        \end{smallmatrix*}\right.
      \\
      \Rrl \vphantom{\Bigl)}
      & \frac{q^{2\res^m(k-i)-2}}{\Phi}(1-q^2)z_2
      & \frac{q^{2\res_m(i-k)}}{\Phi}(1-q^2)z_1
      & 
      \frac{q^{2\res_m(i-j)-2}}{\Phi^2} (1-q^2)\cdot
        \left\{\begin{smallmatrix*}[l]
          z_1 & \text{if $i<j<k$, \ldots}\\
          z_2 & \text{if $j<i<k$, \ldots}
        \end{smallmatrix*}\right.
      \\
    \end{tblr}
  }
\end{table}

\begin{theorem}
  \label{thm:YBE}
  The unfused $R$-vertices of \cref{tab:R-vertices,tab:R-vertices-mixed} and $T$-vertices of \cref{tab:T_right_weights,tab:T_left_weights} satisfy the auxiliary Yang--Baxter $RTT$-equations in~\eqref{eq:RTT-picture} and the $RRR$-equations in~\eqref{eq:RRR-picture} for all combinations of row types $X,Y,Z \in \{\Left,\Right\}$.
\end{theorem}

It then follows from Lemma 5.4 in \cite{BBBG:Iwahori} that the fused models satisfies the standard Yang--Baxter equations in~\eqref{eq:RTT-picture} and~\eqref{eq:RRR-picture} for all $X,Y,Z \in \{\Left,\Right\}$. 
\begin{corollary}
  The fused left- and right-moving models form a system of solvable lattice models.
  In particular, each of them is individually solvable, and there are Yang--Baxter equations for any combination of left- and right-moving rows.
\end{corollary}

We will prove \cref{thm:YBE} by showing that the Yang--Baxter equations are satisfied for a specific choice of $\Phi$ and $\mathfrak{X}_{i,j}$, and that this then implies that they hold for all $\Phi$ and~$\mathfrak{X}_{i,j}$.

For this choice it would be possible to choose the metaplectic specialization where many of the Yang--Baxter equations are proven.
For example all the metaplectic $RTT$-equations are shown in \cite{BBB, Gray:thesis} and the metaplectic $X=Y=Z=\Gamma$ $RRR$-equation in \cite{BBB}.
The remaining metaplectic equations should be able to be obtained from the preprint \cite{BBBGray}.
However, the relevant part of the preprint was not included in the published version which appeared as an appendix to \cite{BBB} and it relied itself on results in the preprint of \cite{BBB} which were also not published.

In what follows we will instead give an independent and self-contained proof of \cref{thm:YBE} starting from the Iwahori specialization where many things simplify drastically.
The Yang--Baxter equations involving only right-moving row types were originally proved in \cite{BBBG:Iwahori}, but we do not rely on this proof either since this case is covered by our general argument with no additional effort.
We choose this particular specialization because all the Iwahori $T$-weights in \cref{tab:T_right_specializations,tab:T_left_specializations} and the Iwahori specializations of the $R$-weights in \cref{tab:R-vertices,tab:R-vertices-mixed} except for $\Rrl$ depend on the color indices $i$, $j$, and $k$ only by their relative orders in the palette.
In contrast, $\Rrl$ depends on the actual values of $i$, $j$, and $k$ for the involved colors.
Indeed, $\Rrl$ depends on the residues 
$$[0, m - 1] \ni \res_m (x) \equiv x \equiv \res^m (x) \in [1, m].$$
This is important since we want to reduce the proof of the Yang--Baxter equation to checking a finite number of equations independent of $m$.

We do not actually need the Yang--Baxter equations involving $\Rrl$ for the rest of this paper, but we will, for completeness, show that the Yang--Baxter equations involving $\Rrl$ follow from the other equations appearing in \cref{thm:YBE} using the following lemma and proposition.

The following lemma is a slight generalization of \cite[Lemma~4.1 and~Corollary~4.2]{Naprienko22} which treat the case $U = V = W$.
\begin{lemma}
  \label{lemma:YB-bracket_inverses}
  Let $U, V$ and $W$ be vector spaces with linear maps $f: U \otimes V \to V \otimes U$, $g: U \otimes W \to W \otimes U$ and $h: V \otimes W \to W \otimes V$.
  If $f$ is invertible, then
  \begin{equation*}
    \llbracket f, g, h \rrbracket = 0 \iff \llbracket f^{-1}, h, g \rrbracket = 0.
  \end{equation*}
  If $h$ is invertible, then
  \begin{equation*}
    \llbracket f, g, h \rrbracket = 0 \iff \llbracket g, f, h^{-1} \rrbracket = 0.
  \end{equation*}
\end{lemma}
\begin{proof}
  From the definition \eqref{eq:YBE} of the Yang--Baxter bracket we get by direct computation that if $f$ is invertible, then
  \begin{equation*}
    (\id_W \otimes f^{-1}) \llbracket f, g, h \rrbracket (f^{-1} \otimes \id_W) = - \llbracket f^{-1}, h, g \rrbracket,
  \end{equation*}
  and if $h$ is invertible, then
  \begin{equation*}
    (h^{-1} \otimes \id_U) \llbracket f, g, h \rrbracket (\id_U \otimes h^{-1}) = - \llbracket g, f, h^{-1} \rrbracket. \qedhere
  \end{equation*}
\end{proof}

\begin{proposition}
  \label{prop:inverse}
  The maps $\Rlr$ and $\Rrl$ for the lattice models of this paper defined by \cref{tab:R-vertices-mixed} satisfy
  \begin{equation*}
    \Rrl(z_1,z_2) \Rlr(z_2,z_1) = C \cdot \id_{\Vl} \otimes \id_{\Vr} \quad \text{and} \quad \Rlr(z_2,z_1) \Rrl(z_1,z_2) = C \cdot \id_{\Vr} \otimes \id_{\Vl}
  \end{equation*} 
  where $C = (z_2 - q^2 z_1) (z_2 - q^{2 m} z_1)$.
\end{proposition}

See \cref{appendix:proof_of_R-matrix_inverse} for the proof.
Suppressing the factor of $C$, the first equation corresponds diagrammatically to showing that
\begin{equation}
  \label{eq:diagram_mixed_R-matrices_inverses}
  \begin{tikzpicture}[baseline={(0,-0.5ex)}, scale = 0.4]
    \draw (-1,-1) node[left] {$a$} -- (0,0);
    \draw (-1,1) node[left] {$b$} -- (0,0);
    \draw (0,0) -- (1,1) -- (3,-1) node[right] {$d$};
    \draw (0,0) -- (1,-1) -- (3,1) node[right] {$c$};
    \node[dot, label={[label distance=1em]above:$\scriptstyle{\Rlr}$}] at (0,0) {};
    \node[dot, label={[label distance=1em]above:$\scriptstyle{\Rrl}$}] at (2,0) {};
  \end{tikzpicture}
  \quad = \quad
  \begin{tikzpicture}[baseline={(0,-0.5ex)}, scale = 0.4]
      \draw (-1,-1) node[left] {$a$} -- (3,-1) node[right] {$d$};
      \draw (-1,1) node[left] {$b$} -- (3,1) node[right] {$c$};
  \end{tikzpicture}
\end{equation}
In contrast to the other proofs of this section, the computations in \cref{appendix:proof_of_R-matrix_inverse} depend strongly on the number of colors $m$ especially since the interior edges on the left-hand side of \eqref{eq:diagram_mixed_R-matrices_inverses} can become a sum over all $m$ colors of the palette even for fixed, given boundary edges $a$, $b$, $c$ and $d$.
Similar \emph{color loops} can occur in Yang--Baxter equations involving $\Rrl$, and it is somewhat surprising that we obtain such a simple right hand side in \cref{prop:inverse}, and that the Yang--Baxter equations hold for all $m$.
As we show in \cref{appendix:proof_of_R-matrix_inverse}, this is explained by the fact that these sums of internal color loops become telescoping sums resulting in this simple expression for $C$.

\begin{corollary}
  \label{cor:R_RightLeft}
  To prove \cref{thm:YBE} it is enough to prove the cases which do not involve~$\Rrl$ (i.e.\ those involving only $\Tr$, $\Tl$, $\Rrr$, $\Rll$ and $\Rlr$).
\end{corollary}
\begin{proof}
  Since we may rescale $\Rrl$ by the $C$ in \cref{prop:inverse} without affecting the Yang--Baxter equation, the corollary follows from \cref{prop:inverse,lemma:YB-bracket_inverses}.
  Indeed, from the first identity in \cref{prop:inverse} gives us
  \begin{equation*}
    \llbracket \Rrl(z_1,z_2), T_R(z_1), T_L(z_2) \rrbracket = 0 \iff \llbracket \Rlr(z_2,z_1), T_L(z_2), T_R(z_1) \rrbracket = 0,
  \end{equation*}
  and so the $RTT$-equation involving $\Rrl$ follows from the one involving $\Rlr$.
  Similarly, the first and the second identities respectively in \cref{prop:inverse} give us
  \begin{align*}
    \llbracket \Rrl(z_1,z_2), \Rrl(z_1,z_3), \Rll(z_2,z_3) \rrbracket = 0 &\iff \llbracket \Rlr(z_2,z_1), \Rll(z_2,z_3), \Rrl(z_1,z_3) \rrbracket = 0 \\
    & \iff \llbracket \Rll(z_2,z_3), \Rlr(z_2,z_1), \Rlr(z_3,z_1) \rrbracket = 0  
  \end{align*}
  and the same equivalences with $\Left$ and $\Right$ switched.
  Thus the $RRR$-equations involving $\Rrl$ follow from the ones involving only $\Rrr, \Rll$ and $\Rlr$.
\end{proof}

\begin{lemma}
  \label{lemma:at_most_three_colors}
  For fixed boundary conditions, the admissible states in a Yang--Baxter equation without vertices of type $\Rrl$ (i.e.\ consisting only of types $\Tr$, $\Tl$, $\Rrr$, $\Rll$ and $\Rlr$) involve at most three colors, and those must appear on the boundary and therefore be fixed.
\end{lemma}
\begin{proof}
  As a first step, we would like to show that in an admissible state of either of the systems in \eqref{eq:RTT-picture} and \eqref{eq:RRR-picture}, any color appearing on an internal edge must also appear on the boundary.
  From this it would follow that when we compute the partition functions, we only need to consider the colors that appear on the boundary.

  To show this, first consider the system on the left-hand side of the $RTT$-equation in \eqref{eq:RTT-picture} and suppose, for a contradiction, that we have an admissible state with an internal edge with a color that does not appear on the boundary.
  Because of color conservation, this is only possible if $R^X_Y = \Rlr$, $T_X = \Tr$ and $T_Y = \Tl$, but that is not a compatible choice of $R$- and $T$-vertices.
  Similarly, for the right-hand side of \eqref{eq:RTT-picture} we could argue that we must have $R^X_Y = \Rlr$, $T_X = \Tr$ and $T_Y = \Tl$, which again is not a compatible choice of $R$- and $T$-vertices.

  For the left-hand side of the $RRR$-equation in \eqref{eq:RRR-picture} we would be forced to have $R^X_Y = \Rlr$ and $R^Y_Z = \Rlr$ and hence $\Right = Y = \Left$, a contradiction.
  Similarly, for the right-hand side we would be forced to have $R^Y_Z = \Rlr$ and $R^X_Y = \Rlr$ and hence $\Left = Y = \Right$, a contradiction.

  Thus we have showed that any color appearing on an internal edge of an admissible state must also appear on the boundary, and it only remains to be shown that an admissible state can have at most three colors on the boundary.
  So consider some colored edge on the boundary of an admissible state.
  By color conservation, it is part of a path of this color through the system and since one end of the path is on the boundary, the other one must also be on the boundary.
  Hence any color on the boundary must appear at least twice on the boundary, which means that there can be at most three distinct colors on the six boundary edges.

  So the boundary conditions can involve at most three distinct colors which are fixed for the whole system, and any admissible state can only involve those three colors.
\end{proof}

\begin{proposition}
  \label{prop:twisting}
  If \cref{thm:YBE} holds for a particular choice of $\Phi \neq 0$ and $\mathfrak{X}_{i,j}$ satisfying~\eqref{eq:X-condition}, then it holds for all such $\Phi$ and $\mathfrak{X}_{i,j}$.
\end{proposition}

\begin{proof}
  The proof generalizes those of Propositions~4.12 and~4.14 of \cite{BBBG:duality} where all involved vertices are right-moving. 
  We will show that for fixed boundary edges, all state configurations appearing on either side of the Yang--Baxter equations (either $RTT$ or $RRR$) have the same factors of $\Phi$ and $\mathfrak{X}_{i, j}$. 

  For these manipulations it is convenient to interpret the unfused edge spins as maps $\mathcal{P} \to \{ 0, 1 \}$, where $\mathcal{P}=\mathcal{P}_m$ is the palette, and lift these to the ring $\mathcal{P}^{\ast}$ of maps $\mathcal{P} \to \mathbb{Z}$. 
  For $f \in \mathcal{P}^\ast$ we define the $\mathcal{P}$-average $\langle f \rangle = \sum_{c \in \mathcal{P}} f (c) \in \mathbb{Z}$ which is linear in $f$.

  We also let $\phi : \mathcal{P} \times \mathcal{P} \to \mathbb{C}^{\times}$ such that \begin{equation} 
    \label{eq:phi-condition} 
    \phi (c, d) \phi (d, c) = 1 \text{ and }  \phi (c, c) = 1  \text{ for all } c, d \in \mathcal{P}
  \end{equation}
  and introduce the antisymmetric bilinear form $\langle \cdot, \cdot \rangle : \mathcal{P}^{\ast} \times \mathcal{P}^{\ast} \to \mathbb{C}$ defined by
  \begin{equation*}
    \langle f, g \rangle = \sum_{c, d \in \mathcal{P}} f (c) g (d) \log \phi (c, d) 
  \end{equation*}
  for $f, g \in \mathcal{P}^{\ast}$.
  Here we chosen have a branch cut for the logarithm that avoids the finite image of $\phi$ and our argument is independent of this choice since we will be considering expressions on the form $\exp (\langle f, g \rangle)$.
  Note that $\phi (c, c) = 1$ implies that terms with $c = d$ are excluded. Note also that if we, for $a \in \mathcal{P}$, denote the indicator function $x \mapsto \bigl\{ \begin{smallmatrix*}[l] 1 & \text{if } x = a\\ 0 & \text{otherwise} \end{smallmatrix*}$ by $1_a \in \mathcal{P}^{\ast}$ then $\exp (\langle 1_a, 1_b \rangle) = \phi (a, b)$.

  According to the statement, we assume that the Yang--Baxter equations hold for $\Phi' \neq 0$ and $\mathfrak{X}'_{i, j}$ such that $\mathfrak{X}'_{i, j} \mathfrak{X}'_{j, i} = q^2$ and $\mathfrak{X}'_{i, i} = - q^2$ which, in particular, means that $\mathfrak{X}'_{i, j} \neq 0$.
  Any other $\mathfrak{X}_{i, j}$ satisfying the same relations can be expressed as $\mathfrak{X}'_{i, j} \cdot \phi (c_i, c_j)$ with $\phi$ satisfying~\eqref{eq:phi-condition}.
  This means that it is enough to show that the states have the same factors of $\Phi$ and $\phi$.
  The way these factors appear in the vertex weights of \cref{tab:T_right_weights,tab:T_left_weights,tab:R-vertices,tab:R-vertices-mixed} is described in \cref{tab:twists}.
  
\begin{table}[tb]
  \tikzstyle{default}=[baseline={(0,-0.5ex)}, scale=0.3]
  \caption{\label{tab:twists} The factors of $\Phi$ and $\phi$ appearing in the weights of \cref{tab:T_right_weights,tab:T_left_weights,tab:R-vertices,tab:R-vertices-mixed}.
  The arrows signify our choices of inputs and outputs such that colors are conserved.}
  \vspace{-0.5em}
  \begin{equation*}
    \begin{tblr}{colspec = {|c|ccccc|}, cells = {mode=math}, row{2,3,4,5,6,7} = {ht=2em}}
    \hline
    \SetCell[c=2]{c} \textbf{Vertex type} & & \textbf{Input} \to \textbf{Output} & \boldsymbol{\phi}\textbf{-factor} &  &
    \boldsymbol{\Phi}\textbf{-factor}\\
    \hline 
    \SetCell[r=2]{c}
    \begin{tikzpicture}[default]
      \draw (-1,0) node[left] {$a$} -- (1,0) node[right] {$c$};
      \draw (0,-1) node[below] {$d$} -- (0,1) node[above] {$b$};
      \node[dot,minimum size=3pt] at (0,0){};
    \end{tikzpicture}
    &
    \Tr
    & \searrow & \exp \left( \frac{1}{4} \langle a+c, b+d \rangle \right) &  & \Phi^{\langle d \rangle}\\
    \hline 
    &
    \Tl
    & \swarrow & \exp \left( -\frac{1}{4} (\langle a+c, b+d \rangle \right) &  & \Phi^{\langle d \rangle}\\
    \hline
    \SetCell[r=4]{c}
    \begin{tikzpicture}[default]
      \draw (-1,-1) node[left] {$a$} -- (1,1) node[right] {$c$};
      \draw (-1,1) node[left] {$b$} -- (1,-1) node[right] {$d$};
      \node[dot,minimum size=3pt] at (0,0){};
    \end{tikzpicture}
    &
    \Rrr
    & \rightarrow & \exp \left( \frac{1}{4} (\langle a+c, b+d \rangle \right) &  & \Phi^{\frac{\langle -a-c+b+d \rangle}{2}} \\
    \hline
    &
    \Rll
    & \leftarrow & \exp \left( \frac{1}{4} (\langle a+c, b+d \rangle \right) &  & \Phi^{\frac{\langle a+c-b-d \rangle}{2}} \\
    \hline
    &
    \Rlr
    & \downarrow & \exp \left( -\frac{1}{4} (\langle a+c, b+d \rangle \right) &  & \Phi^{\frac{\langle a+c+b+d \rangle}{2}} \\
    \hline
    &
    \Rrl
    & \uparrow & \exp \left( -\frac{1}{4} (\langle a+c, b+d \rangle \right) &  & \Phi^{\frac{\langle -a-c-b-d \rangle}{2}} \\
    \hline
  \end{tblr}
\end{equation*}
\end{table}

We start by considering the $RRR$-equation involving only $\Rrr$-vertices and label the spins of the edges as in \eqref{eq:RRR-picture}.
We then have the following conservation equations:
\begin{equation}
  \left\{\begin{array}{l}
    a + b = i + j\\
    j + k = e + f\\
    i + c = d + k
  \end{array}\right. \hspace{4em} \left\{\begin{array}{l}
    b + c = n + l\\
    a + n = f + m\\
    m + l = d + e
  \end{array}\right. \label{eq:conservations}
\end{equation}
Note that both systems of equations imply the same global conservation for boundary edges
\begin{equation}
  a + b + c = d + e + f.
\end{equation}

By multiplying the $\phi$-contributions from all vertices in the Yang--Baxter equation using~\cref{tab:twists} we get that the total contributions to the left- and right-hand sides are $\exp \left( \frac{1}{2} \operatorname{LHS}_{\phi} \right)$ and $\exp \left( \frac{1}{2} \operatorname{RHS}_{\phi} \right)$, respectively where
\begin{equation}
  \label{eq:phi-RRR}
  \begin{split}
  \operatorname{LHS}_{\phi} & = \langle a, b \rangle + \langle i, j \rangle + \langle i, c \rangle + \langle d, k \rangle + \langle j, k \rangle + \langle e, f \rangle \\
  \operatorname{RHS}_{\phi} & = \langle a, n \rangle + \langle m, f \rangle + \langle b, c \rangle + \langle l, n \rangle + \langle m, l \rangle + \langle d, e \rangle.
\end{split}
\end{equation}

It is easy to verify by hand as in~\cite{BBBG:duality} that, modulo the conservation equations \eqref{eq:conservations}, these are equal (and thus only depend on the boundary edges). 

For the $\Phi$-contributions we get the powers
\begin{equation}
  \label{eq:Phi-RRR}
  \begin{split}
    \operatorname{LHS}_{\Phi} &= \tfrac{1}{2}\left(\langle -a-i+b+j \rangle + \langle -i-d+c+k \rangle + \langle -j-e+k+f\rangle\right) \\
    \operatorname{RHS}_{\Phi} &= \tfrac{1}{2}\left(\langle -b-l+c+n \rangle + \langle -a-m+n+f \rangle + \langle -m-d+l+e \rangle\right)
  \end{split}
\end{equation}
which, after using the conservation equations \eqref{eq:conservations} are both equal to $\langle -a+c-d+f \rangle$.\footnote{As mentioned earlier the $\phi$ and $\Phi$ invariance for the Yang-Baxter equations only involving the right-moving model was covered in \cite{BBBG:duality} but this case~\eqref{eq:Phi-RRR} had been left out in the proof.}

We will now argue that this computation immediately gives us all the other cases of $RRR$-equations.
First note that the expressions for the $\phi$ and $\Phi$-factors for $\Rll$, $\Rlr$ and $\Rrl$ are the same as for~$\Rrr$, except that the spins for the left-moving edges have been negated.
This is also true true for the conservation equations, which can be expressed as $\pm a \pm b = \pm c \pm d$, where the signs are positive for the spins of right-moving edges and negative for the spins of left-moving edges.

To prove the right-moving $RRR$-equations we only used the conservation equations, and the properties of the antisymmetric form and the $\mathcal{P}$-average. 
This means that since we negate the spins for left-moving edges in both the weights and the conservation equations, the other $RRR$-equations are just reparametrizations of the $\Rrr$ case, and therefore also hold.

Next we consider the $RTT$-equation involving only right-moving vertices, where we label the spins of the edges as in \eqref{eq:RTT-picture}. We then have the same conservation equations \eqref{eq:conservations} as before.

Furthermore, we also obtain exactly the same $\phi$-factors on each side as we did for the $RRR$-equation above\footnote{This is of course easy to compute, but can also be seen by rotating the $T$-vertices so that they become $R$-vertices.}, and since the conservation equations are the same, the $\phi$-factors are equal by the same calculation.

Similarly, the powers of $\Phi$ for the left- and right-hand sides are
\begin{equation}
  \begin{split}
    \operatorname{LHS}_{\Phi} & = \frac{1}{2}\langle -a-i+b+j \rangle + \langle k \rangle + \langle f \rangle = \langle b - i + k + f \rangle \\ 
    \operatorname{RHS}_{\Phi} & = \frac{1}{2}\langle -m-d+l+e\rangle + \langle n \rangle + \langle f \rangle = \langle l - d + n + f \rangle 
\end{split}
\end{equation}
which are both equal to $\langle b+c-d+f \rangle$ by the conservation equations.

As before, we will now argue that these computations also prove the other $RTT$-equations.
Again the expressions for the $\phi$ and $\Phi$-factors for $\Tl$ are the same as for $\Tr$, except that the spins for the left-moving edges have been negated, and the conservation equations for $\Tl$ and $\Tr$ are related in the same way.
Hence the computations for the remaining $RTT$-equations are simply reparametrizations of the one above.
\end{proof}

We note that in the proof, the only assumptions on the weights that are used is the color conservation and the appearance of the $\phi$- and $\Phi$-factors as in \cref{tab:twists}.

\begin{lemma}
  \cref{thm:YBE} holds for a particular the Iwahori specialization of $\Phi$ and $\mathfrak{X}_{i,j}$.
\end{lemma}

\begin{proof}
  We use \cref{cor:R_RightLeft} to be able to disregard all Yang--Baxter equations involving $\Rrl$.
  \cref{lemma:at_most_three_colors} then tells us that each case of the remaining Yang--Baxter equations involves at most three colors, and that those must appear on the boundary and therefore be fixed.

  For the Iwahori specialization all the remaining vertex weights only depend on the relative order of the colors and not on their absolute index in the palette (like the weights for $\Rrl$ do).
  This means that it is enough to show that each Yang--Baxter equation is satisfied in the case where the palette size is $m=3$ independent of the number of rows $r$.

  We have now reduced the problem to checking a fixed, finite number of equalities of elements in the fraction field $\mathbb{C}(z_1,z_2,z_3, v)$ which we have verified with a symbolic computer algebra system (\texttt{SageMath}).
\end{proof}

Together with \cref{prop:twisting} this finishes the proof of \cref{thm:YBE}.

\begin{remark}
\label{rem:quasi-solvable}
The solvability of these models addresses an important issue raised in~\cite{Buciumas-Scrimshaw}.
In that paper they construct a lattice model describing Demazure atoms and characters for Cartan types $B$ and $C$ consisting of alternating left- and right-moving rows called $\Delta$ and $\Gamma$ respectively as well as a new type of vertex called a \emph{cap} connecting two rows on the right boundary.
The $\Delta$ and $\Gamma$ vertex configurations and weights presented there agree with the crystal limits of the left- and right-moving Iwahori models in this paper. 
The right-moving $\Gamma$ model was originally constructed in~\cite{BBBG:demazure} and was later generalized to the Iwahori ice model in~\cite{BBBG:Iwahori}, as was mentioned earlier.

Buciumas and Scrimshaw showed that the model in~\cite{Buciumas-Scrimshaw} is \emph{quasi-solvable}, meaning that it satisfies Yang--Baxter equations using three (mixed) $R$-matrices.
However, they were unable to find a suitable fourth $R$-matrix required for full solvability.
One reason for this that they mentioned was that there are internal loops of paths with colors not fixed by the boundary conditions in the Yang--Baxter equations (which are therefore summed over), seemingly resulting in an $m$-dependence on only one side of the equation.

For the families in this paper we have found all four $R$-matrices and show that they satisfy the corresponding Yang--Baxter equations.
We show that the issue of color loops for the general family is resolved by the fact that the $m$-dependent color loop sums are telescoping sums reducing to an expression that matches an $m$-dependence in the fourth $R$-matrix on the other side of the equation.
However, when taking the crystal limit $v\to 0$ in the Iwahori specialization of our families, the $R$-matrix corresponding to the missing $R$-matrix of~\cite{Buciumas-Scrimshaw} becomes very degenerate with most vertex weights being zero.
\end{remark}

\subsection{Demazure-type operators}
\label{sec:demazure-operators}
In~\cite[Section~5]{BBBG:duality} the Yang--Baxter equation was used to prove that the partition functions $Z(\mathfrak{S}^\Right_{\mu,\sigma})(\mathbf{z})$ with $\sigma_i \neq \sigma_{i+1}$ satisfy the recursion relation 
\begin{equation} 
  \label{eq:demazure-recursion}
  Z^\Right_{\mu, s_i \sigma}(\mathbf{z}) = \frac{\mathfrak{X}^\text{Iwahori}_{\sigma_i, \sigma_{i+1}}}{\mathfrak{X}_{\sigma_i, \sigma_{i+1}}} \cdot \begin{cases} T_i \cdot Z^\Right_{\mu,\sigma}(\mathbf{z}) & \textrm{if } \sigma_i > \sigma_{i+1}, \\ T_i^{-1} \cdot Z^\Right_{\mu,\sigma}(\mathbf{z}) & \textrm{if } \sigma_i < \sigma_{i+1}, \end{cases} 
\end{equation}
where $\mathfrak{X}^\text{Iwahori}$ is the specialization of $\mathfrak{X}$ to the Iwahori case \eqref{eq:Iwahori-specialization} and where $T_i$ is a Demazure-type operator which acts on functions $f : \mathbb{C}^r \to \mathbb{C}$ by
\begin{equation}
  T_i\cdot f(\mathbf{z}) = \frac{z_i - q^{-2}z_{i+1}}{z_{i+1} - z_i} f(s_i \mathbf{z}) + \frac{(q^{-2}-1)z_{i+1}}{z_{i+1} - z_i} f(\mathbf{z}).
\end{equation}
As shown in \cite{BBBG:Iwahori, BBL} these Demazure operators braid and satisfy the quadratic relation $T_i^2 = (q^{-2} - 1)T_i + q^{-2}$.

Using the same argument as in \cite{BBBG:duality}, which is in turn a generalization of the argument in \cite[Proposition~7.1]{BBBG:Iwahori}, it is not difficult to obtain similar recursion relations for the left-moving models.
These match the ones above except with an exchange of $\sigma$ and $\mathbf{z}$ to their reflected counterparts in agreement with the left-right duality of \cref{mainthm:duality}.

We note that by applying \eqref{eq:demazure-recursion} recursively we can relate partition functions in the same $W$-orbit of boundary colors.

In particular, all partition functions with distinct boundary colors can be computed as follows.
Recall that $\mu$ modulo $m$, the size of the palette, determines the top boundary colors.
Fix such a $\mu$ with distinct top boundary colors $\sigma_\mathrm{dist}$ read from left to right. 
Then the left- and right-systems $\mathfrak{S}_{\mu, w_0 \sigma_\mathrm{dist}}^\Left$ and $\mathfrak{S}_{\mu,\sigma_\mathrm{dist}}^\Right$ only contain one admissible state each, where all paths cross as shown in the example in \cref{fig:base-case}. 
The partition functions are easily computed as
\begin{equation}
  \label{eq:base-case}
  Z_{\mu,\sigma_\mathrm{dist}}^\Right(\mathbf{z}) = \mathbf{z}^N Z_{\mu,w_0 \sigma_\mathrm{dist}}^{\Left,N}(w_0 \mathbf{z}) = \mathbf{z}^\mu \prod_{i<j} \Phi \mathfrak{X}_{(\sigma_\mathrm{dist})_i, (\sigma_\mathrm{dist})_j}.
\end{equation}

More generally, for arbitrary boundary colors on the horizontal edges the partition function is only non-zero if these colors are a permutation $w \sigma_\mathrm{dist}$ of the top boundary colors $\sigma_\mathrm{dist}$. 
Let $s_{i_1} s_{i_2} \ldots s_{i_\ell}$ be a reduced word for $w$ and let $T_w = T_{i_1} T_{i_2} \ldots T_{i_\ell}$, which is well-defined since the Demazure operators braid.
Then from repeated use of \eqref{eq:demazure-recursion}, and the equivalent version for the left-moving case, we obtain that
\begin{equation}
  Z_{\mu,w\sigma_\mathrm{dist}}^\Right(\mathbf{z}) = \mathbf{z}^N Z_{\mu,w_0 w \sigma_\mathrm{dist}}^{\Left,N}(w_0 \mathbf{z}) = C_{w,\sigma_\mathrm{dist}} T_w \mathbf{z}^\mu,
\end{equation}
where $C_{w, \sigma_\mathrm{dist}}$ is a constant depending on $\Phi$ and $\mathfrak{X}$ that can be read off from \eqref{eq:demazure-recursion} and \eqref{eq:base-case}.

\begin{figure}[htbp]
  \newcommand{\scl}{0.8} 
    \centering
    \begin{tikzpicture}[scale=\scl]
      \begin{scope}
        \foreach \x in {0,...,4} {
          \draw (\x, -0.5) -- (\x, 2.5);
        }
        \foreach \y/\i in {0/3,1/2,2/1} {
          \draw (-0.5, \y) -- (4.5, \y);
        }

        \newcommand{\lineWidth}{1}
        \pgfmathsetmacro{\pathShift}{\lineWidth/(10*\scl)}
        \draw[blue, line width=\lineWidth mm] (4,2.5) -- (4,2) -- (-0.5,2);
        \draw[green, line width=\lineWidth mm] (2,2.5) -- (2,1) -- (-0.5,1);
        \draw[red, line width=\lineWidth mm] (0,2.5) -- (0,0) -- (-0.5,0);
      \end{scope}

      \begin{scope}[xshift=7cm]
        \foreach \x in {0,...,4} {
          \draw (\x, -0.5) -- (\x, 2.5);
        }
        \foreach \y/\i in {0/3,1/2,2/1} {
          \draw (-0.5, \y) -- (4.5, \y);
        }

        \newcommand{\lineWidth}{1}
        \pgfmathsetmacro{\pathShift}{\lineWidth/(10*\scl)}
        \draw[blue, line width=\lineWidth mm] (4,2.5) -- (4,0) -- (4.5,0);
        \draw[green, line width=\lineWidth mm] (2,2.5) -- (2,1) -- (4.5,1);
        \draw[red, line width=\lineWidth mm] (0,2.5) -- (0,2) -- (4.5,2);
      \end{scope}
    \end{tikzpicture}
  \caption{Examples of the unique states in the base case systems for the left- and right-moving models.}
  \label{fig:base-case}
\end{figure}

This proves the left-right duality when the colors are distinct, but we will use a different, more general argument (which works for any $\sigma$) in \cref{sec:left-right_duality}.
The proof in \cref{sec:left-right_duality} also mirrors the individual steps for the bijection in the crystal limit in terms of the Schützenberger involution.

In the case of the Iwahori specialization, it was shown in~\cite[Proposition~8.4]{BBBG:Iwahori} that any right-moving partition function can be expressed as a sum over partition functions with distinct colors mirroring the way parahoric Whittaker functions can be expressed in terms of Iwahori Whittaker functions.
Since the left-moving Iwahori lattice model satisfies similar Properties A and B that Proposition~8.4 of \cite{BBBG:Iwahori} builds on, a similar parahoric property for the left-moving Iwahori model should follow proving the full left-right duality for the Iwahori case.
Unfortunately, not all members of the family satisfies the Properties A and B in \cite{BBBG:Iwahori}.

\subsection{Remarks about quantum groups}
\label{sec:quantum-groups}
Now that we have proved that our two families form a system of solvable lattice models, let us make a few remarks about quantum groups.
The theory of quantum groups originated from the effort of finding solutions to Yang--Baxter equations with two early pivotal works being~\cite{Drinfeld:ICM, Jimbo:1985}.
For the purpose of this paper we may consider a quantum group as a certain $q$-deformation of a universal enveloping algebra $U(\mathfrak{g})$ of a Lie algebra $\mathfrak{g}$, which we will denote by $U_q(\mathfrak{g})$. 
There is a special element $\mathcal{R}$ in a completion of the tensor product $U_q(\mathfrak{g}) \otimes U_q(\mathfrak{g})$ called the \emph{universal $R$-matrix}. 
Given three modules $U$, $V$ and $W$ of $U_q(\mathfrak{g})$ one can obtain a solution to the Yang--Baxter equation~\eqref{eq:YBE} from the action of $\mathcal{R}$ on tensor products of these modules \cite{Drinfeld:ICM, Klimyk-Schmudgen}.
For results in the other direction, where an $R$-matrix solving the Yang--Baxter equation $\llbracket R(z_1,z_2), R(z_1,z_3), R(z_2,z_3)\rrbracket = 0$ gives rise to a quantum group presentation with this $R$-matrix, see~\cite{FRT,ReshetikhinSTS}.

Both the weights for the $R$-vertices and the $T$-vertices could in theory be obtained from the universal R-matrix in this way with different choices of the modules.
However, as we shall discuss further below, the module for the vertical edges (which only appear for the $T$-vertices) is not known in general.

Fixing complex numbers for $q$, $\Phi$ and $\mathfrak{X}_{i,j}$, the weights for the $\Rrr$-vertices in this paper come from a (combinatorial) Drinfeld twist of the $R$-matrix of a standard evaluation module of the affine supersymmetric quantum group $U_q(\widehat{\mathfrak{gl}}(m|1))$ where $m$ is the number of colors \cite{BBBG:duality}.
For our purposes we may view the evaluation module as $\mathbb{C}^{m|1} \otimes \mathbb{C}[z, z^{-1}]$ where $\mathbb{C}^{m|1}$ is the standard module of $\mathfrak{gl}(m|1)$.
The name \emph{evaluation module} comes from the fact that, in the standard setting, $z$ is a fixed complex number which the loop parameter of the affine Lie algebra is evaluated to, but we prefer to work with a formal parameter $z$ instead.
For more information on evaluation modules in the representation theory of quantum groups, see for example~\cite[Chapter~12]{Chari-Pressley}.

Drinfeld twists conjugates the universal $R$-matrix, preserving the Yang--Baxter equation but resulting in slightly different $R$-matrix solutions.
The effects of such a Drinfeld twist on $R$-vertex weights was derived in \cite[Section~4]{BBBF}.
They showed that when two paths of colors $c$ and $d$ cross each other the weight gets multiplied by a factor similar to $\phi(c,d)$ in \cref{prop:twisting}.
The $\Phi$-factors in the $\Rrr$ and $\Rll$ weights similarly appear when a colored path crosses unoccupied edges.
By a \emph{combinatorial} Drinfeld twist we mean a transformation of the weights of this type for crossing paths preserving the Yang--Baxter equations.
The difference is that we have not confirmed that it amounts to a valid conjugation of the universal $R$-matrix which is beyond the scope of this paper.
It is also not yet possible to verify for the $T$-vertices since we lack a quantum module for the vertical edges.

These twists results in a factor $\phi(c,d)$ for the right-moving model and a factor $\phi(d,c)$ for the left-moving model with the same vertex configuration such that $\phi(c,d) \phi(d,c) = 1$.
This same feature for crossing colors also appeared for the left- and right-moving colored models in \cite{Perk-Schultz}.

The right-moving Iwahori lattice model is then in particular also associated to $U_q(\widehat{\mathfrak{gl}}(m|1))$ by a (combinatorial) Drinfeld twist.
In fact, it appears as the $n=1$ case of the $n$-fold metaplectic Iwahori lattice model in~\cite{BBBG:metahori} where the quantum group is $U_q(\widehat{\mathfrak{gl}}(m|n))$.
In~\cite[Proposition~2.11]{BBBG:metahori} it is shown as $U_{1/\sqrt{v}}(\widehat{\mathfrak{gl}}(r|n))$ where the number of colors $m$ have been taken to be the rank $r$ and where $v = 1/q^2$ in the Iwahori specialization~\eqref{eq:Iwahori-specialization}.

If we take the same metaplectic Iwahori model and instead specialize to $m=1$ colors we recover the left-moving (spherical) metaplectic $\Gamma$ model with quantum group $U_{1/\sqrt{v}}(\widehat{\mathfrak{gl}}(1|n))$ and the rest of the left-moving family are (combinatorial) Drinfeld twists of this.
Note that in the metaplectic specialization \eqref{eq:metaplectic-specialization} we have $v = q^2$ and that $U_{1/q}(\widehat{\mathfrak{gl}}(1|n))$ describes $n$ colors (or rather supercolors).

It is not clear though what the quantum group interpretation is for mixed models with both left- and right-moving rows, and for the mixed $R$-vertices $\Rlr$ and $\Rrl$.
See \cite[Section~8]{BBG:vertical} for a discussion of a potential direction of exploration based on the work therein on the case $m=n=1$.
See also \cite[Section~3.4]{Buciumas-Scrimshaw} for a discussion regarding this in the crystal limit. 

As mentioned above, it is also not known what the quantum module is for vertical edges in general, but \cite{BBG:vertical} gives some results for the $U_q(\widehat{\mathfrak{gl}}(1|1))$ case.
The $T$-vertices include both horizontal and vertical edges, which is why we use a combinatorial proof in \cref{prop:twisting} for the $\Phi$-factors.
Another reason is that much less is known about the quantum group origins of the $\Phi$-factors, although the weights for $\Rrr$ and $\Rll$ suggest they come from crossing a color path with an ``unoccupied path'' as mentioned above.
However, this behavior is not mirrored for the $T$-vertices or the mixed $R$-vertices.

\subsection{Solvability of the crystal model}
\label{sec:crystal-YBE}
Recall that the fused crystal $T$-vertices of \cref{tab:T_fused_gamma_Iwahori_crystal_weights,tab:T_fused_delta_Iwahori_crystal_weights} were obtained from the Iwahori specialization of the left- and right-moving families of lattice models by taking the crystal limit $v \to 0$ and restricting to the vertex configurations which have at most one color on the vertical edges (by \cref{lemma:crystal_models_at_most_one_color} these are the only configurations that can appear with the type of boundary conditions we are considering).
Recall also that the left-moving weights have been multiplied by $z$ to remove the partition function's dependence on the number of columns.

For the Iwahori specialization we call the left- and right-moving models $\Delta$ and $\Gamma$, respectively, as seen in \cref{fig:specializations}.
These two sets of vertex configurations and weights agree with the $\Gamma$ and $\Delta$ $T$-vertices of the lattice model in~\cite{Buciumas-Scrimshaw} which were used to compute Demazure characters and atoms for $\operatorname{Sp}_{2n}$ and $\operatorname{SO}_{2n+1}$ and is a color refinement of the lattice model in \cite{Ivanov}. Buciumas and Scrimshaw also showed in~\cite{Buciumas-Scrimshaw} that their model is \emph{quasi-solvable} meaning that they found $R$-vertices of types $R^\Delta_\Delta$, $R^\Gamma_\Gamma$, and $R^\Delta_\Gamma$ that satisfy the $RTT$-equations.
However, they were not able to find a suitable $R^\Gamma_\Delta$ matrix.
One obstruction they described was that they obtained color loops in their Yang--Baxter equations resulting in a dependence on the number of colors $m$ on one side of the equation while not on the other. 

We have shown above that the left- and right-moving non-crystal families of lattice models introduced in this paper are solvable with all four types of $R$-vertices and we will in this section investigates what happens with these when we take the crystal limit $v \to 0$.

Recall that the Yang--Baxter equations are still satisfied if we multiply all the weights of a type of $R$-vertex with a power of $v$.
For each type of $R$-vertex there is a unique such $v$-factor such that the weights are finite in the limit $v\to0$ and not all zero.
We present these limits in \cref{tab:R-vertices_crystal,tab:R-vertices-mixed_crystal}.
The fused versions (i.e.\ $k=1$) of the $R^\Delta_\Delta$, $R^\Gamma_\Gamma$ and $R^\Delta_\Gamma$ agree with those of \cite{Buciumas-Scrimshaw}.
For $R^\Gamma_\Delta$ we get a very degenerate case with only one non-zero weight.
Indeed, this is not a very interesting solution to the problem of finding $R^\Gamma_\Delta$ posed in~\cite{Buciumas-Scrimshaw} and it may not fulfill their other requirements.
However, it illuminates (at least in the non-crystal version) how the issue of color loops is resolved by telescoping sums as explained in the proof of \cref{prop:inverse}.

\begin{theorem}
  The crystal $R$-vertices in \cref{tab:R-vertices_crystal,tab:R-vertices-mixed_crystal} and the restricted crystal $T$-vertices in \cref{tab:T_fused_delta_Iwahori_crystal_weights,tab:T_fused_gamma_Iwahori_crystal_weights} satisfy the Yang--Baxter $RTT$-equations in~\eqref{eq:RTT-picture} and the $RRR$-equations in~\eqref{eq:RRR-picture} for all combinations of row types $X,Y,Z \in \{\Gamma,\Delta\}$.
\end{theorem}

\begin{proof}
  All the weights of $\Td, \Tg, \Rdd, \Rgg, \Rdg$ and $\Rgd$ for the crystal models can be obtained by taking the weights of $\Tl, \Tr, \Rll, \Rrr, \Rlr$ and $\Rrl$, respectively, specializing to the Iwahori case~\eqref{eq:Iwahori-specialization}, multiplying by $1$, $1$, $-v$, $-v$, $1$ and $v^{m+2}$, setting $v=0$, and finally restricting $\Td$ and $\Tg$ to those vertex configurations which have at most one color on the vertical edges.

  By \cref{thm:YBE} all compatible configurations of $\Tl, \Tr, \Rll, \Rrr, \Rlr$ and $\Rrl$ satisfy the Yang--Baxter equation, and it is therefore sufficient to show that this is still true after each of the steps above.
  Firstly, the Yang--Baxter equations hold for every choice of $\Phi$ and $\mathfrak{X}_{i,j}$, in particular it holds for the Iwahori specialization.
    Secondly, multiplying all weights of a particular type (say $\Td$) by the same factor also does not affect the Yang--Baxter equations since both the left- and right-hand side are multiplied by the same constant.
    Thirdly, since the equations hold for every $v$, they must in particular hold for $v=0$.

    Next, we need to show that we can restrict $\Td$ and $\Tg$ to the vertex configurations which has at most one color on the vertical edges.
    This will only affect the $RTT$-equation~\eqref{eq:RTT-picture} as that is the only one where the $T$-vertices appear.
    So suppose that we have some boundary conditions for~\eqref{eq:RTT-picture} with at most one color on each edge.
    Then \cref{lemma:crystal_models_at_most_one_color} applied to the lower $T$-vertex on either side implies that the only internal vertical edge on each side can carry at most one color, and therefore the Yang--Baxter equation holds with this restricted set of vertex configurations.

    Finally, recall that the weights in \cref{tab:T_fused_delta_Iwahori_crystal_weights} are actually given by $zT_\Delta$, i.e.\ they have been scaled by a factor of $z$.
    Again, this does not affect the validity of the Yang--Baxter equations as both sides of each equation are scaled by the same factor.
\end{proof}

\begin{table}[htpb]
  \centering
  \caption{Configurations and Boltzmann weights for $\Rdd$ and $\Rgg$ for the fused crystal models, where $i \neq j$ in the multicolored vertex configurations.
    The $\Rdd$ and $\Rgg$ weights can be obtained from the $\Rll$ and $\Rrr$ weights in \cref{tab:R-vertices}, respectively, by first specializing to the Iwahori case \eqref{eq:Iwahori-specialization}, then multiplying by $-v$ and finally setting $v=0$.}
  \label{tab:R-vertices_crystal}
  \vspace{-0.5em}
  \begin{equation*}
    \begin{array}{|c|c|c|c|c|c|}
      \hline

      \multicolumn{6}{|c|}{\boldsymbol{\Rdd} \textbf{ and } \boldsymbol{\Rgg} \textbf{ (crystal)} \vphantom{\displaystyle\int}}

      \\\hline\hline

      \begin{tikzpicture}[baseline, scale=0.75]
        \node[anchor=south] at (-1,-1) {$z_1$};
        \node[anchor=north] at (-1,1) {$z_2$};
      \end{tikzpicture}
      &
      \begin{tikzpicture}[baseline, scale=0.75]
        \draw (-1,-1) -- (1,1);
        \draw (-1,1) -- (1,-1);
        \node[label=above:$c_k$] (0,0) {};
      \end{tikzpicture}  
      &
      \begin{tikzpicture}[baseline, scale=0.75]
        \draw (-1,1) -- (1,-1);
        \draw[ultra thick, red] (-1,-1) -- (1,1)  node[label=below:$c_i$]{};
        \node[label=above:$c_k$] (0,0) {};
      \end{tikzpicture} 
      &
      \begin{tikzpicture}[baseline, scale=0.75]
        \draw (-1,-1) -- (1,1);
        \draw[ultra thick, red] (-1,1) -- (1,-1) node[label=above:$c_i$]{};
        \node[label=above:$c_k$] (0,0) {};
      \end{tikzpicture}
      &
      \begin{tikzpicture}[baseline, scale=0.75]
        \draw[ultra thick, red] (-1,-1) -- (1,1) node[label=below:$c_i$]{};
        \draw[ultra thick, red] (-1,1) -- (1,-1) node[label=above:$c_i$]{};
        \node[label=above:$c_k$] (0,0) {};
      \end{tikzpicture}
      &
      \begin{tikzpicture}[baseline, scale=0.75]
        \draw[ultra thick, blue] (-1,1) -- (1,-1) node[label=above:$c_j$]{};
        \draw[ultra thick, red] (-1,-1) -- (1,1) node[label=below:$c_i$]{};
        \node[label=above:$c_k$] (0,0) {};
      \end{tikzpicture}

      \\\hline
      
      \Rdd\vphantom{\Bigl)}
      & z_2 
      & 0
      & z_1 - z_2
      & z_1
      & z_1-z_2 \text{ if } i > j

      \\\hline
      
      \Rgg\vphantom{\Bigl)}
      & z_2
      & z_1 - z_2
      & 0
      & z_1
      & z_1-z_2 \text{ if } i > j

      \\\hline\hline

      \begin{tikzpicture}[baseline, scale=0.75]
        \node[anchor=south] at (-1,-1) {$z_1$};
        \node[anchor=north] at (-1,1) {$z_2$};
      \end{tikzpicture}
      &
      \begin{tikzpicture}[baseline, scale=0.75]
        \draw (-1,-1) -- (0,0) -- (1,-1);
        \draw[ultra thick, red] (-1,1) -- (0,0) --  (1,1) node[label=below:$c_i$]{};
        \node[label=above:$c_k$] (0,0) {};
      \end{tikzpicture}
      &
      \begin{tikzpicture}[baseline, scale=0.75]
        \draw (-1,1) -- (0,0) --  (1,1);
        \draw[ultra thick, red] (-1,-1) -- (0,0) -- (1,-1) node[label=above:$c_i$]{};
        \node[label=above:$c_k$] (0,0) {};
      \end{tikzpicture}  
      & \multicolumn{3}{c|}{
      \begin{tikzpicture}[baseline, scale=0.75]
        \draw[ultra thick, blue] (-1,1) -- (0,0) --  (1,1) node[label=below:$c_j$]{};
        \draw[ultra thick, red] (-1,-1) -- (0,0) -- (1,-1) node[label=above:$c_i$]{};
        \node[label=above:$c_k$] (0,0) {};
      \end{tikzpicture} 
      }

      \\\hline

      \Rdd\vphantom{\Bigl)}
      & z_1
      & z_2
      & \multicolumn{3}{c|}{
        \left\{\begin{smallmatrix*}[l]
          z_2 & \text{if $i < j < k$ or $j < k \leq i$ or $k \leq i < j$}\\
          z_1 & \text{if $j < i < k$ or $i < k \leq j$ or $k \leq j < i$}
        \end{smallmatrix*}\right.
      }

      \\\hline

      \Rgg\vphantom{\Bigl)}
      & z_2
      & z_1
      & \multicolumn{3}{c|}{
        \left\{\begin{smallmatrix*}[l]
          z_1 & \text{if $i < j < k$ or $j < k \leq i$ or $k \leq i < j$}\\
          z_2 & \text{if $j < i < k$ or $i < k \leq j$ or $k \leq j < i$}
        \end{smallmatrix*}\right.
      }

      \\\hline
    \end{array}
  \end{equation*}  
\end{table}

\begin{table}[htpb]
  \centering
  \caption{Configurations and Boltzmann weights for the mixed $\Rdg$ and $\Rgd$ for the fused crystal models, where $i \neq j$ in the multicolored vertex configurations.
    The $\Rdg$ and $\Rgd$ weights can be obtained from the $\Rlr$ and $\Rrl$ weights in \cref{tab:R-vertices}, respectively, by first specializing to the Iwahori case \eqref{eq:Iwahori-specialization}, then multiplying by $1$ and $v^{m+2}$, respectively, and finally setting $v=0$.}
  \label{tab:R-vertices-mixed_crystal}
  \vspace{-0.5em}
  \begin{equation*}
    \begin{array}{|c|c|c|c|c|c|c|}
      \hline

      \multicolumn{6}{|c|}{\boldsymbol{\Rdg} \textbf{ and } \boldsymbol{\Rgd} \textbf{ (crystal)}\vphantom{\displaystyle\int}} 

      \\\hline\hline

      \begin{tikzpicture}[baseline, scale=0.75]
        \node[anchor=south] at (-1,-1) {$z_1$};
        \node[anchor=north] at (-1,1) {$z_2$};
      \end{tikzpicture}
      &
      \begin{tikzpicture}[baseline, scale=0.75]
        \draw (-1,-1) -- (1,1);
        \draw (-1,1) -- (1,-1);
        \node[label=above:$c_k$] (0,0) {};
      \end{tikzpicture}  
      &
      \begin{tikzpicture}[baseline, scale=0.75]
        \draw (-1,1) -- (1,-1);
        \draw[ultra thick, red] (-1,-1) -- (1,1)  node[label=below:$c_i$]{};
        \node[label=above:$c_k$] (0,0) {};
      \end{tikzpicture} 
      &
      \begin{tikzpicture}[baseline, scale=0.75]
        \draw (-1,-1) -- (1,1);
        \draw[ultra thick, red] (-1,1) -- (1,-1) node[label=above:$c_i$]{};
        \node[label=above:$c_k$] (0,0) {};
      \end{tikzpicture}
      &
      \begin{tikzpicture}[baseline, scale=0.75]
        \draw[ultra thick, red] (-1,-1) -- (1,1) node[label=below:$c_i$]{};
        \draw[ultra thick, red] (-1,1) -- (1,-1) node[label=above:$c_i$]{};
        \node[label=above:$c_k$] (0,0) {};
      \end{tikzpicture}
      &
      \begin{tikzpicture}[baseline, scale=0.75]
        \draw[ultra thick, blue] (-1,1) -- (1,-1) node[label=above:$c_j$]{};
        \draw[ultra thick, red] (-1,-1) -- (1,1) node[label=below:$c_i$]{};
        \node[label=above:$c_k$] (0,0) {};
      \end{tikzpicture}

      \\\hline

      \Rdg \vphantom{\Bigl)}
      & z_1 - z_2 
      & z_2 
      & z_2 
      & z_2
      & z_2 \text{ if } j > i

      \\\hline

      \Rgd \vphantom{\Bigl)}
      & 0
      & 0
      & 0
      & 0
      & z_1 \text{ if } j > i

      \\\hline\hline

      \begin{tikzpicture}[baseline, scale=0.75]
        \node[anchor=south] at (-1,-1) {$z_1$};
        \node[anchor=north] at (-1,1) {$z_2$};
      \end{tikzpicture}
      &
      \begin{tikzpicture}[baseline, scale=0.75]
        \draw (-1,-1) -- (0,0) --  (-1,1);
        \draw[ultra thick, red] (1,1) -- (0,0) -- (1,-1) node[label=above:$c_i$]{};
        \node[label=above:$c_k$] (0,0) {};
      \end{tikzpicture} 
      &
      \begin{tikzpicture}[baseline, scale=0.75]
        \draw (1,-1) -- (0,0) -- (1,1);
        \draw[ultra thick, red] (-1,1) -- (0,0) --  (-1,-1) node[label=above:$c_i$]{};
        \node[label=above:$c_k$] (0,0) {};
      \end{tikzpicture}
      & 
      \multicolumn{3}{c|}{
        \begin{tikzpicture}[baseline, scale=0.75]
          \draw[ultra thick, blue] (1,1) -- (0,0) --  (1,-1) node[label=above:$c_j$]{};
          \draw[ultra thick, red] (-1,1) -- (0,0) -- (-1,-1) node[label=above:$c_i$]{};
          \node[label=above:$c_k$] (0,0) {};
        \end{tikzpicture} 
      }

      \\\hline

      \Rdg \vphantom{\Bigl)} 
      & z_1 
      & z_2 
      & \multicolumn{3}{c|}{z_2 \text{ \footnotesize if $i < j < k$ or $j < k \leq i$ or $k \leq i < j$}}

      \\\hline

      \Rgd \vphantom{\Bigl)}
      & 0
      & 0
      & \multicolumn{3}{c|}{0}

      \\\hline
    \end{array}
  \end{equation*}
\end{table}

\subsection{From \texorpdfstring{$\Gamma$-$\Delta$}{Gamma-Delta} duality to left-right duality}
\label{sec:left-right_duality}
We now prove the duality between the left- and right-moving models as shown in \cref{fig:specializations}.
Recall that, with the normalization of the weights as in~\cref{tab:T_left_weights}, the partition function for the left-moving model depends on the number of fused columns $N$.

\begin{theorem}
  \label{thm:left-right_duality}
  For any $\mu = (\mu_1, \dots, \mu_r) \in (\mathbb{Z}_{\geq 0})^r$ with $\mu_1 > \mu_2 > \dots > \mu_r$, any $N \in \mathbb{Z}_{\geq 0}$ such that $Nm \geq \mu_1$, and any $\sigma = (\sigma_1, \dots, \sigma_r) \in \mathcal{P}^r$ we have that
  \begin{equation}
    \label{eq:left-right-duality}
    \Zr_{\mu,\sigma}(\mathbf{z}) = \mathbf{z}^N \ZlN{N}_{\mu,w_0\sigma}(w_0\mathbf{z}),
  \end{equation}
  where $w_0$ is the longest element of the symmetric group $S_r$, i.e.\ the order reversing permutation.
\end{theorem}

As detailed in \cref{sec:demazure-operators} both sides of \eqref{eq:left-right-duality} satisfy the same recursion.

We will use an argument similar to that used to prove Theorem A.1 in \cite{BBB}.
For the proof we need \emph{mixed} models, which are models where some rows consist of left-moving vertices and the others consist of right-moving vertices.
For any $\Theta \in \{\Left, \Right\}^r$ we consider a grid with $r$ rows where the $i^\text{th}$ row consists of vertices of type $\Theta_i$.
Similarly to before, the boundary conditions are defined by $\mu \in (\mathbb{Z}_{\geq 0})^r$ and $\sigma = (\sigma_1, \dots, \sigma_r) \in (\mathcal{P}_m)^r$.
The top boundary condition is defined by $\mu$ in exactly the same way as before, i.e.~$\mu$ lists the column numbers for which there is a path on the top edge, and the boundary condition for the bottom boundary is still that there are no paths on those edges.
For the left and right boundaries, the boundary conditions depend on whether it is a row of left- or right-moving vertices.
If $\Theta_i = \Left$, then the boundary conditions for the $i^\text{th}$ row are the same as for the left-moving systems, i.e.~there is a path on the left-most edge of color $\sigma_i$ and no path on the right-most edge.
Similarly, if $\Theta_i = \Right$, then there is a path on the right-most edge of color $\sigma_i$ and no path on the left-most edge.
We will denote this mixed system by $\mathfrak{S}^{\Theta,N}_{\mu, \sigma}$, where $Nm$ is the number of columns, and abbreviate its partition function as $Z^{\Theta,N}_{\mu, \sigma}(\mathbf{z})$.

The proof of \cref{thm:left-right_duality} now follows from repeated application of the following two lemmas.

\begin{lemma}
  \label{lemma:mixed_model_exchange_rows}
  Let $\mu$, $N$ and $\sigma$ be as in \cref{thm:left-right_duality} and $\Theta \in \{\Left, \Right\}^r$ with $\Theta_i \neq \Theta_{i+1}$ for some~$i$.
  Then
  \begin{equation}
    \label{eq:mixed_model_exchange_rows}
    Z^{\Theta,N}_{\mu, \sigma}(\mathbf{z}) = Z^{s_i\Theta,N}_{\mu, s_i\sigma}(s_i\mathbf{z}),
  \end{equation}
  where $s_i$ is the $i^\text{th}$ simple transposition in $S_r$.
\end{lemma}

\begin{proof}
  Without loss of generality we can assume that $\Theta_i = \Left$ and $\Theta_{i+1} = \Right$.
  Recall that for the system $\mathfrak{S}_{\mu,\sigma}^{\Theta,N}$ on the left-hand side of \eqref{eq:mixed_model_exchange_rows} the boundary conditions for the left boundary edges on rows $i$ and $i+1$ are that their spins are $\sigma_i$ and unoccupied, respectively.
  We may therefore consider a new system where we have attached the $\Rlr$-vertex~$\begin{tikzpicture}[baseline=-1mm,scale=0.175]
    \draw[ultra thick, blue] (-1,-1) -- (1,1);
    \draw (-1,1) -- (1,-1);
  \end{tikzpicture}^{\,\sigma_i}$
  to the left of these two edges.
  According to \cref{tab:R-vertices-mixed} there are no other admissible $\Rlr$-vertex configurations with the same left spins.
  Thus we may replace the fixed interior spins on the right side of the $\Rlr$-vertex with free, unspecified spins that are summed over without affecting the partition function, giving us a new system where only the boundary spins are fixed.
  This can be illustrated as the following equality of partition functions, where the dashed internal edges are summed over and all of the boundary edges are fixed and equal on both systems.
  \newcommand{\bdry}{0.75}
  \newcommand{\cols}{4} 
  \begin{equation}
    \label{eq:attach_R-vertex_on_the_left} 
    \begin{tikzpicture}[train]
      \newcommand{\posR}{0}

      \foreach \x in {0,...,\cols} {
        \ifthenelse{\x=\posR}{}{
          \draw[unknown] (\x, -\bdry) -- (\x, 2+\bdry);
        }
      }
      \draw[unknown] (1-\bdry, 2) -- (\cols+\bdry, 2);
      \draw[unknown] (1, 0) -- (\cols, 0);
      \draw[unknown] (1, 1) -- (\cols, 1);

      \draw (-0.5, 1)  -- (0.5, 0) -- (1, 0);
      \draw[colored, blue] (-0.5, 0) node [below, black] {$\scriptstyle\sigma_i$} -- (0.5, 1) -- (1, 1);
      \draw[colored, red] (\cols, 0) -- (\cols+\bdry, 0) node[below, black] {$\scriptstyle\sigma_{i+1}$};
      \draw (\cols, 1) -- (\cols+\bdry, 1);
    \end{tikzpicture}
    =
    \begin{tikzpicture}[train]
      \newcommand{\posR}{0}

        \foreach \x in {0,...,\cols} {
          \ifthenelse{\x=\posR}{}{
            \draw[unknown] (\x, -\bdry) -- (\x, 2+\bdry);
          }
        }
        \draw[unknown] (1-\bdry, 2) -- (\cols+\bdry, 2);
        \draw[unknown] (\posR, 0.5) -- (\posR+0.5, 0) -- (\cols, 0);
        \draw[unknown] (\posR, 0.5) -- (\posR+0.5, 1) -- (\cols, 1);

        \draw (-0.5, 1) -- (0, 0.5);
        \draw[colored, blue] (-0.5, 0)  node [below, black] {$\scriptstyle\sigma_i$} -- (0, 0.5);
        \draw[colored, red] (\cols, 0) -- (\cols+\bdry, 0) node [below, black] {$\scriptstyle\sigma_{i+1}$};
        \draw (\cols, 1) -- (\cols+\bdry, 1);

      \end{tikzpicture}  
  \end{equation}
  Since the weight of the $\Rlr$-vertex configuration is $\Phi(z_1 - q^2 z_2)$, the partition function of this new system is simply $\Phi(z_1 - q^2 z_2) Z_{\mu,\sigma}^{\Theta,N}(\mathbf{z})$.

  We can now apply the Yang--Baxter equation \eqref{eq:RTT-picture} to the system on the right-hand side of \eqref{eq:attach_R-vertex_on_the_left} to move the $R$-vertex one step to the right, which gives us the first equality below.
  We can then continue this process to move the $R$-vertex all the way to the right, so that it is attached to the right boundary of the grid.
  This is sometimes known as the \emph{train argument} and can be illustrated as
  \begin{equation*}
    \begin{tikzpicture}[train]
        \newcommand{\posR}{0}

        \foreach \x in {0,...,\cols} {
          \ifthenelse{\x=\posR}{}{
            \draw[unknown] (\x, -\bdry) -- (\x, 2+\bdry);
          }
        }
        \draw[unknown] (1-\bdry, 2) -- (\cols+\bdry, 2);
        \draw[unknown] (\posR, 0.5) -- (\posR+0.5, 0) -- (\cols, 0);
        \draw[unknown] (\posR, 0.5) -- (\posR+0.5, 1) -- (\cols, 1);

        \draw (-0.5, 1) -- (0, 0.5);
        \draw[colored, blue] (-0.5, 0) node[below,black]{$\scriptstyle\sigma_{i}$} -- (0, 0.5);
        \draw[colored, red] (\cols, 0) -- (\cols+\bdry, 0) node[below,black]{$\scriptstyle\sigma_{i+1}$};
        \draw (\cols, 1) -- (\cols+\bdry, 1);
      \end{tikzpicture}
      =
      \begin{tikzpicture}[train]
        \newcommand{\posR}{1}

        \foreach \x in {0,...,\cols} {
          \ifthenelse{\x=\posR}{}{
            \draw[unknown] (\x, -\bdry) -- (\x, 2+\bdry);
          }
        }
        \draw[unknown] (-\bdry, 2) -- (\cols+\bdry, 2);
        \draw[unknown] (0, 1) -- (\posR-0.5, 1) -- (\posR, 0.5) -- (\posR+0.5, 0) -- (\cols, 0);
        \draw[unknown] (0, 0) -- (\posR-0.5, 0) -- (\posR, 0.5) -- (\posR+0.5, 1) -- (\cols, 1);

        \draw (-\bdry, 1) -- (0, 1);
        \draw[colored, blue] (-\bdry, 0)node[below,black]{$\scriptstyle\sigma_{i}$} -- (0, 0);
        \draw[colored, red] (\cols, 0) -- (\cols+\bdry, 0)node[below,black]{$\scriptstyle\sigma_{i+1}$};
        \draw (\cols, 1) -- (\cols+\bdry, 1);
      \end{tikzpicture}
      = \cdots =
      \begin{tikzpicture}[train]

        \newcommand{\posR}{4}

        \foreach \x in {0,...,\cols} {
          \ifthenelse{\x=\posR}{}{
            \draw[unknown] (\x, -\bdry) -- (\x, 2+\bdry);
          }
        }
        \draw[unknown] (-\bdry, 2) -- (\cols+\bdry-1, 2);
        \draw[unknown] (0, 1) -- (\posR-0.5, 1) -- (\posR, 0.5);
        \draw[unknown] (0, 0) -- (\posR-0.5, 0) -- (\posR, 0.5);

        \draw (-\bdry, 1) -- (0, 1);
        \draw[colored, blue] (-\bdry, 0) node[below,black]{$\scriptstyle\sigma_{i}$} -- (0, 0);
        \draw[colored, red] (\cols, 0.5) -- (\cols+0.5, 0) node[below,black]{$\scriptstyle\sigma_{i+1}$};
        \draw (\cols, 0.5) -- (\cols+0.5, 1);
    \end{tikzpicture}
  \end{equation*}
  where again boundary edges are fixed and matching.

  Similarly to before, it can be seen in \cref{tab:R-vertices-mixed} that there is only one admissible vertex configuration for the $R$-vertex in the right-most system, and we can therefore deduce that its partition function is $\Phi(z_1 - q^2 z_2) Z^{s_i\Theta,N}_{\mu, s_i\sigma}(s_i\mathbf{z})$, a multiple of the right-hand side of~\eqref{eq:mixed_model_exchange_rows}.
  Note that the row types and row parameters of row $i$ and $i+1$ were swapped when the $R$-vertex moved through the grid.

  Hence
  \begin{equation*}
    \Phi(z_1 - q^2 z_2) Z^{\Theta,N}_{\mu, \sigma}(\mathbf{z}) = \Phi(z_1 - q^2 z_2) Z^{s_i\Theta,N}_{\mu, s_i\sigma}(s_i\mathbf{z})
  \end{equation*}
  as elements of the fraction field $\mathcal{F}(z_1, \ldots z_r)$ where $\mathcal{F}$ was introduced in \cref{sec:families_of_lattice_models} which implies \eqref{eq:mixed_model_exchange_rows}.
\end{proof}

The following lemma shows that the partition function does not depend on the last row type.

\begin{lemma}
  \label{lemma:mixed_model_last_row}
  Let $\mu$, $N$ and $\sigma$ be as in \cref{thm:left-right_duality} and $\Theta_1,\ldots,\Theta_{r-1} \in \{\Left, \Right\}$.
Then there is a bijection $\ell : \mathfrak{S}_{\mu, \sigma}^{(\Theta_1,\ldots,\Theta_{r-1},\Right),N} \to \mathfrak{S}_{\mu, \sigma}^{(\Theta_1,\ldots,\Theta_{r-1},\Left),N}$ such that $\wt(\mathfrak{s}) = z_r^N \wt(\ell(\mathfrak{s}))$ for all $\mathfrak{s} \in \mathfrak{S}_{\mu, \sigma}^{(\Theta_1,\ldots,\Theta_{r-1},\Right),N}$.
It thus follows for the partition functions that
  \begin{equation}
    \label{eq:mixed_model_last_row}
    Z_{\mu, \sigma}^{(\Theta_1,\ldots,\Theta_{r-1},\Right),N}(\mathbf{z}) = z_r^N Z_{\mu, \sigma}^{(\Theta_1,\ldots,\Theta_{r-1},\Left),N}(\mathbf{z}).
  \end{equation}
\end{lemma}

\begin{proof}
  We construct $\ell$ as follows. 
  Let $\mathfrak{s} \in \mathfrak{S}_{\mu, \sigma}^{(\Theta_1,\ldots,\Theta_{r-1},\Right),N}$, that is, a state where the last row consists of right-moving vertices.
  Because of the boundary condition, there can only be a single path on the last row, say of color $c_i$, which comes down from above on some vertical edge, and then goes straight to the right boundary as illustrated in \cref{fig:last-row}.

  By changing the path to go to the left instead of to the right on the last row (starting from the same vertical edge) we obtain a state $\ell(\mathfrak{s}) \in \mathfrak{S}_{\mu, \sigma}^{(\Theta_1,\ldots,\Theta_{r-1},\Left),N}$ as illustrated in \cref{fig:last-row}.
  With only the last row changing this is clearly a bijection between the two systems.

  \begin{figure}[htpb]
    \centering
    \newcommand{\bdry}{0.75}
    \newcommand{\topBdry}{0.75}
    \begin{equation*}
    \mathfrak{s} = 
    \begin{tikzpicture}[train]
      \draw (-\bdry, 0) -- (4+\bdry, 0);
      \foreach \x in {0,...,4} {
        \draw (\x, -\bdry) -- (\x, 1);
      }
      \draw[unknown] (-\bdry, 1) -- (4+\bdry, 1);
      \draw[unknown] (-\bdry, 2) -- (4+\bdry, 2);
      \foreach \x in {0,...,4} {
        \draw[unknown] (\x, 1) -- (\x, 2+\topBdry);
      }
      \draw[red, colored] (1,1) -- (1,0) -- (4+\bdry,0);
    \end{tikzpicture}
    \qquad\qquad
    \ell(\mathfrak{s}) = 
    \begin{tikzpicture}[train]
      \draw (-\bdry, 0) -- (4+\bdry, 0);
      \foreach \x in {0,...,4} {
        \draw (\x, -\bdry) -- (\x, 1);
      }
      \draw[unknown] (-\bdry, 1) -- (4+\bdry, 1);
      \draw[unknown] (-\bdry, 2) -- (4+\bdry, 2);
      \foreach \x in {0,...,4} {
        \draw[densely dashed] (\x, 1) -- (\x, 2+\topBdry);
      }
      \draw[red, colored] (1,1) -- (1,0) -- (-\bdry,0);
    \end{tikzpicture}
    \end{equation*}
    \caption{The last row of $\mathfrak{s}$ and $\ell(\mathfrak{s})$.}
    \label{fig:last-row}
  \end{figure}  

  Let us now consider the weight.
  The only vertex configurations that can appear on the last row of $\mathfrak{s}$ or $\ell(\mathfrak{s})$ are of types $\texttt{a}_1$, $\texttt{b}_2$ and $\texttt{c}_2$.
  Under the map $\ell$ every right-moving $\texttt{a}_1$ configuration on the last row is mapped to a left-moving $\texttt{b}_2$ configuration, and similarly $\texttt{b}_2$ to $\texttt{a}_1$, and $\texttt{c}_2$ to $\texttt{c}_2$.
  Comparing \cref{tab:T_right_weights,tab:T_left_weights} we see that their weights are preserved if we multiply the left-moving weights by $z$ whenever the column color $c_j$ is equal to $c_i$.
  Therefore the weight of $\mathfrak{s}$ is equal to the weight of $\ell(\mathfrak{s})$ multiplied by $z_r^N$ with $N$ being the number of blocks, that is, fused columns.
\end{proof}

\begin{proof}[Proof of \cref{thm:left-right_duality}]
  Consider the right-moving system $\Sr_{\mu,\sigma}(\mathbf{z})$ on the left-hand side of the equality.
  By \cref{lemma:mixed_model_last_row} we can replace the last row with a left-moving row:
  \begin{equation*}
    \Zr_{\mu,\sigma}(\mathbf{z}) = z_r^N Z^{(\Right,\dots,\Right,\Left),N}_{\mu, \sigma}(\mathbf{z}).
  \end{equation*}
  We can then repeatedly apply \cref{lemma:mixed_model_exchange_rows} to move the left-moving row to the top:
  \begin{equation*}
    z_r^N Z^{(\Right,\dots,\Right,\Left),N}_{\mu, \sigma}(\mathbf{z})
    = z_r^N Z^{(\Right,\dots,\Left,\Right),N}_{\mu, s_{r-1}\sigma}(s_{r-1}\mathbf{z})
    = \dots
    = z_r^N Z^{(\Left,\Right,\dots,\Right),N}_{\mu, w_1\sigma}(w_1\mathbf{z}),
  \end{equation*}
  where $w_1 = s_1 s_2 \dots s_{r-1}$.
  Then we can use \cref{lemma:mixed_model_last_row} again:
  \begin{equation*}
    z_r^N z_{r-1}^N Z^{(\Left,\Right,\dots,\Right),N}_{\mu, w_1\sigma}(w_1\mathbf{z})
    = z_r^N z_{r-1}^N Z^{(\Left,\Right,\dots,\Right,\Left),N}_{\mu, w_1\sigma}(w_1\mathbf{z}),
  \end{equation*}
  and \cref{lemma:mixed_model_exchange_rows} to move the last row up to the second row:
  \begin{equation*}
    z_r^N z_{r-1}^N Z^{(\Left,\Right,\dots,\Right,\Left),N}_{\mu, w_1\sigma}(w_1\mathbf{z})
    = \dots
    = z_r^N z_{r-1}^N Z^{(\Left,\Left,\Right,\dots,\Left),N}_{\mu, w_2\sigma}(w_2\mathbf{z}),
  \end{equation*}
  where $w_2 = s_2 \dots s_{r-1} s_1 \dots s_{r-1}$.

  By repeating this for every row, we will eventually end up with
  \begin{equation*}
    \mathbf{z}^N Z^{\Left,N}_{\mu, w_0\sigma}(w_0\mathbf{z}),
  \end{equation*}
  where $w_0 = s_{r-1} (s_{r-2} s_{r-1}) (s_{r-3} \dots s_{r-1}) \dots (s_2 \dots s_{r-1}) (s_1 \dots s_{r-1})$, the longest element in $S_r$, which finishes the proof.
\end{proof}

\begin{remark}
\label{rem:packets}
As mentioned above, the duality between the (non-crystal) left- and right-moving families cannot, in general, be refined to a bijection on states (see \cref{fig:left-right_duality_example} for an example). 
This issue was studied in~\cite{BBF:orange} for the metaplectic $\Gamma$ and $\Delta$ case (in the language of Gelfand--Tsetlin patterns predating the constructions of the corresponding lattice models).
They empirically observed that the Gelfand--Tsetlin patterns could be partitioned into smaller \emph{packets} which give a partial refinement of the duality, i.e.~that the partition functions of the packet and its dual agree. 

For special subsets of Gelfand--Tsetlin patterns they showed in~\cite[Chapter~10]{BBF:orange} that the corresponding packets are singletons.
Specifically, they showed that the metaplectic $\Gamma$-$\Delta$ duality reduces to a state-by-state bijection using the Schützenberger involution on Gelfand--Tsetlin patterns on these particular subsets.

Furthermore, they showed that the packets in a related \emph{Statement E} are described using facets of a particular simplex, but it is unclear how that can be interpreted in terms of lattice models states.

Since the members of each lattice model family is related by a form of Drinfeld twisting, and since the sets of states appearing on each side of the equality are the same for all family members, it would be natural to consider packets that are invariant for the whole family.
The results of \cref{sec:crystal} for the crystal limit of the Iwahori specialization could then be an important piece of the puzzle.
Note that the bijection in this paper is on states which have first been transformed by a fusion process explained in \cref{sec:introduction_to_lattice_models}, while the bijection in \cite{BBF:orange} would have a lattice model interpretation in terms of states that are not fused.
Since the fusion process merges groups of columns of the lattice model, and the Gelfand--Tsetlin pattern corresponding to a state determines the occupied columns at different rows as explained in \cref{sec:GTP}, this means that the bijection between states and Gelfand--Tsetlin patterns is slightly different in the two papers.
\end{remark}

\section{Crystal models and the Schützenberger involution}
\label{sec:crystal}
In this section we will show that the $\Gamma$-$\Delta$ duality for the crystal limit refines to a weight-respecting bijection of states based on the Schützenberger involution.
Moreover, we show that the individual row exchanges in \cref{lemma:mixed_model_exchange_rows} have a state-by-state interpretation in terms of involutions $t_i$ on Gelfand--Tsetlin patterns defined by Berenstein and Kirillov in~\cite{KirillovBerenstein}.
Gelfand--Tsetlin patterns are in bijection with semistandard Young tableaux for which the Berenstein--Kirillov involutions are the well-known Bender--Knuth involutions introduced in \cite{BenderKnuth}.
Applying these involutions in the same order as the row exchanges in \cref{thm:left-right_duality} gives exactly the Schützenberger involution.

In this section we will only discuss fused models, that is, where a column corresponds to a block in the unfused model.
We fix an arbitrary $r$-tuple of colors (not necessarily distinct) and let it specify the colors on the top boundary of our systems from left to right.
Our arguments will not depend on this choice.
We will denote mixed systems with the fused crystal configurations and weights from \cref{tab:T_fused_gamma_Iwahori_crystal_weights,tab:T_fused_delta_Iwahori_crystal_weights} by $\mathfrak{S}^{\Theta}_{\lambda+\rho,\sigma}$ where $\Theta \in \{\Gamma,\Delta\}^r$ are the row types, the integer partition $\lambda \in \mathbb{Z}^r$ and $\rho = (r-1,r-2,\ldots,1,0)$ determine the occupied top boundary (fused) column numbers, and $\sigma = (\sigma_1,\ldots, \sigma_r)$ are the colors for the occupied horizontal boundary edges (from top to bottom).
Note that in this section $\lambda + \rho$ gives the fused column numbers rather than the unfused column numbers, and that there is no longer any dependence on the number of fused columns $N$ for the partition functions since the $\Delta$ crystal weights have been rescaled to avoid this dependence.
We will sometimes consider the union over all possible $\sigma$ and denote this by $\mathfrak{S}^{\Theta}_{\lambda+\rho} = \bigsqcup_{\sigma} \mathfrak{S}^{\Theta}_{\lambda+\rho,\sigma}$.
In this section it is important that the systems only contain admissible states (i.e.\ with non-zero weight).
As before we abbreviate the partition functions $Z(\mathfrak{S}^\Theta_{\lambda+\rho,\sigma})(\mathbf{z})$ and $Z(\mathfrak{S}^\Theta_{\lambda+\rho})(\mathbf{z})$ to $Z^\Theta_{\lambda+\rho,\sigma}(\mathbf{z})$ and $Z^\Theta_{\lambda+\rho}(\mathbf{z})$, respectively.

\subsection{Gelfand--Tsetlin patterns}
\label{sec:GTP}
Recall that with our usual boundary conditions on the left, right and bottom boundaries, when we let $v \to 0$ in the fused Gamma and Delta Iwahori models we only obtain states where all vertical edges carry at most a single color according to \cref{lemma:crystal_models_at_most_one_color}.
This is why we may restrict to vertex configurations of this type, shown in \cref{tab:T_fused_gamma_Iwahori_crystal_weights,tab:T_fused_delta_Iwahori_crystal_weights}, without affecting the partition function.
Therefore the top boundary conditions are given by $r$ \emph{distinct} fused column numbers and associated colors from the fixed $r$-tuple of colors.
This is in contrast to the non-crystal limit where there may be multiple paths on the same fused edge on the top boundary, as it corresponds to an entire block.

We will show that the states with top boundary columns given by $\lambda + \rho$ are in bijection with Gelfand--Tsetlin patterns with top row $\lambda$.
A Gelfand--Tsetlin pattern $\mathfrak{T}$ is a triangular arrangement of non-negative integers
\begin{equation}
  \label{eq:GTP}
  \mathfrak{T}= \left\{ \begin{array}{ccccccccc} a_{0,0} &  & a_{0,1} &  & a_{0,2} &  & \cdots &  & a_{0, r-1}\\ & a_{1,1} &  & a_{1,2} &  &  & \cdots & a_{1,r-1} & \\ &  & \ddots &  & \vdots &  & \iddots &  & \\ &  &  &  & a_{r-1, r-1} &  &  &  & \end{array} \right\}
\end{equation}
such that each triangle of neighboring integers satisfies the inequalities
\begin{equation}
  \label{eq:triangle}
  \begin{array}{ccccc}
    a_{i-1,j-1} & & & & a_{i-1,j} \\
    & \rotatebox[origin=c]{-40}{$\geq$} &  & \rotatebox[origin=c]{40}{$\geq$} & \\
    &  & a_{i,j} &  &
  \end{array}
\end{equation}
We denote by $\GTP_\lambda$ the set of Gelfand--Tsetlin patterns with top row $\lambda = (\lambda_1, \ldots, \lambda_r)$, that is, those satisfying $a_{0,i} = \lambda_{i+1}$ in \eqref{eq:GTP}.
We will also consider the subsets of \emph{left-strict} or \emph{right-strict} Gelfand--Tsetlin patterns where the inequalities in~\eqref{eq:triangle} are restricted to $a_{i-1,j-1} > a_{i,j} \geq a_{i-1,j}$ or $a_{i-1,j-1} \geq a_{i,j} > a_{i-1,j}$, respectively.
As will be shown below in \cref{prop:states-GTP-bijection}, the left-strict Gelfand--Tsetlin patterns are in bijection with the states of the $\Gamma$ crystal model and the right-strict patterns are in bijection with states of the $\Delta$ crystal model.
We will therefore also say that these two subsets of Gelfand--Tsetlin patterns are of type $\Gamma$ and $\Delta$, respectively.

More generally, we say that a pair of consecutive rows in a Gelfand--Tsetlin pattern is of type $\Gamma$ or $\Delta$ if the triangular arrangements of inequalities between these two rows are all left-strict or all right-strict, respectively. 
We may consider subsets of Gelfand--Tsetlin patterns with mixed row-pair types specified by a list $\Theta' = (\Theta_1', \ldots, \Theta'_{r-1})$ where $\Theta'_i \in \{\Gamma, \Delta\}$ counting from top to bottom.
We denote these subsets by $\GTP^{\Theta'}_{\lambda} \subset \GTP_\lambda$ and say that an element of $\GTP^{\Theta'}_{\lambda}$ is $\Theta'$-strict or of row-pair type $\Theta'$.
Note that they are not disjoint.
Note also that the list of row-pair types $\Theta'$ has length $r-1$ while the list of row types $\Theta$ for lattice models defined in~\cref{sec:left-right_duality} has length $r$.
These are related in the following proposition which proves that the set of admissible mixed states with row types $\Theta = (\Theta_1, \ldots, \Theta_r)$ and top row boundary condition $\lambda + \rho$ is in bijection with the set of Gelfand--Tsetlin patterns with $r$ rows, row-pair types $\Theta' = (\Theta_1,\dots,\Theta_{r-1})$ and top row $\lambda + \rho$.
Later, in \cref{prop:state-GTP-weight} we show that this bijection preserves the notion of weight (which we define for Gelfand--Tsetlin patterns below).

\begin{proposition}
  \label{prop:states-GTP-bijection}
  Fix a list of row types $\Theta = (\Theta_1, \dots, \Theta_r) \in \{\Gamma, \Delta\}^r$ and let $\Theta' = (\Theta_1,\dots,\Theta_{r-1})$.
  There is a bijection $\iota_\Theta : \mathfrak{S}^\Theta_{\lambda+\rho} \to \GTP^{\Theta'}_{\lambda+\rho}$.
\end{proposition}

The proof follows below.
For similar statements for other lattice models see for example \cite{BBF, BBCFG}.
From now on, we write the truncation of a list row types $\Theta = (\Theta_1, \dots, \Theta_{r-1}, \Theta_r)$ to a list of row-pair types as $\Theta' = (\Theta_1, \dots, \Theta_{r-1})$.

\begin{remark}
  \label{rem:unique_coloring}
  As we will see in the proof, the Gelfand--Tsetlin pattern does not carry any information about the colors of the paths in the lattice model state; it only depends on the positions of the fused vertical edges the paths occupy.
  Recall that the fused columns do not restrict which colors their edges can carry, unlike the unfused columns.
  However, the color information is not lost since given the layouts of the paths and given the fixed top boundary colors we show that there is a unique way of coloring the paths for the crystal model.
  The notion of an uncolored state can be obtained by considering a palette with only a single color, say gray, and we can think of a gray edge simply as \emph{occupied}.
  This deterministic coloring feature is not shared with the Iwahori model for general $v \neq 0$ as can be seen for example in \cref{fig:example_fused_right-moving_state} where we can swap red and green between the second and third row.
\end{remark}

\begin{proof}[Proof of \cref{prop:states-GTP-bijection}]
  Suppose we have some mixed state $\mathfrak{s}$ with row types $\Theta$ and top boundary condition given by $\lambda + \rho$.
  Recall that $\mathfrak{s}$ has $r$ rows indexed from $1$ to $r$ starting from the top, that there are $r$ paths starting from the top boundary, and that on each row one path exits the lattice to either the right or left depending on $\Theta$.
  Therefore there will be paths on $r$ of the vertical edges at the top boundary, on $r-1$ of the vertical edges directly below the first row, and so on.
  We can then define a Gelfand--Tsetlin pattern $\mathfrak{T}$ by setting the row with index $0\leq i \leq r-1$ to be the column numbers of the occupied vertical edges below row $i$ in the lattice, taken from left to right.
  In particular, for $i=0$ this means that we set $a_{0,j} = (\lambda + \rho)_j$ by imagining a zeroth row above the top boundary.
  Recall that the column numbers are decreasing from left to right so that, by definition, each row in $\mathfrak{T}$ is strictly decreasing. 

  For any $1 \leq i \leq r-1$, to see that the necessary inequalities for the row-pair $(i-1, i)$ in $\mathfrak{T}$ are satisfied, suppose first that $\Theta_i = \Gamma$, i.e.~that the $i^\text{th}$ row of $\mathfrak{s}$ is a Gamma row.
  Recall that for a Gamma vertex we consider the top and left edges to be inputs, and the bottom and right edges to be outputs.
  
  We have to check the two inequalities coming from the left-strict version of \eqref{eq:triangle}: (i) $a_{i-1,j-1} >  a_{i,j}$ and (ii) $a_{i,j} \geq a_{i-1,j}$ for every $j$.
  For (i), suppose instead that $a_{i-1,j-1} \leq a_{i,j}$ for some $j$.
  Consider the section of vertices on the $i^\text{th}$ row that are left of column $a_{i,j}$, including the vertex on column $a_{i,j}$ itself, together with their edges.
  The number of occupied bottom (top) edges in this section is given by the number of entries in $\mathfrak{T}$ on row $i$ (row $i-1$) that are larger than or equal to $a_{i,j}$.
  If $a_{i-1,j-1} < a_{i,j}$, this means that there are strictly fewer occupied edges at the top than at the bottom of the section, since the rows in $\mathfrak{T}$ are strictly decreasing.
  The boundary conditions imply that the left boundary edge of the section is unoccupied, and thus the number of ingoing paths is strictly smaller than the number of outgoing paths for the section, regardless of the occupancy of its right boundary edge, contradicting the fact that the number of paths has to be preserved for this section.

  If, on the other hand, $a_{i-1,j-1} = a_{i,j}$, then there are as many paths coming in from the top edges as there are going out via the bottom edges of the section of vertices.
  However, the right-most vertex of the section (at column $a_{i,j}$) must, by the assumption, have both top and bottom edges occupied, and so has to be of type $\texttt{a}_2$ or $\texttt{a}_2'$.
  This means that there must be an additional path going out to the right, but, as before, none coming in from the left, which again leads to a contradiction of path preservation.

  For (ii), suppose for a contradiction that $a_{i,j} < a_{i-1,j}$.
  We proceed similarly to above, but instead consider the section of vertices on the $i^\text{th}$ row that are \emph{right} of column $a_{i,j}$, including the vertex on column $a_{i,j}$ itself.
  The number of occupied bottom (top) edges in this section is given by the number of entries in $\mathfrak{T}$ on row $i$ (row $i-1$) that are \emph{smaller} than or equal to $a_{i,j}$.
  If $a_{i,j} < a_{i-1,j}$, this means that there are strictly fewer occupied edges at the top than at the bottom of the section.
  The boundary conditions imply that the right boundary edge of the section is occupied, and thus the number of ingoing paths is strictly smaller than the number of outgoing paths for the section, regardless of the occupancy of its left boundary edge, giving a contradiction.

  The $\Delta$ case is similar, and hence $\mathfrak{T}$ is a Gelfand--Tsetlin pattern with row-pair types $\Theta'$.

  Conversely, given a Gelfand--Tsetlin pattern $\mathfrak{T}$ with $r$ rows, row-pair types $\Theta'$ and top row $\lambda+\rho$ we can define an admissible state $\mathfrak{s}$ as follows.
  We will first specify which edges are occupied without considering colors and then show that there is only one way to assign colors to the occupied edges to obtain an admissible state with the fixed top boundary colors.
  
  For every $0\leq i \leq r-1$ and $i \leq j \leq r-1$, let the vertical edge in column $a_{i,j}$ below row $i$ in the lattice be occupied.
  Then fill in the horizontal edges in each row $i$ of $\mathfrak{s}$ by adding a horizontal uncolored path from each top vertical edge to the next bottom vertical edge (if there is one, otherwise to the boundary), to the right if $\Theta_i = \Gamma$, or to the left if $\Theta_i = \Delta$.
  Because of the left-strict and right-strict Gelfand--Tsetlin inequalities, the paths will neither overlap (i.e.~occupy the same edge) nor go straight down, so we get an uncolored state whose only vertex configurations are the uncolored versions of those in \cref{tab:T_fused_gamma_Iwahori_crystal_weights} for $\Gamma$ rows and those in \cref{tab:T_fused_delta_Iwahori_crystal_weights} for $\Delta$ rows.
  
  Now note that given an uncolored version of one of these vertex configurations and a coloring of its occupied input edges, there is a unique way to color the occupied output edges such that the resulting vertex configuration becomes admissible.
  The colors of the top boundary edges are given.
  For each row, the left or right boundary edge is unoccupied for row type $\Gamma$ or $\Delta$, respectively.
  This means that we know the colors of the occupied input edges of the top left or top right vertex, respectively, and can therefore deduce the colors of its occupied output edges.
  By proceeding iteratively along the row we can assign unique colors to each occupied horizontal edge as well as each occupied vertical edge directly below the row.
  We can then continue downwards row by row until every occupied edge has been colored.
\end{proof}

The weight $\wt(\mathfrak{T})$ of a Gelfand--Tsetlin pattern $\mathfrak{T}$ with $r$ rows is defined as the list of integers $(d_{r-1}, d_{r-2} - d_{r-1}, \ldots, d_0 - d_1)$ of differences of consecutive row sums from the bottom up
\begin{equation}
  \label{eq:GTP-weight}
  d_i = \sum_{j=i}^{r-1} a_{i,j}.
\end{equation}
The first entry can also be interpreted as a difference of consecutive row sums if we imagine that there is an extra empty row below the $(r-1)^\text{th}$ row, so that $d_r = 0$.

\begin{proposition}
  \label{prop:state-GTP-weight}
  Let $\Theta \in \{\Gamma, \Delta\}^{r}$ and $\mathfrak{T} \in \GTP^{\Theta'}_{\lambda+\rho}$.
  If $\mathfrak{s} = \iota^{-1}_\Theta (\mathfrak{T})$ is the state in $\mathfrak{S}^\Theta_{\lambda+\rho}$ associated to $\mathfrak{T}$ by \cref{prop:states-GTP-bijection}, then
  \begin{equation}
    \wt(\mathfrak{s})(\mathbf{z}) = \mathbf{z}^{w_0 \wt(\mathfrak{T})},
  \end{equation}
  where $w_0$ is the longest element of the symmetric group $S_r$, i.e.~the order reversing permutation.
\end{proposition}

In particular, this means that the weight $\wt(\iota^{-1}_\Theta (\mathfrak{T}))(\mathbf{z})$ of the state with row types $\Theta$ associated to $\mathfrak{T}$ does not depend on the last row type $\Theta_r$ (remember that in this section the fused Delta weights have been rescaled by $z$, so this does not contradict \cref{lemma:mixed_model_last_row}).

\begin{proof}
  First note from \cref{tab:T_fused_gamma_Iwahori_crystal_weights,tab:T_fused_delta_Iwahori_crystal_weights} that the left-hand side is a monomial in $z_1, \dots, z_r$.
  Furthermore, the exponent of $z_i$ is determined by the weight of the $i^\text{th}$ row in $\mathfrak{s}$, and it is therefore sufficient to show that this weight is equal to 
  \begin{equation*}
    z_i^{d_{i-1} - d_i},
  \end{equation*}
  where $d_{i}$ is the sum of the $i^\text{th}$ row of the Gelfand--Tsetlin pattern as above (where again we interpret $d_r$ as $0$).
  We only present the proof for $i<r$; the remaining case $i=r$ is essentially the same but with some obvious modifications due to the fact that there is no row indexed by $r$ in the Gelfand--Tsetlin pattern.

  First consider the case when $\Theta_i = \Gamma$.
  We can see in \cref{tab:T_fused_gamma_Iwahori_crystal_weights} that the admissible $\Gamma$ weights are either $1$ or $z_i$ for row $i$, and that a vertex configuration has weight $z_i$ precisely when its left edge is occupied.
  So the exponent of $z_i$ is equal to the number of vertices on row $i$ with occupied left-edges, which in turn is equal to the number of occupied internal horizontal edges (since there is no path on the left boundary). 
  On row $i$ there are a number of paths coming in from above at the columns numbered $a_{i-1,i-1}, \dots, a_{i-1,r-1}$, by the definition of $\mathfrak{s}$ in the proof of \cref{prop:states-GTP-bijection}.
  Furthermore, all but one path leave the row via the bottom edges at the columns numbered $a_{i,i}, \dots, a_{i,r-1}$, and the remaining path goes out to the right boundary.
  Hence the number of occupied internal horizontal edges on row $i$ is the summed differences of these column numbers, namely
  \begin{equation*}
    \sum_{j=i-1}^{r-1} a_{i-1,j} - \sum_{j=i}^{r-1} a_{i,j},
  \end{equation*}
  which is what we wanted to show.

  For the case when $\Theta_i = \Delta$ we can similarly see from \cref{tab:T_fused_delta_Iwahori_crystal_weights} that the exponent of $z_i$ is equal to the number of vertices on row $i$ with \emph{unoccupied} left-edges, which in turn is equal to the number of unoccupied internal horizontal left-edges (since the left boundary edge is occupied).
  Again, on row $i$ in the lattice there are a number of paths coming in from above at the columns numbered $a_{i-1,i-1}, \dots, a_{i-1,r-1}$ and all but one path leave the row via the bottom edges at the columns numbered $a_{i,i}, \dots, a_{i,r-1}$, while the remaining path goes out to the left boundary at column number $N$.
  Hence the number of occupied internal horizontal edges on row $i$ is $(N + \sum_{j=i}^{r-1} a_{i,j}) - \sum_{j=i-1}^{r-1} a_{i-1,j}$, which means that the number of \emph{unoccupied} such edges is
  \begin{equation*}
    \sum_{j=i-1}^{r-1} a_{i-1,j} - \sum_{j=i}^{r-1} a_{i,j},
  \end{equation*}
  which is what we wanted to show.
\end{proof}

To summarize, we have shown that the states $\mathfrak{s}$ of the mixed system with row types $\Theta$ and top boundary given by $\lambda + \rho$ is in bijection with $\Theta'$-strict Gelfand--Tsetlin patterns $\mathfrak{T}$ with top row $\lambda + \rho$ and that this bijection is weight-respecting, meaning that $\wt(\mathfrak{s})(\mathbf{z}) = \mathbf{z}^{w_0 \wt(\mathfrak{T})}$.

In the next subsection we will show that the above $\Theta'$-strict Gelfand--Tsetlin patterns are in bijection with regular (not necessarily strict) patterns with top row $\lambda$ instead of $\lambda + \rho$.
We show that each Berenstein--Kirillov involution $t_{r-i}$ is well-behaved under this bijection and maps $\Theta'$-strict patterns to $(s_i\Theta)'$-strict patterns.
We will then use this to transfer the Berenstein--Kirillov involutions on Gelfand--Tsetlin patterns to bijections of the states of mixed models where the row types differ by a swap of an adjacent $\Gamma$ and $\Delta$ row.
By mirroring the decomposition of the Schützenberger involution in terms of Berenstein--Kirillov involutions we then obtain a bijection between the states of $\Gamma$ and $\Delta$ models that refine the Gamma-Delta duality from \cref{thm:left-right_duality} for the crystal models.

\subsection{Berenstein--Kirillov involutions}
We now define the Berenstein--Kirillov involutions on Gelfand--Tsetlin patterns introduced in~\cite{KirillovBerenstein} and mentioned earlier in this section, but using the index conventions of \cite{BBF:orange}.
Let the entries $a_{i,j}$ of a rank $r$ Gelfand--Tsetlin pattern be indexed as in~\eqref{eq:GTP} such that $i \in \{0,\ldots, r-1\}$ enumerates the rows from top to bottom.
There are $r-1$ Berenstein--Kirillov involutions, and each involution acts on a single row $i \in \{1,\ldots, r-1\}$ of the pattern, leaving the remaining rows, including the top row, unchanged.
This mean that there is one involution acting on each row except for the top row and due to differing indexing conventions the involution that acts on row $i$ is denoted $t_{r-i}$ for $i \in \{1,\ldots, r-1\}$. 
Since the top row is always unchanged, these are involutions of $\GTP_\lambda$, the set of Gelfand--Tsetlin patterns with top row $\lambda$.

To describe the action of the Berenstein--Kirillov involution $t_{r-i}$, first note that the defining inequalities of the Gelfand--Tsetlin patterns restrict every entry $a_{i,j}$ to the interval with upper bound given by the minimum of the (up to) two neighboring entries to the northwest and southwest, and lower bound given by the maximum of the (up to) two neighboring entries to the northeast and southeast.
The involution $t_{r-i}$ is then given by reflecting each entry $a_{i,j}$ on the $i^\text{th}$ row in its admissible interval, that is, replacing $a_{i,j}$ by
\begin{equation}
  \label{eq:ti}
  a'_{i,j} = \min(a_{i-1,j-1}, a_{i+1,j}) + \max(a_{i-1,j}, a_{i+1,j+1}) - a_{i,j}.
\end{equation}
When $j=i$ or $j=r-1$, that is, when $a_{i,j}$ is on one of the diagonals, we interpret the first and second terms as $a_{i-1,j-1}$ and $a_{i-1,j}$, respectively. 
An example is given in \cref{fig:Berenstein-Kirillov-example}.

In terms of lattice models, we will show that by interpreting the Gelfand--Tsetlin patterns as mixed states by \cref{prop:states-GTP-bijection}, these Berenstein--Kirillov involutions can be thought of as swapping a Gamma and a Delta row, just like the mixed $R$-matrix did in \cref{sec:left-right_duality}.
In this way we obtain a state-by-state refinement not only of the complete Gamma-Delta duality, but of each individual step.

There is a well-known bijection $\tab : \GTP_\lambda \to \SSYT(\lambda)$ between Gelfand--Tsetlin patterns $\mathfrak{T} \in \operatorname{GTP}_{\lambda}$ and semistandard Young tableaux $T \in \operatorname{SSYT}(\lambda)$ such that the rows of $\mathfrak{T}$ define consecutive shapes of $T$ with increasing box contents. 
Under this bijection the Berenstein--Kirillov involutions $t_i$ are the Bender--Knuth involutions $\operatorname{BK}_i$ on semistandard Young tableaux introduced in~\cite{BenderKnuth}. 

Another involution $q_i$ of Gelfand--Tsetlin patterns can be defined recursively by $q_{i} = q_{i-1} t_i \cdots t_2 t_1 $ for $i \in \{1, \ldots, r-1\}$ with $q_0 = 1$, and then $q_{r-1}$ coincides with the Schützenberger involution $\Sch$ on semistandard Young tableaux by \cite[Theorem~2.1]{KirillovBerenstein}.
The other involutions $q_i$ for $i < r-1$ are called partial Schützenberger involutions.
It was proved in~\cite{KirillovBerenstein} that the compositions $s_i \coloneq q_i q_1 q_i$ corresponds to the Lascoux--Schützenberger action of the symmetric group on semistandard Young tableaux, with $s_i$ a simple transposition, defined in for example~\cite{Lothaire}.
Note that while the Lascoux--Schützenberger involutions $s_i$ satisfy the braid relations of $S_r$, the involutions $t_i$ and $\operatorname{BK}_i$ do not.

Since the Berenstein--Kirillov involution $t_{r-i}$ only acts and depends on rows $i-1$, $i$ and $i+1$ of a Gelfand--Tsetlin pattern (if these rows exist) it is often convenient to restrict our attention to three such rows.
We call three consecutive rows of some Gelfand--Tsetlin pattern a \emph{short pattern} $\mathfrak{t}$ and we will use the following parametrization of its entries
\begin{equation}
  \label{eq:short-pattern}
  \mathfrak{t} =
  \left\{ \begin{array}{lllllllll}  x_0 &  & x_1 &  & x_2 & \cdots & x_{\ell + 1} &  & x_{\ell + 2}\\
  & y_0 &  & y_1 &  & \cdots &  & y_{\ell + 1} & \\
  &  & z_0 &  & z_1 & \cdots & z_{\ell} &  & 
  \end{array} \right\}.
\end{equation}
Mirroring the definition for Gelfand--Tsetlin patterns, we define the weight of a short pattern to be the pair of integers consisting of the two differences of consecutive row sums, i.e.
\begin{equation*}
  \wt (\mathfrak{t}) = \left(\sum_{i=0}^{l+1} y_i - \sum_{i=0}^{l} z_i, \sum_{i=0}^{l+2} x_i - \sum_{i=0}^{l+1} y_i\right).
\end{equation*}
We also define an involution $t_\mathrm{short}$ on short patterns, acting like a Berenstein--Kirillov involution on the middle row.

Recall from \cref{prop:states-GTP-bijection} that mixed lattice model states are in bijection with Gelfand--Tsetlin patterns with corresponding row-pair types.
We will now show how the Berenstein--Kirillov involutions on Gelfand--Tsetlin patterns can be naturally transferred to involutions on lattice model states.
More precisely, for any $\Theta = (\Theta_1, \dots, \Theta_r)$ with $\Theta_i \neq \Theta_{i+1}$ we will show that we have a commutative diagram of bijections:
\begin{equation}
  \label{eq:big-commutative-diagram}
  \begin{tikzcd}[baseline=2mm]
    \mathfrak{S}_{\lambda + \rho}^\Theta  \arrow[r, leftrightarrow, "\iota_\Theta"] \arrow[d, leftrightarrow, dashed, "t_{r-i}"] 
    & \GTP_{\lambda+\rho}^{\Theta'}          \arrow[r, leftrightarrow, "\sh_{\Theta'}"] \arrow[d, leftrightarrow, "t_{r-i}"] 
    & \GTP_\lambda                        \arrow[r, leftrightarrow, "\tab"] \arrow[d, leftrightarrow, "t_{r-i}"] 
    & \SSYT(\lambda)                      \arrow[d, leftrightarrow, "\operatorname{BK}_{r-i}"] 
    \\
    \mathfrak{S}_{\lambda + \rho}^{s_i\Theta} \arrow[r, leftrightarrow, "\iota_{s_i\Theta}"] 
    & \GTP_{\lambda+\rho}^{(s_i\Theta)'} \arrow[r, leftrightarrow, "\sh_{(s_i\Theta)'}"]
    & \GTP_\lambda \arrow[r, leftrightarrow, "\tab"]
    & \SSYT(\lambda)   
  \end{tikzcd}
\end{equation}
Here the left-most horizontal bijection $\iota_\Theta$ was defined in \cref{prop:states-GTP-bijection}, the middle horizontal bijection $\sh_{\Theta'}$ will be defined below and proved to be a bijection in \cref{lemma:strict_non-strict_GTP_bijection}, and the right-most horizontal bijection is well-known.
Every $t_{r-i}$ preserves the top row of a Gelfand--Tsetlin pattern, and therefore $t_{r-i}$ maps $\GTP_\lambda$ to itself, giving the third vertical bijection.
In \cref{lemma:ti-GTPTheta_commutes} we will show that $t_{r-i}$ also maps the subset $\GTP_{\lambda+\rho}^{\Theta'}$ to $\GTP_{\lambda+\rho}^{(s_i\Theta)'}$, which gives the second vertical bijection, and that the middle square commutes.
Finally, we define the new dashed vertical map between lattice model states (which we will also call $t_{r-i}$) from the involution on Gelfand--Tsetlin patterns such that the left-most square commutes.

For a fixed list of row-pair types $\Theta' \in \{\Gamma, \Delta\}^{r-1}$ we define the map $\sh_{\Theta'} :\GTP_\lambda \to \GTP_{\lambda+\rho}^{\Theta'}$ as follows. 
For any $\mathfrak{T} \in \GTP_\lambda$ with entries $a_{i,j}$ as in~\eqref{eq:GTP}, we define the entries $b_{i,j}$ of $\sh_{\Theta'}(\mathfrak{T})$ to be
\begin{equation}
  b_{i,j} = a_{i,j} + \rho^{\Theta'}_{i,j},
\end{equation}
where $\rho^{\Theta'}$ is the unique element of $\GTP_{\lambda+\rho}^{\Theta'}$ with smallest possible non-negative entries.
This $\rho^{\Theta'}$ can be constructed by taking a copy of the row vector $\rho = (r-1,r-2,\ldots,0)$ for the first row, and for each successive row $i$ (starting from $1$) copy the entries on the previous row and truncate it by removing the left-most entry if $\Theta'_i = \Gamma$ and the right-most entry if $\Theta'_i = \Delta$.

\begin{example}
  Let $r = 4$ and $\Theta' = (\Gamma, \Delta, \Delta)$.
  Then $\rho^{\Theta'}$ is the Gelfand--Tsetlin pattern
  \begin{align*}
    \left\{ \begin{array}{ccccccc}  3 &  & 2 &  & 1 &  & 0\\  & 2 &  & 1 &  & 0 & \\  &  & 2 &  & 1 &  & \\  &  &  & 2 &  &  & \end{array} \right\}.
  \end{align*}
\end{example}

\begin{lemma}
  \label{lemma:strict_non-strict_GTP_bijection}
  The map $\sh_{\Theta'} :\GTP_\lambda \to \GTP_{\lambda+\rho}^{\Theta'}$ is a bijection. 
\end{lemma}

\begin{proof}
  The map $\sh_{\Theta'}$ is clearly injective, so it is sufficient to prove that it is also surjective, i.e.~that if $\mathfrak{T} \in \GTP_{\lambda+\rho}^{\Theta'}$ then $\mathfrak{T}-\rho^{\Theta'} \in \GTP_\lambda$.
  We thus need to show that, for any $1 \leq i \leq r-1$, the necessary Gelfand--Tsetlin inequalities for the row-pair $(i-1, i)$ in $\mathfrak{T}-\rho^{\Theta'}$ are satisfied. 
  We will present the case when $\Theta_i' = \Gamma$ in detail; the other case $\Theta_i' = \Delta$ is analogous.
  If we denote the entries of $\mathfrak{T}$ by $a_{i,j}$ as in~\eqref{eq:GTP}, then by definition we have that
  \begin{equation*}
    a_{i-1,j-1} > a_{i,j} \geq a_{i-1,j}.
  \end{equation*}
  Because the entries of $\mathfrak{T}$ are integers and $\rho^{\Theta'}_{i-1,j-1} = \rho^{\Theta'}_{i,j}+1$, it follows from the first inequality that $a_{i-1,j-1} - \rho^{\Theta'}_{i-1,j-1} \geq a_{i,j} - \rho^{\Theta'}_{i,j}$.
  Furthermore, the second inequality together with the fact that $\rho^{\Theta'}_{i,j} = \rho^{\Theta'}_{i-1,j}$ implies that $a_{i,j} - \rho^{\Theta'}_{i,j} \geq a_{i-1,j} - \rho^{\Theta'}_{i-1,j}$.
  Thus, the necessary Gelfand--Tsetlin inequalities for the row-pair $(i-1, i)$ in $\mathfrak{T}-\rho^{\Theta'}$ are satisfied.

  Together with the analogous arguments for the case $\Theta_i' = \Delta$ this show that $\mathfrak{T}-\rho^{\Theta'} \in \GTP_\lambda$.
\end{proof}

The following lemma shows how $t_{r-i}$ descends from a map on $\GTP_{\lambda+\rho}$ to a map between subsets $\GTP_{\lambda+\rho}^{\Theta'}$ and thus why the bijection $\sh_{\Theta'}$ is natural for our purposes.
\begin{lemma}
  \label{lemma:ti-GTPTheta_commutes}
  Let $\Theta \in \{\Gamma, \Delta\}^r$ and $\mathfrak{T} \in \GTP_{\lambda+\rho}^{\Theta'}$ where $\Theta' \in \{\Gamma, \Delta\}^{r-1}$ is the truncation of~$\Theta$. 
  Fix a row $1 \leq i \leq r-1$ such that $\Theta_i \neq \Theta_{i+1}$.
  Then $t_{r-i} \mathfrak{T} \in \GTP_{\lambda+\rho}^{(s_i \Theta)'}$ and we have the following commutative diagram
  \begin{equation}
  \label{eq:GTP-diagram}
  \begin{tikzcd}[baseline=2mm]
    \GTP_{\lambda+\rho}^{\Theta'}          \arrow[r, leftrightarrow, "\sh_{\Theta'}"'] \arrow[d, leftrightarrow, "t_{r-i}"] 
    & \GTP_\lambda                        \arrow[d, leftrightarrow, "t_{r-i}"] 
    \\
    \GTP_{\lambda+\rho}^{(s_i\Theta)'} \arrow[r, leftrightarrow, "\sh_{(s_i\Theta)'}"'] 
    & \GTP_\lambda 
  \end{tikzcd}
\end{equation}
\end{lemma}

\begin{proof}
  We will show the argument for $i < r - 1$. The case $i = r-1$ is the same but with some obvious modifications due to the fact that it is the last row.
  
  If $\mathfrak{T} \in \GTP_{\lambda+\rho}^{\Theta'}$, then we already know that $t_{r-i} \mathfrak{T} \in \GTP_{\lambda+\rho}$, but it remains to be shown that it has row-pair types $(s_i \Theta)'$.
  The involution $t_{r-i}$ only affects the entries of $\mathfrak{T}$ at the $i^\text{th}$ row, which means that it can only affect the $i^\text{th}$ and $(i+1)^\text{th}$ row-pair types.
  We can therefore focus on the rows indexed $i-1$, $i$ and $i+1$.
  Since $\Theta_i \neq \Theta_{i+1}$ there are two cases: First, if $\Theta_i = \Gamma$ and $\Theta_{i+1} = \Delta$, then the row-pair $(i-1,i)$ is left-strict and the row-pair $(i,i+1)$ is right-strict, and thus these three rows look as follows:
  \begin{equation}
    \label{eq:left-right-strict}
    \left\{
    \begin{array}{llllllllll}
      a_{i-1,i-1} &   &         &      & a_{i-1,i}   &   &           &      & a_{i-1,i+1} &        \\
                  & > &         & \geq &             & > &           & \geq &             &        \\
                  &   & a_{i,i} &      &             &   & a_{i,i+1} &      &             & \cdots \\
                  &   &         & \geq &             & > &           & \geq &             &        \\
                  &   &         &      & a_{i+1,i+1} &   &           &      & a_{i+1,i+2} &
    \end{array}
  \right\}
  \end{equation}

  Second, if $\Theta_i = \Delta$ and $\Theta_{i+1} = \Gamma$, then the first row-pair is right-strict and the second row-pair is left-strict, giving:
  \begin{equation}
    \label{eq:right-left-strict}
    \left\{
    \begin{array}{llllllllll}
      a_{i-1,i-1} &   &         &      & a_{i-1,i}   &   &           &      & a_{i-1,i+1} &        \\
                  & \geq &         & > &             & \geq &           & > &             &        \\
                  &   & a_{i,i} &      &             &   & a_{i,i+1} &      &             & \cdots \\
                  &   &         & > &             & \geq &           & > &             &        \\
                  &   &         &      & a_{i+1,i+1} &   &           &      & a_{i+1,i+2} &
    \end{array}
  \right\}
  \end{equation} 

  Recall that $t_{r-i}$ acts on an entry $a_{i,j}$ by reflecting it in its admissible interval, which (away from the boundary diagonals) is given by $[\min(a_{i-1,j-1}, a_{i+1,j}), \max(a_{i-1,j}, a_{i+1,j+1})]$ and on the boundary by a similar expression (see after \eqref{eq:ti}).
  Thus, the strict inequalities in \eqref{eq:left-right-strict} and \eqref{eq:right-left-strict} are also reflected and hence $t_{r-i}$ is a bijection between $\GTP_{\lambda+\rho}^{\Theta'}$ and $\GTP_{\lambda+\rho}^{(s_i\Theta)'}$.

  We will now prove that the diagram commutes for the case when $\Theta_i = \Gamma$ and $\Theta_{i+1} = \Delta$; the other case is similar.
  Start in the upper right corner of \eqref{eq:GTP-diagram} with some $\mathfrak{T} \in \GTP_\lambda$.
  For an entry $a_{i,j}$ not on the boundary (i.e. $i < j < r-1$), consider the subpattern of $\mathfrak{T}$ consisting of the entry $a_{i,j}$ together with its diagonally adjacent entries, i.e.
  \begin{equation}
    \label{eq:subpattern}
    \begin{array}{lll}
      a_{i-1,j-1}   &         & a_{i-1,j} \\
                  & a_{i,j} &         \\
      a_{i+1,j} &         & a_{i+1,j+1}
    \end{array}   
  \end{equation}
  Since $\Theta_i = \Gamma$ and $\Theta_{i+1} = \Delta$, applying the map $\sh_{\Theta'}$ to $\mathfrak{T}$ transforms this subpattern to
  \begin{equation*}
    \begin{array}{lll}
      a_{i-1,j-1}+(k+1)   &         & a_{i-1,j}+k \\
                  & a_{i,j}+k &         \\
      a_{i+1,j}+(k+1) &         & a_{i+1,j+1}+k
    \end{array}   
  \end{equation*}
  for some integer $k$ that depends on $i$, $j$ and $\Theta_1, \dots, \Theta_{i-1}$.
  By the definition in \eqref{eq:ti}, applying $t_{r-i}$ to $\mathfrak{T}$ turns the middle entry into
  \begin{align*}
    (a_{i,j} + k)' &= \min(a_{i-1,j-1}+(k+1), a_{i+1,j}+(k+1)) \\
                   &\quad + \max(a_{i-1,j}+k, a_{i+1,j+1}+k) - (a_{i,j}+k) \\
                   &= (k+1) + \min(a_{i-1,j-1}, a_{i+1,j}) + \max(a_{i-1,j}, a_{i+1,j+1}) - a_{i,j} \\
                   &= (k+1) + a_{i,j}'.
  \end{align*}

  If we instead go down and then left in the diagram, the middle entry of the subpattern in~\eqref{eq:subpattern} turns first into $a_{i,j}'$ and then into $a_{i,j}' + (k+1)$, where the integer $k$ is the same as the one above because $s_i$ leaves $\Theta_1, \dots, \Theta_{i-1}$ unchanged.
  
  If the entry $a_{i,j}$ is on a boundary, then $a_{i+1,j}$ or $a_{i+1,j+1}$ does not exist.
  In the definition of $t_{r-i}$ this is mirrored by removing the corresponding entry or entries appearing inside the $\min$ or $\max$ in the expression for the interval, and the proof follows analogously.
\end{proof}

As a result of \cref{prop:states-GTP-bijection,lemma:ti-GTPTheta_commutes}, we can transfer the Berenstein--Kirillov involution $t_{r-i}$ to a bijection $t_{r-i} : \mathfrak{S}^\Theta_{\lambda+\rho} \to \mathfrak{S}^{s_i\Theta}_{\lambda+\rho}$ of lattice model states, where $\lambda$ is any partition and $\Theta$ is any list of row-types satisfying $\Theta_i \neq \Theta_{i+1}$.
This completes the commutative diagram in \eqref{eq:big-commutative-diagram} and we give an example in~\cref{fig:Berenstein-Kirillov-example}.

\begin{figure}[htpb]
  \centering
\begin{equation*}
\begin{tikzpicture}[scale=0.5, baseline={([yshift=-\the\fontdimen22\textfont2]equalline)}]
    \coordinate (equalline) at ($ (0,1) $);
    \foreach \x in {0,...,4} {
      \pgfmathtruncatemacro{\colnum}{4-\x}
      \draw (\x, -1) -- (\x, 3) node[above]{$\scriptstyle\colnum$};
    }
    \foreach \y/\lbl in {0/$\Gamma$,1/$\Delta$,2/$\Gamma$} {
      \draw (-1, \y) node[left]{\lbl} -- (5, \y);
    }

    \draw[green, ultra thick] (2,3) -- (2,0) -- (5,0);
    \draw[blue, ultra thick] (0,3) -- (0,2) -- (5,2);
    \draw[red, ultra thick] (4,3) -- (4,1) -- (-1,1);
\end{tikzpicture}
\ 
\longleftrightarrow
\ 
\begin{Bsmallmatrix}
4 & & 2 & & 0 \\
& 2 & & 0 &  \\
& & 2 & &
\end{Bsmallmatrix}
\ 
\xleftrightarrow{t_{r-1}}
\ 
\begin{Bsmallmatrix}
4 & & 2 & & 0 \\
& 4 & & 2 &  \\
& & 2 & &
\end{Bsmallmatrix}
\ 
\longleftrightarrow
\ 
\begin{tikzpicture}[scale=0.5, baseline={([yshift=-\the\fontdimen22\textfont2]equalline)}]
    \coordinate (equalline) at ($ (0,1) $);
    \foreach \x in {0,...,4} {
      \pgfmathtruncatemacro{\colnum}{4-\x}
      \draw (\x, -1) -- (\x, 3) node[above]{$\scriptstyle\colnum$};
    }
    \foreach \y/\lbl in {0/$\Gamma$,1/$\Gamma$,2/$\Delta$} {
      \draw (-1, \y) node[left]{\lbl} -- (5, \y);
    }

    \draw[green, ultra thick] (2,3) -- (2,0) -- (5,0);
    \draw[blue, ultra thick] (0,3) -- (0,1) -- (5,1);
    \draw[red, ultra thick] (4,3) -- (4,2) -- (-1,2);
\end{tikzpicture}
\end{equation*}

\caption{An example of a mixed state in $\mathfrak{S}_{\lambda+\rho}^\Theta$ with row types $\Theta = (\Gamma, \Delta, \Gamma)$ and its associated Gelfand--Tsetlin pattern in $\GTP_{\lambda+\rho}^{\Theta'}$ and how these transform under the Berenstein--Kirillov involution $t_{r-1}$ where $r=3$ acting on the middle row to obtain a state with row types $s_1 \Theta = (\Delta, \Gamma, \Gamma)$.}
  \label{fig:Berenstein-Kirillov-example}
\end{figure}

We would also like to know how this new bijection $t_{r-i}$ on lattice model states changes the Boltzmann weights of the states, which is the content of the following theorem.

\begin{theorem}
  \label{thm:ti-state-weights}
  Let $\Theta \in \{\Gamma,\Delta\}^r$ with $\Theta_i \neq \Theta_{i+1}$.
  The bijection $t_{r-i} : \mathfrak{S}^\Theta_{\lambda+\rho} \to \mathfrak{S}^{s_i\Theta}_{\lambda+\rho}$ defined by the commutative diagram \eqref{eq:big-commutative-diagram} satisfies $\wt(t_{r-i}\mathfrak{s})(\mathbf{z}) = \wt(\mathfrak{s})(s_i \mathbf{z})$.
\end{theorem}
\noindent The theorem follows from \cref{prop:state-GTP-weight} and \cref{cor:ti_acts_by_si_on_GTP_weights} below.

\begin{lemma}
  \label{lem:short_pattern_ti_wt}
  If $\mathfrak{t}$ is a short pattern and $\wt(\mathfrak{t}) = (w_2, w_1)$, then $\wt(t_\mathrm{short}\mathfrak{t}) = (w_1, w_2)$.
\end{lemma}

\begin{proof}
  Let $d_0$, $d_1$ and $d_2$ be the sums of the entries on the first, second and third row of $\mathfrak{t}$ from top to bottom.
  Then by definition, $\wt(\mathfrak{t}) = (d_1-d_2,d_0-d_1) = (w_2,w_1)$.
  The short pattern $t_\mathrm{short} \mathfrak{t}$ will have the same row sum for the first and third row, but its second row sum might be different, so call it $d_1'$.

  If we label the entries as in~\cref{eq:short-pattern}, then by the definition of $t_\mathrm{short}$ we have that
  \begingroup
  \allowdisplaybreaks
  \begin{align*}
    d_1'
    = \sum_{j=0}^{l+1} y_j'
    &= \bigl(x_0 + \max (x_1, z_0) - y_0 \bigr) + \sum_{j = 1}^{l} \bigl(\min (x_j, z_{j-1}) + \max (x_{j+1}, z_j) - y_j \bigr) \\
    &\quad + \bigl(\min (x_{l+1}, z_l) + x_{l + 2} - y_{l+1} \bigr) \\
    &= x_0 + x_{l+2} + \sum_{j=0}^l \bigl(\min(x_{j+1}, z_j) + \max(x_{j+1}, z_j) \bigr) - \sum_{j=0}^{l+1} y_j \\
    &= x_0 + x_{l+2} + \sum_{j=0}^l (x_{j+1} + z_j) - \sum_{j=0}^{l+1} y_j \\
    &= \sum_{j=0}^{l+2} x_j + \sum_{j=0}^l z_j - \sum_{j=0}^{l+1} y_j
    = d_0 - d_1 + d_2.
  \end{align*}
  \endgroup
  Then
  \begin{align*}
    \wt (t_\mathrm{short} \mathfrak{t}) &= (d_1'-d_2, d_0-d_1') = (d_0-d_1, d_1-d_2) = (w_1, w_2). \qedhere
  \end{align*}
\end{proof}

\begin{corollary}
  \label{cor:ti_acts_by_si_on_GTP_weights}
  If $\mathfrak{T}$ is a Gelfand--Tsetlin pattern with $r$ rows, then
  \begin{equation*}
    \wt(t_{r-i}\mathfrak{T}) = s_{r-i} \wt(\mathfrak{T})
  \end{equation*}
  for any $1 \leq i < r$, where $s_{r-i} \in S_r$ is a simple transposition.
\end{corollary}

\begin{proof}
  If $i < r-1$, then this follows from \cref{lem:short_pattern_ti_wt} by looking at the short pattern consisting of row $i-1$, $i$ and $i+1$.
  Because of the order of the entries in $\wt(\mathfrak{T})$, where the first entry of the weight is the bottom row sum, we get $s_{r-i}$ instead of $s_i$.
  
  If $i = r-1$, then $t_{r-i} = t_1$ acts on the bottom row consisting of a single entry and the statement can be proved analogously to \cref{lem:short_pattern_ti_wt}.
  Indeed, consider the last two rows of $\mathfrak{T}$:
  \begin{equation*}
    \begin{array}{ccccc}
     \ddots &     & \vdots &     & \iddots \\
            & x_0 &        & x_1 &         \\
            &     & y_0    &     &
    \end{array}
  \end{equation*}
  The map $t_{1}$ will only change this by replacing the last entry with
  \begin{equation*}
    y_0' \coloneq x_0 + x_1 - y_0
  \end{equation*}
  and so the first two entries of $\wt(t_1\mathfrak{T})$ are
  \begin{align*}
    &y_0' = x_0 + x_1 - y_0 \\
    &(x_0 + x_1) - y_0' = y_0,
  \end{align*}
  which is the second and first entry of $\wt(\mathfrak{T})$, respectively.
\end{proof}

Together with \cref{prop:state-GTP-weight} this finishes the proof of \cref{thm:ti-state-weights}.

Let us now consider the bijection $\iota_\Delta^{-1} \circ \Sch \circ \iota_\Gamma : \mathfrak{S}^\Gamma_{\lambda+\rho} \to \mathfrak{S}^\Delta_{\lambda+\rho}$ from Gamma states to Delta states corresponding to the Schützenberger involution $\Sch = t_1 (t_2 t_1)(t_3 t_2 t_1) \cdots (t_{r-1} \cdots t_1)$ on Gelfand--Tsetlin patterns or tableaux.
We will abuse notation and often denote this map on states by $\Sch$ as well.
This is the bijection that we want to show is a state-by-state refinement of the Gamma-Delta duality (that is, the left-right duality in \cref{thm:left-right_duality} in the special case of crystal models).
It follows from \cref{prop:state-GTP-weight} and \cref{cor:ti_acts_by_si_on_GTP_weights} that the bijection is weight-respecting in the sense that
\begin{equation}
  \label{eq:Gamma-Delta_weight_respecting}
  \wt(\Sch(\mathfrak{s}))(\mathbf{z}) = \wt(\mathfrak{s})(w_0 \mathbf{z}).
\end{equation}
In order for the map $\Sch$ to be a state-by-state refinement of the Gamma-Delta duality we also need to show that the order of the colors in the horizontal boundary conditions is reversed, that is, $\Sch$ descends to a map $\mathfrak{S}^\Gamma_{\lambda+\rho,\sigma} \to \mathfrak{S}^\Delta_{\lambda+\rho,w_0\sigma}$ for every left boundary permutation $\sigma$ of the top boundary colors.
This can be deduced from the two algorithms for computing Lascoux--Schützenberger keys in \cite{BBBG:demazure}, and will be explained in \cref{sec:demazure}.

However, it turns out that even more is true: $\Sch: \mathfrak{S}^\Gamma_{\lambda+\rho} \to \mathfrak{S}^\Delta_{\lambda+\rho}$ can be factorized using the bijections $t_i$, and $\ell$ from \cref{lemma:mixed_model_last_row} in a way that precisely mirrors the proof of the left-right duality at the level of partition functions in \cref{thm:left-right_duality}.
Indeed, the factorization $\Sch = t_1 (t_2 t_1)(t_3 t_2 t_1) \cdots (t_{r-1} \cdots t_1)$ of the Schützenberger involution on Gelfand--Tsetlin patterns combined with \cref{eq:big-commutative-diagram} and the commutative diagram
\begin{equation}
  \label{eq:commutative_diagram_ell}
  \begin{tikzcd}[baseline=2mm, row sep=small]
    \mathfrak{S}_{\lambda + \rho}^{(\Theta_1, \dots, \Theta_{r-1}, \Gamma)}  \arrow[rd, "\iota_{(\Theta_1, \dots, \Theta_{r-1}, \Gamma)}"] \arrow[dd, "\ell"] & \\
    & \GTP_{\lambda+\rho}^{(\Theta_1, \dots, \Theta_{r-1})} \\
    \mathfrak{S}_{\lambda + \rho}^{(\Theta_1, \dots, \Theta_{r-1}, \Delta)} \arrow[ru, "\iota_{(\Theta_1, \dots, \Theta_{r-1}, \Delta)}"'] &
  \end{tikzcd}
\end{equation}
implies that
\begin{equation}
  \label{eq:Sch_factorization}
  \Sch = \ell \, t_1 \, \ell \, (t_2 \, t_1) \, \ell \, (t_3 \, t_2 \, t_1) \, \ell \cdots (t_{r-2} \cdots t_1) \, \ell \, (t_{r-1} \cdots t_1) \, \ell
\end{equation}
as a bijection $\mathfrak{S}^\Gamma_{\lambda+\rho} \to \mathfrak{S}^\Delta_{\lambda+\rho}$.
To summarize, $t_{r-i}$ swaps the types $\Theta_{i}$ and $\Theta_{i+1}$, and $\ell$ changes a bottom $\Gamma$-row into a bottom $\Delta$-row and combining them as in \eqref{eq:Sch_factorization} gives us a bijection between $\Gamma$ and $\Delta$ states where each $\Gamma$ row is brought to the bottom and then swapped to a $\Delta$ row.

With the scaled $\Delta$ weights used in this section $\ell$ is actually weight-preserving, and therefore~\eqref{eq:Gamma-Delta_weight_respecting} also follows from \cref{thm:ti-state-weights}, which states that $t_{r-i}$ acts on the weight by swapping $z_i$ and $z_{i+1}$.

Turning to the effect on $\sigma$, recall from \cref{lemma:mixed_model_last_row} that $\ell$ leaves $\sigma$ unchanged, and the following theorem shows that the maps $t_{r-i}$ swap the horizontal boundary colors on row $i$ and $i+1$.
Combined with \cref{thm:ti-state-weights} this means that they provide state-by-state refinements of \cref{lemma:mixed_model_exchange_rows}.

\begin{theorem}
  \label{thm:ti-boundary-conditions}
  The bijection $t_{r-i} : \mathfrak{S}^{\Theta}_{\lambda+\rho} \to \mathfrak{S}^{s_i\Theta}_{\lambda+\rho}$ $($where $\Theta_i \neq \Theta_{i+1})$ descends to a bijection $t_{r-i} : \mathfrak{S}^{\Theta}_{\lambda+\rho, \sigma} \to \mathfrak{S}^{s_i\Theta}_{\lambda+\rho,s_i\sigma}$.
\end{theorem}

The proof is postponed to \cref{sec:crystal-boundary-colors}.

\begin{corollary}[The state-by-state Gamma-Delta duality]
  \label{cor:Schützenberger_on_states}
  The map
  \begin{equation*}
    \Sch = \ell \, t_1 \, \ell \, (t_2 \, t_1) \, \ell \, (t_3 \, t_2 \, t_1) \, \ell \cdots (t_{r-2} \cdots t_1) \, \ell \, (t_{r-1} \cdots t_1) \, \ell : \mathfrak{S}^{\Gamma}_{\lambda+\rho, \sigma} \to \mathfrak{S}^{\Delta}_{\lambda+\rho,w_0\sigma}
  \end{equation*}
  is a bijection between Gamma states with boundary conditions given by $\lambda + \rho$ and $\sigma$, and Delta states with boundary conditions given by $\lambda + \rho$ and $w_0\sigma$.
  Furthermore, it is weight-respecting in the following sense:
  \begin{equation*}
    \wt(\Sch(\mathfrak{s}))(\mathbf{z}) = \wt(\mathfrak{s})(w_0 \mathbf{z}).
  \end{equation*}
\end{corollary}

To summarize, we have thus shown that the Schützenberger involution on states refines the Gamma-Delta duality to individual states.
Furthermore, the factorization of $\Sch$ in \cref{cor:Schützenberger_on_states} mirrors the steps taken in the proof of the duality at the level of partition functions in \cref{thm:left-right_duality}.
Recall that it built on two equalities of partition functions: the equality in \cref{lemma:mixed_model_exchange_rows} which was used to swap two rows of different row types, and the equality in \cref{lemma:mixed_model_last_row} which was used to change the row type of the last row.
\cref{thm:ti-boundary-conditions} shows that $t_{r-i}$ is a refinement of the row swap in \cref{lemma:mixed_model_exchange_rows} to individual states, and \cref{lemma:mixed_model_last_row} shows that $\ell : \mathfrak{S}_{\mu, \sigma}^{(\Theta_1,\ldots,\Theta_{r-1},\Gamma)} \to \mathfrak{S}_{\mu, \sigma}^{(\Theta_1,\ldots,\Theta_{r-1},\Delta)}$ is a refinement of the row type change of the last row to individual states.

In the proof of \cref{thm:left-right_duality} we then made a sequence of row swaps using \cref{lemma:mixed_model_exchange_rows} bringing each row in turn to the last row and changing its type with \cref{lemma:mixed_model_last_row}, in a way that exactly mirrors the factorization of $\Sch$.

\subsection{Demazure atoms and Lascoux--Schützenberger keys}
\label{sec:demazure}
In this subsection we will make a momentary shift of focus to Demazure atoms, before returning to the proof of \cref{thm:ti-boundary-conditions} in \cref{sec:crystal-boundary-colors}. 

Recall that crystal Demazure atoms were introduced in \cref{sec:introduction} as certain disjoint subsets $\mathcal{B}^{\circ}_{\lambda} (w)$ enumerated by Weyl elements $w \in S_r$ of the Kashiwara--Nakashima crystal $\mathcal{B}_{\lambda}$ consisting of semistandard Young tableaux of shape~$\lambda$.
As mentioned in \cref{sec:introduction}, it was shown in \cite{BBBG:demazure} that there is a bijection between the right-moving, $\Gamma$ crystal model states and the elements of these atoms. 
More specifically, let $\TG \coloneq \tab \circ \sh_{\Gamma'}^{-1} \circ \operatorname{\iota}_\Gamma : \mathfrak{S}_{\lambda + \rho}^{\Gamma} \rightarrow \SSYT(\lambda)$ be the bijection that takes a $\Gamma$-state to the corresponding tableau.
If the top boundary colors are given by $\sigma_0 = (c_r, \ldots, c_1)$ from left to right and we denote the right-boundary colors from top to bottom by $\sigma = w \sigma_0$, where $w \in S_r$, then the bijection is given by $\Sch \circ \TG : \mathfrak{S}_{\lambda + \rho, w \sigma_0}^{\Gamma} \rightarrow \mathcal{B}^{\circ}_{\lambda} (w)$ where $\Sch$ is the Schützenberger involution on tableaux.

Recall from the previous subsection that the state-by-state left-right duality is given by the Schützenberger involution.
Using the commutative diagram in \eqref{eq:big-commutative-diagram} we then get a more direct and natural bijection between the \emph{left-moving} crystal states and the Demazure atom as follows.

\begin{theorem}
  Let $\sigma_0$ be the tuple of decreasing colors $(c_r, \ldots, c_1)$.
  Let $\TD : \mathfrak{S}^\Delta_{\lambda + \rho} \to \SSYT(\lambda)$ be the map defined by reading off the occupied column numbers at each row of the state, subtracting $(\ldots,2,1,0)$ from each row and interpreting the resulting rows as consecutive shapes of a tableau.
  This is a bijection which for each $w \in S_r$ refines to a bijection $\mathfrak{S}^\Delta_{\lambda + \rho, w_0 w\sigma_0} \to \mathcal{B}_\lambda^\circ(w)$. 
\end{theorem}

\begin{proof}
  We have that $\TD = \tab \circ \sh_{\Delta'}^{-1} \circ \operatorname{\iota}_\Delta$, showing that it is a bijection.

  By Theorems 5.5 and 6.1 in \cite{BBBG:demazure}, $\Sch \circ \TG$ is a bijection from the $\Gamma$ states $\mathfrak{S}_{\lambda + \rho, w \sigma_0}^{\Gamma}$ to the tableaux of the Demazure atom $\mathcal{B}^{\circ}_{\lambda} (w)$ for every $w \in S_r$.
  By \cref{cor:Schützenberger_on_states}, $\Sch: \mathfrak{S}^{\Gamma}_{\lambda+\rho, w\sigma} \to \mathfrak{S}^{\Delta}_{\lambda+\rho,w_0w\sigma}$ is a bijection, and so $\Sch \circ \TG \circ \Sch^{-1}$ is a bijection from $\mathfrak{S}^{\Delta}_{\lambda+\rho, w_0 w\sigma_0}$ to $\mathcal{B}^{\circ}_{\lambda} (w)$.
  Now, by \eqref{eq:big-commutative-diagram} we have the following commutative diagram
  \begin{equation}
    \label{eq:GammaDelta-commutative-diagram-replicated}
    \begin{tikzcd}
      \mathfrak{S}^\Gamma_{\lambda + \rho}  \arrow[r, rightarrow, "\TG"] \arrow[d, rightarrow, "\text{Sch}"],
      & \SSYT(\lambda)                      \arrow[d, leftrightarrow, "\text{Sch}"] 
      \\
      \mathfrak{S}_{\lambda + \rho}^\Delta \arrow[r, rightarrow, "\TD"]
      & \SSYT(\lambda)   
    \end{tikzcd}
  \end{equation}
  and so $\Sch \circ \TG \circ \Sch = \TD:\mathfrak{S}^{\Delta}_{\lambda+\rho, w_0 w\sigma_0} \to \mathcal{B}^{\circ}_{\lambda} (w)$, which finishes the proof.
\end{proof}

In the rest of this subsection we will use results from~\cite{BBBG:demazure} to give an alternative proof of the first statement in \cref{cor:Schützenberger_on_states}, i.e.\ that $\Sch : \mathfrak{S}^{\Gamma}_{\lambda+\rho} \to \mathfrak{S}^{\Delta}_{\lambda+\rho}$ refines to a bijection respecting the boundary colors, without factorizing $\Sch$ into Berenstein--Kirillov involutions $t_{r-i}$ and using \cref{thm:ti-boundary-conditions} on each factor.

This alternative proof builds on the two algorithms developed in~\cite{BBBG:demazure} to determine the Demazure atom $\mathcal{B}_{\lambda}^{\circ} (w)$ a given tableau $T \in \SSYT(\lambda)$ belongs to.
More precisely, Algorithm~1 and 2 of~\cite{BBBG:demazure} give maps $A_1, A_2 : \SSYT(\lambda) \to S_r$ such that $T \in \mathcal{B}_{\lambda}^{\circ} (w_0 A_2(T))$ and $\Sch(T) \in \mathcal{B}_{\lambda}^{\circ} (w_0 A_1(T))$.
Note that $A_1 = A_2 \circ \Sch$.

By Theorem~6.1 of~\cite{BBBG:demazure} $A_1(T_\Gamma(\mathfrak{S}^\Gamma_{\lambda+\rho, w\sigma_0})) = w_0 w$, meaning that Algorithm~1 can be interpreted as a computation of the right boundary condition of a right-moving $\Gamma$ crystal state, and Algorithm~2 was derived from Algorithm~1 using manipulations of the crystal~$\mathcal{B}_\lambda$.
One can analogously show (see \cref{rem:circling} in~\cref{sec:crystal-boundary-colors}) that $A_2 (T_{\Delta} (\mathfrak{S}^{\Delta}_{\lambda + \rho, w\sigma_0})) = w$ meaning that Algorithm~2 can be interpreted as a computation of the left boundary condition of a left-moving $\Delta$ crystal state.
Thus we get in this paper an alternative, independent proof of Algorithm~2 without using the above mentioned crystal manipulations.
In \cref{sec:crystal-boundary-colors} we also give a similar algorithm for any mixed row type.

If $\mathfrak{s} \in \mathfrak{S}^\Gamma_{\lambda+\rho, w \sigma_0}$, then $\Sch(\mathfrak{s}) \in \mathfrak{S}^\Delta_{\lambda+\rho, w' \sigma_0}$ for some $w'$.
In fact, $w' = A_2 \circ \TD \circ \Sch(\mathfrak{s}) =  A_2 \circ \Sch \circ \TG (\mathfrak{s}) = A_1 \circ \TG(\mathfrak{s}) = w_0 w$ which completes the proof.

\subsection{Proof of Theorem \ref{thm:ti-boundary-conditions}}
\label{sec:crystal-boundary-colors}
In this section we will prove \cref{thm:ti-boundary-conditions} by showing that if $\mathfrak{s} \in \mathfrak{S}^{\Theta}_{\lambda+\rho, \sigma}$ then the horizontal boundary colors of $t_{r-i} (\mathfrak{s})$ are given by $s_i \sigma$.
Recall that the map $t_{r-i}$ on states is defined by the corresponding map on Gelfand--Tsetlin as illustrated in the commutative diagram~\eqref{eq:big-commutative-diagram}. 

So suppose that we have some state $\mathfrak{s} \in \mathfrak{S}^{\Theta}_{\lambda+\rho, \sigma}$ of a mixed crystal model with row types given by $\Theta$, top boundary determined by $\lambda + \rho$ and horizontal boundary colors given by $\sigma$.
In the proof we will need to keep track of how the colors of the paths are permuted at various intermediate steps between the global input and output boundary edges.
We will call an $r$-tuple of colors a \emph{flag}.
For any $i$ and $j$ satisfying either $0 \leq i < r$ and $0 \leq j \leq N$ or $i=r$ and $j=0$, consider the subsection of the lattice consisting of the first $i$ rows, along with the left-most $j$ vertices from the next row if this next row is of type $\Gamma$ and the right-most $j$ vertices from the next row if it is a $\Delta$ row.
Then define $\Sigma_{i,j}$ to be the flag consisting of the colors of the occupied edges on the output boundary of this subsection in counterclockwise order.
See \cref{fig:example_flags} for an example.
Note that we have $\Sigma_{i,N} = \Sigma_{i+1,0}$ for every $0 \leq i < r$.

We will use $\Sigma_i$ as short for $\Sigma_{i,0}$.
Note that $\Sigma_0$ and $\Sigma_r$ records the colors of the occupied edges on the input and output boundary, respectively, of the entire lattice.
Furthermore, $\Sigma_r$ and $\sigma$ determine each other, and in particular if $\Theta = \Delta$ (i.e.~left-moving), then $\sigma = \Sigma_r$, and if $\Theta = \Gamma$, then $\sigma = w_0 \Sigma_r$, where $w_0$ is the longest element in $S_r$ which acts on a flag by reversing it.

\begin{figure}
  \newcommand{\cred}{\textcolor{red}{c_1}}
  \newcommand{\cgreen}{\textcolor{green}{c_2}}
  \newcommand{\cblue}{\textcolor{blue}{c_3}}
  \begin{tikzpicture}[scale=1.2]
    \foreach \x in {0,...,4} {
      \draw (\x, -1) -- (\x, 3);
    }
    \foreach \y in {0,...,2} {
      \draw (-1, \y) -- (5, \y);
    }

    \draw[densely dashed, thick, -Stealth] (-0.4, 2.8) -- (-0.4, 1.7) -- (4.4, 1.7) -- (4.4, 2.8) node[above, anchor=south west, xshift=-2.5ex] {$\scriptstyle j: 0, 1, 2, 3 \text{ from left to right}$};
    \draw[densely dashed, thick, -Stealth] (-0.5, 2.8) -- (-0.5, 1.6) -- (3.5, 1.6) -- (3.5, 0.6) -- (4.5, 0.6) -- (4.5, 2.8);
    \draw[densely dashed, thick, -Stealth] (-0.6, 2.8) -- (-0.6, 1.5) -- (2.5, 1.5) -- (2.5, 0.5) -- (4.6, 0.5) -- (4.6, 2.8);
    \draw[densely dashed, thick, -Stealth] (-0.7, 2.8) -- (-0.7, 1.4) -- (1.5, 1.4) -- (1.5, 0.4) -- (4.7, 0.4) -- (4.7, 2.8);
    \draw (7, 2)   node {$\Sigma_{1,0} = (\cgreen, \cred, \cblue)$};
    \draw (7, 1.5) node {$\Sigma_{1,1} = (\cgreen, \cred, \cblue)$};
    \draw (7, 1)   node {$\Sigma_{1,2} = (\cgreen, \cred, \cblue)$};
    \draw (7, 0.5) node {$\Sigma_{1,3} = (\cred, \cgreen, \cblue)$};

    \newcommand{\lineWidth}{1}
    \draw[green, line width=\lineWidth mm] (2,3) -- (2,0) -- (5,0);
    \draw[blue, line width=\lineWidth mm] (0,3) -- (0,2) -- (5,2);
    \draw[red, line width=\lineWidth mm] (4,3) -- (4,1) -- (-1,1);
  \end{tikzpicture}
  \caption{An example of a mixed crystal state with $\Theta=(\Gamma, \Delta, \Gamma)$, $\sigma = (\cblue, \cred, \cgreen)$ along with some of its flags $\Sigma_{1,j}$.}
  \label{fig:example_flags}
\end{figure}

Recall from \cref{rem:unique_coloring} that given any uncolored state of a (mixed) crystal model and a coloring of its occupied top edges, there is a unique way to color the remaining occupied edges such that the resulting state becomes admissible.
For what follows we need some more details about this unique coloring, namely how $\Sigma_{i,j+1}$ depends on $\Sigma_{i,j}$.

We can see from \cref{tab:T_fused_gamma_Iwahori_crystal_weights} that on a $\Gamma$ row, when two paths meet, they cross if and only if the left color is larger than the top one.
Similarly, we can see from \cref{tab:T_fused_delta_Iwahori_crystal_weights} that on a $\Delta$ row, when two paths meet they cross if and only if the top color is larger than the right one.

Now suppose that we have a coloring of the subsection corresponding to $\Sigma_{i,j}$.
By definition, $\Sigma_{i,j+1}$ covers one extra vertex compared to $\Sigma_{i,j}$.
If there is only one path that goes through this vertex, then $\Sigma_{i,j+1} = \Sigma_{i,j}$.
On the other hand, if two paths meet at this vertex, then
\begin{equation}
  \Sigma_{i,j+1} =
  \begin{cases}
    s_k \Sigma_{i,j} & \text{if } (\Sigma_{i,j})_k > (\Sigma_{i,j})_{k+1} \\
    \Sigma_{i,j}     & \text{otherwise.}
  \end{cases}
  \label{eq:adjacent_Sigma_relation}
\end{equation}
where $k$ is the index of the first of these paths in $\Sigma_{i,j}$ (which means that the other path is indexed by $k+1$) and $s_k$ is the corresponding simple transposition.
In the first case the paths cross, in the second they do not.
To derive \eqref{eq:adjacent_Sigma_relation} note that if the row is of type $\Delta$ then the path of index $k$ in $\Sigma_{i,j}$ corresponds to the top edge and $k+1$ to the right edge while for type $\Gamma$ they are the left and top edges respectively. 

This defines an action of the \emph{Coxeter monoid} $\M_r$ corresponding to the symmetric group $S_r$, which is the monoid generated by elements $\m_1, \dots, \m_{n-1}$ subject to the relations
\begin{equation}
\begin{alignedat}{2}
  \m_i^2             & = \m_i                   &       & \text{for all } 1 \leq i < r   \\
  \m_i \m_{i+1} \m_i & = \m_{i+1} \m_i \m_{i+1} & \quad & \text{for all } 1 \leq i < r-1 \\
  \m_i \m_j          & = \m_j \m_i              &       & \text{when } |i-j| > 1.
  \label{eq:Coxeter_monoid_relations}
\end{alignedat}
\end{equation}
Note that it has the same number of generators as $S_r$ and that only the first relation is different.
In fact, taking an element with reduced expression $s_{i_1} \dots s_{i_k}$ in $S_r$ to the element $\m_{i_1} \dots \m_{i_k}$ in $\M_r$ gives a well-defined bijection between $S_r$ and $\M_r$ \cite{Tsaranov_Coxeter_monoids}.
In this way we can think of $\M_r$ as having the same elements as $S_r$ but with a different multiplication, known as the \emph{Demazure product}.
This monoid is also known as a \emph{$0$-Hecke monoid} or simply a \emph{Hecke monoid}, because the additive inverses of the generators of a degenerate Hecke algebra with $q=0$ generate a Coxeter monoid under multiplication.

Mirroring \cref{eq:adjacent_Sigma_relation}, we define an action of $\M_r$ on flags $\mathbf{d} = (d_1, \dots, d_r) \in \mathcal{P}^r$ by
\begin{equation}
  \m_i \cdot \mathbf{d} =
  \begin{cases}
    s_i \mathbf{d} & \text{if } d_i > d_{i + 1} \\
    \mathbf{d}     & \text{otherwise.}
  \end{cases}
  \label{eq:monoid_action}
\end{equation}
It is not difficult to check that this indeed is an action of $\M_r$.
We can then rewrite \cref{eq:adjacent_Sigma_relation} as
\begin{equation*}
  \Sigma_{i,j+1} = \m_k \Sigma_{i,j}.
\end{equation*}

In this way each vertex of type $\texttt{a}_2$ or $\texttt{a}_2'$ corresponds to acting on the flag with a certain generator $\m_k$ of the Coxeter monoid.

\begin{example}
  For the state in \cref{fig:example_flags}, the vertices of type $\texttt{a}_2$ and $\texttt{a}_2'$ correspond to the following equalities
  \begin{align*}
    \Sigma_{0,3} &= \m_1 \Sigma_{0,2} \\
    \Sigma_{0,5} &= \m_2 \Sigma_{0,4} \\
    \Sigma_{1,3} &= \m_1 \Sigma_{1,2},
  \end{align*}
  and $\Sigma_{i,j+1}$ is equal to $\Sigma_{i,j}$ in all other cases.
\end{example}

We can also express this in terms of the associated Gelfand--Tsetlin pattern $\iota_\Theta(\mathfrak{s}) \in \GTP_{\lambda+\rho}^{\Theta'}$.
Recall that a $\Gamma$ row-pair is left-strict, and is therefore of the following form:
\begin{equation*}
  \left\{
    \begin{array}{llllllllllllll}
      & x_0 &            &     &          & x_1 &            &     &        & x_k &            &     &          & x_{k+1} \\
      {\color{lightgray}\Gamma:} &  & \rcw{>} &     & \rccw{\geq} &     & \rcw{>} &     & \cdots &     & \rcw{>} &     & \rccw{\geq} &         \\
      &    &            & y_0 &          &     &            & y_1 &        &     &            & y_k &          &
    \end{array}
  \right\}.
\end{equation*}
The $\texttt{a}_2$ and $\texttt{a}_2'$ vertices in the lattice model correspond to non-strict inequalities that are actually equalities in the Gelfand--Tsetlin pattern.
We will therefore decorate the Gelfand--Tsetlin pattern by highlighting these equalities, and we will also identify them with monoid generators as follows:
\begin{equation}
  \left\{
    {\color{lightgray}
      \begin{array}{llllllllllllll}
        & x_0 &         &     &              & x_1 &         &     &        & x_k &         &     &                & x_{k+1} \\
        \Gamma: &    & \rcw{>} &     & \cb \m_{l+1} &     & \rcw{>} &     & \cdots &     & \rcw{>} &     & \cb \m_{l+k+1} &         \\
        & &         & y_0 &              &     &         & y_1 &        &     &         & y_k &                &
      \end{array}
    }
  \right\},
  \label{eq:highlighted_gamma_row}
\end{equation}
where $l$ is the number of $\Delta$ row-pairs above this one (i.e.\ the number of paths on the left boundary above this row).
Then the flags $\Sigma_{i+1}$ and $\Sigma_i$ are related via the action of the highlighted monoid generators, applied sequentially from left to right, i.e.
\begin{equation*}
  \Sigma_{i+1} = \m_{l+k+1} \cdots \m_{l+1} \Sigma_i.
\end{equation*}

On the other hand, a $\Delta$ row-pair is right-strict and has the following form:
\begin{equation*}
  \left\{
    \begin{array}{llllllllllllll}
      & x_0 &            &     &          & x_1 &            &     &        & x_k &            &     &          & x_{k+1} \\
      {\color{lightgray}\Delta:} &    & \rcw{\geq} &     & \rccw{>} &     & \rcw{\geq} &     & \cdots &     & \rcw{\geq} &     & \rccw{>} &         \\
      &    &            & y_0 &          &     &            & y_1 &        &     &            & y_k &          &
    \end{array}
  \right\}
\end{equation*}
which we can decorate analogously:
\begin{equation}
  \left\{
    {\color{lightgray}
      \begin{array}{llllllllllllll}
        & x_0 &              &     &          & x_1 &                         &     &        & x_k &                &     &          & x_{k+1} \\
        \Delta: &    & \cb \m_{l+1} &     & \rccw{>} &     & \color{black}{\m_{l+2}} &     & \cdots &     & \cb \m_{l+k+1} &     & \rccw{>} &           \\
        &    &              & y_0 &          &     &                         & y_1 &        &     &                & y_k &          &
      \end{array}
    }
  \right\},
  \label{eq:highlighted_delta_row}
\end{equation}
where again $l$ is the number of $\Delta$ row-pairs above this one.
Then, similarly to the $\Gamma$ case, the flags $\Sigma_{i+1}$ and $\Sigma_i$ are related via the action of the highlighted monoid generators, applied sequentially from \emph{right to left}, i.e.
\begin{equation*}
  \Sigma_{i+1} = \m_{l+1} \dots \m_{l+k+1} \Sigma_i.
\end{equation*}

\begin{remark}
  \label{rem:circling}
  This highlighting is equivalent to the boxes used in \cite{BBF:orange} and the circles used in \cite{BBBG:demazure}, except that we are highlighting the non-strict inequalities instead of boxing/circling the corresponding element below the inequality.\footnote{Note that while we consider $\GTP_{\lambda+\rho}^{\Theta'}$, in \cite{BBBG:demazure} the associated pattern in $\GTP_{\lambda}$ is considered, but this does not affect the highlighting.}

  Note also that the nondescending product $\Pi_\text{ND}(s_{i_1}, \dots, s_{i_k})$ of simple reflections used in \cite{BBBG:demazure} is the same as computing the product of the corresponding generators $\m_{i_1}, \dots, \m_{i_k}$ of the Coxeter monoid $\M_r$ and then taking the corresponding element of the symmetric group $S_r$.
  This can most easily be seen from the alternative description in \cite[Remark 1.1]{BBBG:demazure}, where the generators $S_i$ (not to be confused with the symmetric group $S_r$) of the degenerate Hecke algebra $\mathcal{H}$ satisfy the same relations as the generators $\m_i$ of the Coxeter monoid $\M_r$.

  In \cite{BBBG:demazure} an algorithm was described for computing the boundary colors from a Gelfand--Tsetlin pattern for a $\Gamma$-state called Algorithm~1 and which agrees with the procedure described above.
  We have here shown that the associated Algorithm~2 in the same paper also has an interpretation in terms of boundary colors but for the new $\Delta$-states constructed here.
\end{remark}

Since we are only interested in the highlighting of the monoid generators, we will usually simplify the diagrams by removing everything else, and use an arrow to record in which order the monoid generators should be read as a word (i.e.~the opposite order of how they should be applied).
For example, \eqref{eq:highlighted_gamma_row} would become
\begin{equation*}
  \begin{tikzpicture}[sipattern]
    \matrix (table) {
      \hc $\m_{l+1}$ & $\m_{l+2}$ & \dots & \hc $\m_{l+k+1}$ & \hspace{0.7em} \\
    };
    \draw[arrow] (table-1-5.east) to (table-1-4.east);
  \end{tikzpicture}
\end{equation*}
and \eqref{eq:highlighted_delta_row} would become
\begin{equation*}
  \begin{tikzpicture}[sipattern]
    \matrix (table) {
      \hc $\m_{l+1}$ & \dots &  \hc $\m_{l+k+1}$ & \hspace{0.7em} \\
    };
    \draw[arrow] (table-1-3.east) to (table-1-4.east);
  \end{tikzpicture}.
\end{equation*}

\begin{example}
  The Gelfand--Tsetlin pattern corresponding to the state in \cref{fig:example_flags} with row-pair types $\Theta' = (\Gamma, \Delta)$ is
  \begin{equation}
    \left\{
      \begin{array}{ccccccc}
        4 &   & 2 &   & 0 \\
          & 2 &   & 0 &   \\
          &   & 2 &   &   \\
      \end{array}
    \right\}.
  \end{equation}
  We decorate the pattern according to the row-pairs $\Gamma$ and $\Delta$ and highlight equalities as described above:
  \begin{equation} 
  \begin{tikzpicture}[sipattern]
    \matrix (table) {
      & |[color=lightgray]| $4$ &  &  &  & |[color=lightgray]|$2$ &  &  &  &|[color=lightgray]| $0$ & \\
      |[color=lightgray]|$\Gamma$: &  &|[color=lightgray]| $\rcw{>}$ &  & \hc $\m_1$ &  & |[color=lightgray]|$\rcw{>}$ &  & \hc $\m_2$ &  & \draw[arrow, yshift=.5ex] (0.7,0) to (0,0);\\
    &  &  &|[color=lightgray]| $2$ &  &  &  &|[color=lightgray]| $0$ &  & \\
    |[color=lightgray]|$\Delta$: &  &  &  & \hc $\m_1$ &  &|[color=lightgray]| $\rccw{>}$ &  &  & & \draw[arrow, yshift=.5ex] (0,0) to (0.7,0);\\
    &  &  &  &  &|[color=lightgray]| $2$ &  &  &  & &\\
  };
  \end{tikzpicture}
\end{equation}
  The product of the highlighted elements should be taken from left to right on a $\Delta$ row and from right to left on a $\Gamma$ row as signified by the arrows.
  We simplify this into a corresponding pattern of simple transpositions (where we combine the two arrows to a single curved arrow): 
  \begin{equation}
    \begin{tikzpicture}[sipattern]
      \matrix (table) {
        \hc $\m_1$ & \hc $\m_2$ & \\
        \hc $\m_1$ &           & \\
      };
      \draw[left arrow] (table-2-3.west) to (table-1-3.west);
    \end{tikzpicture},
  \end{equation}
  which gives us the following equalities
  \begin{align*}
    \Sigma_{1} &= \m_2 \m_1 \Sigma_{0} \\
    \Sigma_{2} &= \m_1 \Sigma_{1}.
  \end{align*}
\end{example}

Now suppose that we have some mixed state $\mathfrak{s} \in \mathfrak{S}^{\Theta}_{\lambda+\rho}$ and let $i$ be such that $\Theta_i \neq \Theta_{i+1}$.
This means that we can apply $t_{r-i}$ to $\mathfrak{s}$, and what we want to show is that the boundary colors of $t_{r-i}(\mathfrak{s})$ are given by $s_i \sigma$.
By the definition of $\Sigma_r$, this is equivalent to showing that $\Sigma_r$ is the same for both $\mathfrak{s}$ and $t_{r-i}(\mathfrak{s})$.

Since the top boundary is the same in both cases, it is enough to show that the two elements in the Coxeter monoid $\M_r$ taking $\Sigma_0$ to $\Sigma_r$ are equal.
These elements can be computed from the corresponding Gelfand--Tsetlin patterns, $\mathfrak{T}$ and $t_{r-i}(\mathfrak{T})$.
Since the patterns can only differ on row $i$ and we are only interested in the inequalities between the rows, we only need to consider the short patterns $\mathfrak{t}$ and $t_\mathrm{short}(\mathfrak{t})$ consisting of rows $i-1$, $i$ and $i+1$.
If we can show that the two elements of $\M_r$ corresponding to $\mathfrak{t}$ and $t_\mathrm{short}(\mathfrak{t})$ are equal we are done, as the pre- and post-factors corresponding to the rows below and above the short patterns are automatically equal.

A $\Delta$ row followed by a $\Gamma$ row can be described by a short pattern of the following form involving rows $i-1$, $i$ and $i+1$:
\begin{equation}
  \label{eq:short-Delta-Gamma}
  \left\{
    \begin{array}{llllllllllllllll}
      & x_0 &            &     &          & x_1 &             &     &          & x_2 & \cdots & x_{k+1} &             &         &          & x_{k+2} \\
      {\color{lightgray}\Delta:} &    & \rcw{\geq} &     & \rccw{>} &     & \rcw{\geq}  &     & \rccw{>} &     &        &         & \rcw{\geq}  &         & \rccw{>} &         \\
      &    &            & y_0 &          &     &             & y_1 &          &     & \cdots &         &             & y_{k+1} &          &         \\
      {\color{lightgray}\Gamma:} &    &            &     & \rcw{>}  &     & \rccw{\geq} &     & \rcw{>}  &     &        &         & \rccw{\geq} &         &          &         \\
      &    &            &     &          & z_0 &             &     &          & z_1 & \cdots & z_k     &             &         &          &
    \end{array}
  \right\}.
\end{equation}
Similarly, a $\Gamma$ row followed by a $\Delta$ row can be described by a short pattern of the following form:
\begin{equation}
  \label{eq:short-Gamma-Delta}
  \left\{
    \begin{array}{llllllllllllllll}
      &x_0 &         &     &             & x_1 &          &     &             & x_2 & \cdots & x_{k+1} &          &         &             & x_{k+2} \\
      {\color{lightgray}\Gamma:} &    & \rcw{>} &     & \rccw{\geq} &     & \rcw{>}  &     & \rccw{\geq} &     &        &         & \rcw{>}  &         & \rccw{\geq} &         \\
      &    &         & y_0 &             &     &          & y_1 &             &     & \cdots &         &          & y_{k+1} &             &         \\
      {\color{lightgray}\Delta:}&    &         &     & \rcw{\geq}  &     & \rccw{>} &     & \rcw{\geq}  &     &        &         & \rccw{>} &         &             &         \\
      &    &         &     &             & z_0 &          &     &             & z_1 & \cdots & z_k     &          &         &             &
    \end{array}
  \right\}.
\end{equation}
The map $t_\mathrm{short}$ is a bijection between the sets of these short patterns, and since it is an involution we may without loss of generality assume that $\mathfrak{t}$ is of the first kind.
Note that $t_\mathrm{short}$ will only change the middle row, and we can therefore denote the entries of $\mathfrak{t}$ by $x_i$, $y_i$ and $z_i$ and the entries of $t_\mathrm{short}(\mathfrak{t})$ by $x_i$, $y_i'$ and $z_i$.

Now decorate both $\mathfrak{t}$ and $t_\mathrm{short}(\mathfrak{t})$ as before and remove everything except the monoid generators $\m_k$, which gives us patterns of the following forms (shown here with example highlights):
\begin{equation}
  \label{eq:short-Delta-Gamma-highlights}
  \mDG =
  \begin{tikzpicture}[sipattern]
    \matrix (table) {
      & \hc $\m_{l+1}$ & \hc $\m_{l+2}$ & \hc $\m_{l+3}$ & \hspace{1.4em} & \hc $\m_{l+k}$ & \hc $\m_{l+k+1}$ & $\m_{l+k+2}$     \\
      &                & $\m_{l+2}$     & \hc $\m_{l+3}$ &                & \hc $\m_{l+k}$ & $\m_{l+k+1}$     & \hc $\m_{l+k+2}$ \\
    };
    \draw[right arrow] (table-2-1.east) to (table-1-1.east);
    \path (table-1-5.center) -- (table-2-5.center) node[midway] (A) {\dots};
  \end{tikzpicture}
\end{equation}
and
\begin{equation}
  \label{eq:short-Gamma-Delta-highlights}
  \mGD =
  \begin{tikzpicture}[sipattern]
    \matrix (table) {
          $\m_{l+1}$ & \hc $\m_{l+2}$ & \hc $\m_{l+3}$ & \hspace{1.4em} & $\m_{l+k}$     & \hc $\m_{l+k+1}$ & \hc $\m_{l+k+2}$ &  \\
      \hc $\m_{l+1}$ & \hc $\m_{l+2}$ & $\m_{l+3}$     &                & \hc $\m_{l+k}$ & $\m_{l+k+1}$     &                  &  \\
    };
    \draw[left arrow] (table-2-8.west) to (table-1-8.west);
    \path (table-1-4.center) -- (table-2-4.center) node[midway] (A) {\dots};
  \end{tikzpicture},
\end{equation}
where $l$ is the number of $\Delta$ rows above row $i$.

Note that in the case $i=r-1$, the last rows of \eqref{eq:short-Delta-Gamma} and \eqref{eq:short-Gamma-Delta} do not exist and therefore the resulting patterns in \eqref{eq:short-Delta-Gamma-highlights} and \eqref{eq:short-Gamma-Delta-highlights} would only consist of a single column with only the top generator.

With this notation, the proof of \cref{thm:ti-boundary-conditions} reduces to proving that the element in $\M_r$ obtained by taking the product of the highlighted elements in the order indicated by the arrow is the same for both $\mDG$ and $\mGD$.

Since the submonoid generated by $\m_{l+1}, \m_{l+2}, \dots, \m_{l+k+2}$ is isomorphic to the one generated by $\m_1, \m_2, \dots, \m_{k+2}$, we can, without loss of generality, assume that $l = 0$.

\textbf{Blocks:}
Note that a column where neither of the $\m_i$ are highlighted splits the pattern into two parts, where everything on the left commutes with everything on the right.
Furthermore, if the $i^\text{th}$ column in $\mDG$ is such a column, then so is the $i^\text{th}$ column in $\mGD$.
To see this, note that if neither $\m_i$ is highlighted in $\mDG$, then $\min(x_i, z_{i-1}) > y_i$ and therefore $y_i' > \max(x_{i+1}, z_i)$, which means that neither $\m_i$ is highlighted in $\mGD$

The columns with no highlighted entries therefore splits $\mDG$ and $\mGD$ into \emph{blocks}, such that any two elements in different blocks commute and the blocks in $\mDG$ and $\mGD$ cover the same columns.

\begin{example}
  \mbox{}\label{example:blocks}  
  The following pattern consists of three blocks:
  \newcommand{\sqbracketbelow}[3]{%
    \draw[thick]
    ([yshift=-5pt]#1.south west) --
    ([yshift=-9pt]#1.south west) --
    ([yshift=-9pt]#2.south east) --
    ([yshift=-5pt]#2.south east);
    \node at ($([yshift=-16pt]#1.south west)!0.5!([yshift=-16pt]#2.south east)$) {#3};
  }
  \begin{equation*}
    \begin{tikzpicture}[sipattern]
      \matrix (table) {
        & \hc $\m_1$ &     $\m_2$ & \hc $\m_3$ & $\m_4$ & \hc $\m_5$ & \hc $\m_6$ & $\m_7$ &     $\m_8$ & \hc $\m_9$ \\
        &            & \hc $\m_2$ & \hc $\m_3$ & $\m_4$ & \hc $\m_5$ &     $\m_6$ & $\m_7$ & \hc $\m_8$ & \hc $\m_9$ \\
      };
      \draw[right arrow] (table-2-1.east) to (table-1-1.east);

      \sqbracketbelow{table-2-2}{table-2-4}{\footnotesize Block 1}
      \sqbracketbelow{table-2-6}{table-2-7}{\footnotesize Block 2}
      \sqbracketbelow{table-2-9}{table-2-10}{\footnotesize Block 3}
    \end{tikzpicture}
  \end{equation*} 
  and after applying $t_\mathrm{short}$ we obtain another pattern with blocks in the same columns:
  \begin{equation*}
    \begin{tikzpicture}[sipattern]
      \matrix (table) {
         \hc $\m_1$ & \hc $\m_2$ & \hc $\m_3$ & $\m_4$ &     $\m_5$ & \hc $\m_6$ & $\m_7$ & \hc $\m_8$ & \hc $\m_9$ &  \\
             $\m_1$ & \hc $\m_2$ &     $\m_3$ & $\m_4$ & \hc $\m_5$ &     $\m_6$ & $\m_7$ & \hc $\m_8$ &            &  \\
      };
      \draw[left arrow] (table-2-10.west) to (table-1-10.west);

      \sqbracketbelow{table-2-1}{table-2-3}{\footnotesize Block 1}
      \sqbracketbelow{table-2-5}{table-2-6}{\footnotesize Block 2}
      \sqbracketbelow{table-2-8}{table-2-9}{\footnotesize Block 3}
    \end{tikzpicture}.
  \end{equation*}
\end{example}

\begin{lemma}
  \label{lemma:short_GTP_highlighting_gamma_delta}
  Let $i$ be such that column $i$ and $i+1$ are both inside a single block.
  Then the highlighting pattern in column $i$ of $\mGD$ is determined by that in column $i+1$ in $\mDG$ as follows:
  \begin{enumerate}
    \item if both of the $\m_{i+1}$ are highlighted in $\mDG$, then the same is true for the $s_i$ in $\mGD$; \label[case]{case:si-double}

    \item if only the top $\m_{i+1}$ is highlighted in $\mDG$, then only the bottom $s_i$ is highlighted in $\mGD$; \label[case]{case:si-top}

    \item if only the bottom $\m_{i+1}$ is highlighted in $\mDG$, then only the top $s_i$ is highlighted in $\mGD$. \label[case]{case:si-bottom}
  \end{enumerate}
\end{lemma}

\begin{proof}
  First note that since column $i+1$ is inside a block, these three cases cover all possibilities.

  If we are in \cref*{case:si-top}, i.e.~only the top $\m_{i+1}$ is highlighted in $\mDG$, then we must have $y_{i+1} = x_{i+1} < z_i$.
  Now since column $i$ is also inside a block, at least one of the $\m_i$ in $\mGD$ must be highlighted and therefore $y_i' = \max(x_{i+1}, z_i) = z_i$, which means that only the bottom $\m_i$ is highlighted in $\mGD$.

  The other two cases are similar.
\end{proof}

\Cref{lemma:short_GTP_highlighting_gamma_delta} tells us that $\mDG$ and $\mGD$ almost determine each other, except for the left-most column of the blocks in $\mDG$ and the right-most column of the blocks in $\mGD$.
However, every way of highlighting these columns (such that at least one entry is highlighted) is equivalent.
To see this, first note that the left-most column of a block in $\mDG$ has at least one highlighted entry.
The elements in this column also commute with every highlighted entry to the left, and we can therefore write them next to each other in the product.
Then, because of the quadratic relation $\m_i^2 = \m_i$ in $\M_r$, all three possible highlighting patterns (top, bottom, both) are equivalent.
We can therefore assume that only the top entry is highlighted, and to remember this special property of the left-most column in a block in $\mDG$ we will remove the bottom entry.
We will do the same thing with the right-most column of the blocks in $\mGD$.

\begin{example}
  In \cref{example:blocks}, the patterns would become
  \begin{equation*}
    \begin{tikzpicture}[sipattern]
      \matrix (table) {
        & \hc $\m_1$ &     $\m_2$ & \hc $\m_3$ & $\m_4$ & \hc $\m_5$ & \hc $\m_6$ & $\m_7$ &     $\m_8$ & \hc $\m_9$ \\
        &            & \hc $\m_2$ & \hc $\m_3$ & $\m_4$ & \hc $\m_5$ &     $\m_6$ & $\m_7$ & \hc $\m_8$ & \hc $\m_9$ \\
      };
      \draw[right arrow] (table-2-1.east)   to (table-1-1.east);
    \end{tikzpicture}
    \rightarrow
    \begin{tikzpicture}[sipattern]
      \matrix (table) {
        & \hc $\m_1$ &     $\m_2$ & \hc $\m_3$ & $\m_4$ & \hc $\m_5$ & \hc $\m_6$ & $\m_7$ & \hc $\m_8$ & \hc $\m_9$ \\
        &            & \hc $\m_2$ & \hc $\m_3$ & $\m_4$ &            &     $\m_6$ & $\m_7$ &            & \hc $\m_9$ \\
      };
      \draw[right arrow] (table-2-1.east)   to (table-1-1.east);
    \end{tikzpicture}
  \end{equation*}
  and
  \begin{equation*}
    \begin{tikzpicture}[sipattern]
      \matrix (table) {
         \hc $\m_1$ & \hc $\m_2$ & \hc $\m_3$ & $\m_4$ &     $\m_5$ & \hc $\m_6$ & $\m_7$ & \hc $\m_8$ & \hc $\m_9$ &  \\
             $\m_1$ & \hc $\m_2$ &     $\m_3$ & $\m_4$ & \hc $\m_5$ &     $\m_6$ & $\m_7$ & \hc $\m_8$ &            &  \\
      };
      \draw[left arrow] (table-2-10.west) to (table-1-10.west);
    \end{tikzpicture}
    \rightarrow
    \begin{tikzpicture}[sipattern]
      \matrix (table) {
         \hc $\m_1$ & \hc $\m_2$ & \hc $\m_3$ & $\m_4$ &     $\m_5$ & \hc $\m_6$ & $\m_7$ & \hc $\m_8$ & \hc $\m_9$ &  \\
             $\m_1$ & \hc $\m_2$ &            & $\m_4$ & \hc $\m_5$ &            & $\m_7$ & \hc $\m_8$ &            &  \\
      };
      \draw[left arrow] (table-2-10.west) to (table-1-10.west);
    \end{tikzpicture}.
  \end{equation*}
\end{example}

The idea now is to step-by-step transform $\mDG$ into $\mGD$ in such a way that the corresponding product in $\M_r$ is preserved at each step.

For these intermediate steps we will need to generalize the patterns above to patterns of the following form:
\begin{equation*}
  \begin{tikzpicture}[sipattern]
    \newcommand{\centralLabel}{$\m_i$}
    \matrix (table) {
      \hc $\m_1$ & \hc $\m_2$ & \ldots &     $\m_{i-1}$ & \emptycol & \hc $\m_{i+1}$ & \ldots & \hc $\m_{k+2}$ \\
      \hc $\m_1$ & \hc $\m_2$ & \ldots & \hc $\m_{i-1}$ & \emptycol & \hc $\m_{i+1}$ & \ldots &     $\m_{k+2}$ \\
    };
    \path              (table-1-5.center) -- (table-2-5.center) node[midway,highlighted entry] (A) {\centralLabel};
    \draw[left arrow]  (table-2-5.west)   to (table-1-5.west);
    \draw[right arrow] (table-2-5.east)   to (table-1-5.east);
  \end{tikzpicture}.
\end{equation*}
The corresponding element of $\M_r$ is obtained by taking the product of the highlighted entries on the second row inwards towards the arrows, then the central entry (if it is highlighted), and finally the highlighted entries on the first row outwards, away from the arrows.
Note that it does not matter if we start by multiplying all highlighted entries on the left-hand side of the second row and then all highlighted entries on the right-hand side, or if we start with the right-hand side and continue with the left-hand side, or if we alternate in some fashion, as every element on the left-hand side commute with every element on the right-hand side.

\begin{example}
  The element of $\M_6$ corresponding to
  \begin{equation*}
    \begin{tikzpicture}[sipattern]
      \newcommand{\centralLabel}{$\m_3$}
      \matrix (table) {
            $\m_1$ & \hc $\m_2$ & \emptycol &     $\m_4$ & \hc $\m_5$ \\
        \hc $\m_1$ &     $\m_2$ & \emptycol & \hc $\m_4$ &     $\m_5$ \\
      };
      \path              (table-1-3.center) -- (table-2-3.center) node[midway,highlighted entry] (A) {\centralLabel};
      \draw[left arrow]  (table-2-3.west)   to (table-1-3.west);
      \draw[right arrow] (table-2-3.east)   to (table-1-3.east);
    \end{tikzpicture}
  \end{equation*}
  can be computed in either of the following equivalent ways
  \begin{align*}
        &\m_1 \m_4 \m_3 \m_2 \m_5 \\
    ={} &\m_1 \m_4 \m_3 \m_5 \m_2 \\
    ={} &\m_4 \m_1 \m_3 \m_2 \m_5 \\
    ={} &\m_4 \m_1 \m_3 \m_5 \m_2.
  \end{align*}
\end{example}

The patterns $\mDG$ and $\mGD$ can also be written in this form:
\begin{equation*}
  \mDG =
  \begin{tikzpicture}[sipattern]
    \matrix (table) {
      & \hc $\m_1$ & \hc $\m_2$ & \hspace{1.4em} & \hc $\m_{k+1}$ &     $\m_{k+2}$ \\
      &            &     $\m_2$ &                &     $\m_{k+1}$ & \hc $\m_{k+2}$ \\
    };
    \draw[right arrow] (table-2-1.east) to (table-1-1.east);
    \path (table-1-4.center) -- (table-2-4.center) node[midway] (A) {\dots};
  \end{tikzpicture}
  =
  \begin{tikzpicture}[sipattern]
    \newcommand{\centralLabel}{$\m_1$}
    \matrix (table) {
      \emptycol & \hc $\m_2$ & \hspace{1.4em} & \hc $\m_{k+1}$ &     $\m_{k+2}$ \\
      \emptycol &     $\m_2$ &                &     $\m_{k+1}$ & \hc $\m_{k+2}$ \\
    };
    \path              (table-1-1.center) -- (table-2-1.center) node[midway,highlighted entry] (A) {\centralLabel};
    \draw[left arrow]  (table-2-1.west)   to (table-1-1.west);
    \draw[right arrow] (table-2-1.east)   to (table-1-1.east);

    \path (table-1-3.center) -- (table-2-3.center) node[midway] (A) {\dots};
  \end{tikzpicture}
\end{equation*}
and
\begin{equation*}
  \mGD =
  \begin{tikzpicture}[sipattern]
    \matrix (table) {
          $\m_1$ & \hc $\m_2$ & \hspace{1.4em} & \hc $\m_{k+1}$ & \hc $\m_{k+2}$ &  \\
      \hc $\m_1$ & \hc $\m_2$ &                &     $\m_{k+1}$ &                &  \\
    };
    \draw[left arrow] (table-2-6.west) to (table-1-6.west);
    \path (table-1-3.center) -- (table-2-3.center) node[midway] (A) {\dots};
  \end{tikzpicture}
  =
  \begin{tikzpicture}[sipattern]
    \newcommand{\centralLabel}{$\m_{k+2}$}
    \matrix (table) {
          $\m_1$ & \hc $\m_2$ & \hspace{1.4em} & \hc $\m_{k+1}$ & \emptycol \\
      \hc $\m_1$ & \hc $\m_2$ &                &     $\m_{k+1}$ & \emptycol \\
    };
    \path              (table-1-5.center) -- (table-2-5.center) node[midway,highlighted entry] (A) {\centralLabel};
    \draw[left arrow]  (table-2-5.west)   to (table-1-5.west);
    \draw[right arrow] (table-2-5.east)   to (table-1-5.east);

    \path (table-1-3.center) -- (table-2-3.center) node[midway] (A) {\dots};
  \end{tikzpicture}.
\end{equation*}

The algorithm transforming $\mDG$ into $\mGD$ while preserving the corresponding product in $\M_r$ can then be described as follows.

\begin{algorithm}
  \label{alg:si_pattern}
  We start with some block of the following shape:
  \begin{equation*}
    \mDG =
    \begin{tikzpicture}[sipattern]
      \newcommand{\centralLabel}{$\m_i$}
      \matrix (table) {
        \hc $\m_1$ & \hc $\m_2$ & \ldots &     $\m_{i-1}$ & \emptycol & \hc $\m_{i+1}$ &     $\m_{i+2}$ & \ldots & \hc $\m_n$ \\
        \hc $\m_1$ & \hc $\m_2$ & \ldots & \hc $\m_{i-1}$ & \emptycol & \hc $\m_{i+1}$ & \hc $\m_{i+2}$ & \ldots &     $\m_n$ \\
      };
      \path              (table-1-5.center) -- (table-2-5.center) node[midway,highlighted entry] (A) {\centralLabel};
      \draw[left arrow]  (table-2-5.west)   to (table-1-5.west);
      \draw[right arrow] (table-2-5.east)   to (table-1-5.east);
    \end{tikzpicture}
  \end{equation*}
  for some $i$ (initially $i = 1$).

  If $\m_i$ is highlighted, then there are four possibilities for what the $i^\text{th}$ and $(i+1)^\text{th}$ column can look like:
  \begin{enumerate}
    \item \label[case]{case:mi_middle_up_down}
      \begin{tikzpicture}[sipattern]
        \newcommand{\centralLabel}{$\m_i$}
        \matrix (table) {
          \emptycol & \hc $\m_{i+1}$ \\
          \emptycol & \hc $\m_{i+1}$ \\
        };
        \path              (table-1-1.center) -- (table-2-1.center) node[midway,highlighted entry] (A) {\centralLabel};
        \draw[left arrow]  (table-2-1.west)   to (table-1-1.west);
        \draw[right arrow] (table-2-1.east)   to (table-1-1.east);
      \end{tikzpicture}
      in which case we can replace it with
      \begin{tikzpicture}[sipattern]
        \newcommand{\centralLabel}{$\m_{i+1}$}
        \matrix (table) {
          \hc $\m_i$ & \emptycol \\
          \hc $\m_i$ & \emptycol \\
        };
        \path              (table-1-2.center) -- (table-2-2.center) node[midway, highlighted entry] (A) {\centralLabel};
        \draw[left arrow]  (table-2-2.west)   to (table-1-2.west);
        \draw[right arrow] (table-2-2.east)   to (table-1-2.east);
      \end{tikzpicture}
      using the braid relation $\m_{i+1} \m_i \m_{i+1} = \m_i \m_{i+1} \m_i$;

    \item \label[case]{case:mi_middle_up}
      \begin{tikzpicture}[sipattern]
        \newcommand{\centralLabel}{$\m_i$}
        \matrix (table) {
          \emptycol & \hc $\m_{i+1}$ \\
          \emptycol &     $\m_{i+1}$ \\
        };
        \path              (table-1-1.center) -- (table-2-1.center) node[midway,highlighted entry] (A) {\centralLabel};
        \draw[left arrow]  (table-2-1.west)   to (table-1-1.west);
        \draw[right arrow] (table-2-1.east)   to (table-1-1.east);
      \end{tikzpicture}
      in which case we can trivially replace it with
      \begin{tikzpicture}[sipattern]
        \newcommand{\centralLabel}{$\m_{i+1}$}
        \matrix (table) {
              $\m_i$ & \emptycol \\
          \hc $\m_i$ & \emptycol \\
        };
        \path              (table-1-2.center) -- (table-2-2.center) node[midway, highlighted entry] (A) {\centralLabel};
        \draw[left arrow]  (table-2-2.west)   to (table-1-2.west);
        \draw[right arrow] (table-2-2.east)   to (table-1-2.east);
      \end{tikzpicture};

  \item \label[case]{case:mi_middle_down}
      \begin{tikzpicture}[sipattern]
        \newcommand{\centralLabel}{$\m_i$}
        \matrix (table) {
          \emptycol &     $\m_{i+1}$ \\
          \emptycol & \hc $\m_{i+1}$ \\
        };
        \path              (table-1-1.center) -- (table-2-1.center) node[midway,highlighted entry] (A) {\centralLabel};
        \draw[left arrow]  (table-2-1.west)   to (table-1-1.west);
        \draw[right arrow] (table-2-1.east)   to (table-1-1.east);
      \end{tikzpicture}
      in which case we can trivially replace it with
      \begin{tikzpicture}[sipattern]
        \newcommand{\centralLabel}{$\m_{i+1}$}
        \matrix (table) {
          \hc $\m_i$ & \emptycol \\
              $\m_i$ & \emptycol \\
        };
        \path              (table-1-2.center) -- (table-2-2.center) node[midway, highlighted entry] (A) {\centralLabel};
        \draw[left arrow]  (table-2-2.west)   to (table-1-2.west);
        \draw[right arrow] (table-2-2.east)   to (table-1-2.east);
      \end{tikzpicture};

  \item \label[case]{case:mi_middle}
      \begin{tikzpicture}[sipattern]
        \newcommand{\centralLabel}{$\m_i$}
        \matrix (table) {
          \emptycol & $\m_{i+1}$ \\
          \emptycol & $\m_{i+1}$ \\
        };
        \path              (table-1-1.center) -- (table-2-1.center) node[midway,highlighted entry] (A) {\centralLabel};
        \draw[left arrow]  (table-2-1.west)   to (table-1-1.west);
        \draw[right arrow] (table-2-1.east)   to (table-1-1.east);
      \end{tikzpicture}
      in which case we can trivially replace it with
      \begin{tikzpicture}[sipattern]
        \newcommand{\centralLabel}{$\m_{i+1}$}
        \matrix (table) {
          \hc $\m_i$ & \emptycol \\
                     & \emptycol \\
        };
        \path              (table-1-2.center) -- (table-2-2.center) node[midway] (A) {\centralLabel};
        \draw[left arrow]  (table-2-2.west)   to (table-1-2.west);
        \draw[right arrow] (table-2-2.east)   to (table-1-2.east);
      \end{tikzpicture}.
  \end{enumerate}
  \bigskip
  On the other hand, if $\m_i$ is not highlighted, then there are two possibilities:
  \begin{enumerate}[resume]
    \item \label[case]{case:mi_none}
      \begin{tikzpicture}[sipattern]
        \newcommand{\centralLabel}{$\m_i$}
        \matrix (table) {
          \emptycol & $\m_{i+1}$ \\
          \emptycol & $\m_{i+1}$ \\
        };
        \path              (table-1-1.center) -- (table-2-1.center) node[midway] (A) {\centralLabel};
        \draw[left arrow]  (table-2-1.west)   to (table-1-1.west);
        \draw[right arrow] (table-2-1.east)   to (table-1-1.east);
      \end{tikzpicture}
      in which case we can trivially replace it with
      \begin{tikzpicture}[sipattern]
        \newcommand{\centralLabel}{$\m_{i+1}$}
        \matrix (table) {
          $\m_i$ & \emptycol \\
          $\m_i$ & \emptycol \\
        };
        \path              (table-1-2.center) -- (table-2-2.center) node[midway] (A) {\centralLabel};
        \draw[left arrow]  (table-2-2.west)   to (table-1-2.west);
        \draw[right arrow] (table-2-2.east)   to (table-1-2.east);
      \end{tikzpicture};

    \item \label[case]{case:mi_up}
      \begin{tikzpicture}[sipattern]
        \newcommand{\centralLabel}{$\m_i$}
        \matrix (table) {
          \emptycol & \hc $\m_{i+1}$ \\
          \emptycol & \\
        };
        \path              (table-1-1.center) -- (table-2-1.center) node[midway] (A) {\centralLabel};
        \draw[left arrow]  (table-2-1.west)   to (table-1-1.west);
        \draw[right arrow] (table-2-1.east)   to (table-1-1.east);
      \end{tikzpicture}
      in which case we can trivially replace it with
      \begin{tikzpicture}[sipattern]
        \newcommand{\centralLabel}{$\m_{i+1}$}
        \matrix (table) {
          $\m_i$ & \emptycol \\
                 & \emptycol \\
        };
        \path              (table-1-2.center) -- (table-2-2.center) node[midway, highlighted entry] (A) {\centralLabel};
        \draw[left arrow]  (table-2-2.west)   to (table-1-2.west);
        \draw[right arrow] (table-2-2.east)   to (table-1-2.east);
      \end{tikzpicture}.
  \end{enumerate}
  If $i = n-1$ we are finished.
  Otherwise, repeat from the beginning with $i$ replaced by $i+1$.
\end{algorithm}

\begin{example}
  For
  \begin{tikzpicture}[sipattern]
    \newcommand{\centralLabel}{}
    \matrix (table) {
      & \hc $\m_1$ & \hc $\m_2$ &     $\m_3$ & \hc $\m_4$ & \hc $\m_5$ \\
      &            & \hc $\m_2$ & \hc $\m_3$ &     $\m_4$ & \hc $\m_5$ \\
    };
    \draw[right arrow] (table-2-1.east)   to (table-1-1.east);
  \end{tikzpicture}
  the steps are as follows:
  \begin{align*}
    \begin{tikzpicture}[sipattern]
      \newcommand{\centralLabel}{$\m_1$}
      \matrix (table) [ampersand replacement=\&] {
        \emptycol \& \hc $\m_2$ \&     $\m_3$ \& \hc $\m_4$ \& \hc $\m_5$ \\
        \emptycol \& \hc $\m_2$ \& \hc $\m_3$ \&     $\m_4$ \& \hc $\m_5$ \\
      };
      \path              (table-1-1.center) -- (table-2-1.center) node[midway, highlighted entry] (A) {\centralLabel};
      \draw[left arrow]  (table-2-1.west)   to (table-1-1.west);
      \draw[right arrow] (table-2-1.east)   to (table-1-1.east);
    \end{tikzpicture}
    &=
    \begin{tikzpicture}[sipattern]
      \newcommand{\centralLabel}{$\m_2$}
      \matrix (table) [ampersand replacement=\&] {
        \hc $\m_1$ \& \emptycol \&     $\m_3$ \& \hc $\m_4$ \& \hc $\m_5$ \\
        \hc $\m_1$ \& \emptycol \& \hc $\m_3$ \&     $\m_4$ \& \hc $\m_5$ \\
      };
      \path              (table-1-2.center) -- (table-2-2.center) node[midway, highlighted entry] (A) {\centralLabel};
      \draw[left arrow]  (table-2-2.west)   to (table-1-2.west);
      \draw[right arrow] (table-2-2.east)   to (table-1-2.east);
    \end{tikzpicture}
    =
    \begin{tikzpicture}[sipattern]
      \newcommand{\centralLabel}{$\m_3$}
      \matrix (table) [ampersand replacement=\&] {
        \hc $\m_1$ \& \hc $\m_2$ \& \emptycol \& \hc $\m_4$ \& \hc $\m_5$ \\
        \hc $\m_1$ \&     $\m_2$ \& \emptycol \&     $\m_4$ \& \hc $\m_5$ \\
      };
      \path              (table-1-3.center) -- (table-2-3.center) node[midway, highlighted entry] (A) {\centralLabel};
      \draw[left arrow]  (table-2-3.west)   to (table-1-3.west);
      \draw[right arrow] (table-2-3.east)   to (table-1-3.east);
    \end{tikzpicture}
    \\
    &=
    \begin{tikzpicture}[sipattern]
      \newcommand{\centralLabel}{$\m_4$}
      \matrix (table) [ampersand replacement=\&] {
        \hc $\m_1$ \& \hc $\m_2$ \&     $\m_3$ \& \emptycol \& \hc $\m_5$ \\
        \hc $\m_1$ \&     $\m_2$ \& \hc $\m_3$ \& \emptycol \& \hc $\m_5$ \\
      };
      \path              (table-1-4.center) -- (table-2-4.center) node[midway, highlighted entry] (A) {\centralLabel};
      \draw[left arrow]  (table-2-4.west)   to (table-1-4.west);
      \draw[right arrow] (table-2-4.east)   to (table-1-4.east);
    \end{tikzpicture}
    =
    \begin{tikzpicture}[sipattern]
      \newcommand{\centralLabel}{$\m_5$}
      \matrix (table) [ampersand replacement=\&] {
        \hc $\m_1$ \& \hc $\m_2$ \&     $\m_3$ \& \hc $\m_4$ \& \emptycol \\
        \hc $\m_1$ \&     $\m_2$ \& \hc $\m_3$ \& \hc $\m_4$ \& \emptycol \\
      };
      \path              (table-1-5.center) -- (table-2-5.center) node[midway, highlighted entry] (A) {\centralLabel};
      \draw[left arrow]  (table-2-5.west)   to (table-1-5.west);
      \draw[right arrow] (table-2-5.east)   to (table-1-5.east);
    \end{tikzpicture},
  \end{align*}
  and so the resulting pattern is 
  \begin{tikzpicture}[sipattern]
    \matrix (table) {
      \hc $\m_1$ & \hc $\m_2$ &     $\m_3$ & \hc $\m_4$ & \hc $\m_5$ & \\
      \hc $\m_1$ &     $\m_2$ & \hc $\m_3$ & \hc $\m_4$ &       & \\
    };
    \draw[left arrow]  (table-2-6.west)   to (table-1-6.west);
  \end{tikzpicture}.
\end{example}

Note that the result of the algorithm is a pattern of the same form as $\mGD$, and that the $\m_i$-column was determined by the $\m_{i+1}$-column in $\mDG$.
Furthermore, when the $i^\text{th}$ and $(i+1)^\text{th}$ columns are inside a block, \cref*{case:mi_middle_up_down,case:mi_middle_up,case:mi_middle_down} above exactly match \cref*{case:si-double,case:si-top,case:si-bottom} in \cref{lemma:short_GTP_highlighting_gamma_delta}.
The remaining cases, \ref*{case:mi_middle} describe what happens at the end of a block, \ref*{case:mi_up} at the beginning of a block, and \ref*{case:mi_none} between blocks and agree with how blocks are treated in connection with \cref{lemma:short_GTP_highlighting_gamma_delta}.

To summarize, the algorithm has step-by-step transformed $\mDG$ to $\mGD$ while preserving the associated element in $\M_r$.
We have thus proved the following proposition:

\begin{proposition}
  The elements of $\M_r$ associated to $\mDG$ and $\mGD$ are the same, and hence the boundary conditions of $t_{r-i}(\mathfrak{s})$ are the same as those of $\mathfrak{s}$ but with row $i$ and $i+1$ switched.
\end{proposition}

This concludes the proof of \cref{thm:ti-boundary-conditions}.

\newpage

\appendix

\section{Proof of Proposition \ref{prop:inverse}}
\label{appendix:proof_of_R-matrix_inverse}

In this appendix we prove that the mixed $R$-matrices $\Rlr : \Vl \otimes \Vr \to \Vr \otimes \Vl$ and $\Rrl : \Vr \otimes \Vl \to \Vl \otimes \Vr$ from \cref{tab:R-vertices-mixed} satisfy
\begin{equation}
  \label{appendix:eq:mixed_R-matrices_inverses}
  \Rrl(z_1, z_2) \Rlr(z_2, z_1) = C \cdot \id_{\Vl} \otimes \id_{\Vr}
\end{equation}
for some constant $C$ depending on $q$, $z_1$, $z_2$ and the palette size $m$.
Since $\Vl$ and $\Vr$ are finite-dimensional vector spaces over the field of rational functions in $q$, $z_1$ and $z_2$ for a fixed integer $m$, this means that $\Rlr(z_2, z_1) \Rrl(z_1, z_2) = C \cdot \id_{\Vl} \otimes \id_{\Vr}$ as well.
Indeed, let $A = \frac{1}{C} \Rrl(z_1, z_2) \tau $ and $B = \tau \Rlr(z_2, z_1)$ where $\tau$ is the linear map defined on pure tensors by $u \otimes v \mapsto v \otimes u$.
Then $A,B \in \operatorname{End}(\Vl \otimes \Vr)$ and the left-inverse $A$ of $B$ is also the right-inverse of $B$.

Both $R$-vertices have the same vertex color $c_k$.
Due to the fact that the weights only depend on the possible edge colors $c_i$ and $c_j$ by expressions in $i-k$ and $j-k$ mod $m$ we may without loss of generality assume that $k=1$.
Indeed, the conditions $\left\{\begin{smallmatrix} i < j < k \text{ or } j < k \leqslant i \text{ or } k \leqslant i < j \\ j < i < k \text{ or } i < k \leqslant j \text{ or } k \leqslant j < i \end{smallmatrix}\right.$ are equivalent to $\res_m (i - j) - \res_m (i - k) + \res_m (j - k) = \left\{\begin{smallmatrix} m \\ 0 \end{smallmatrix}\right.$.

We first enumerate the possible admissible states using the configurations of \cref{tab:R-vertices-mixed}, starting from the possible left boundary edges and iteratively filling in the rest.
Here, dashed edges signify unknown edges and the equations should be interpreted as equalities and unions of sets of states.
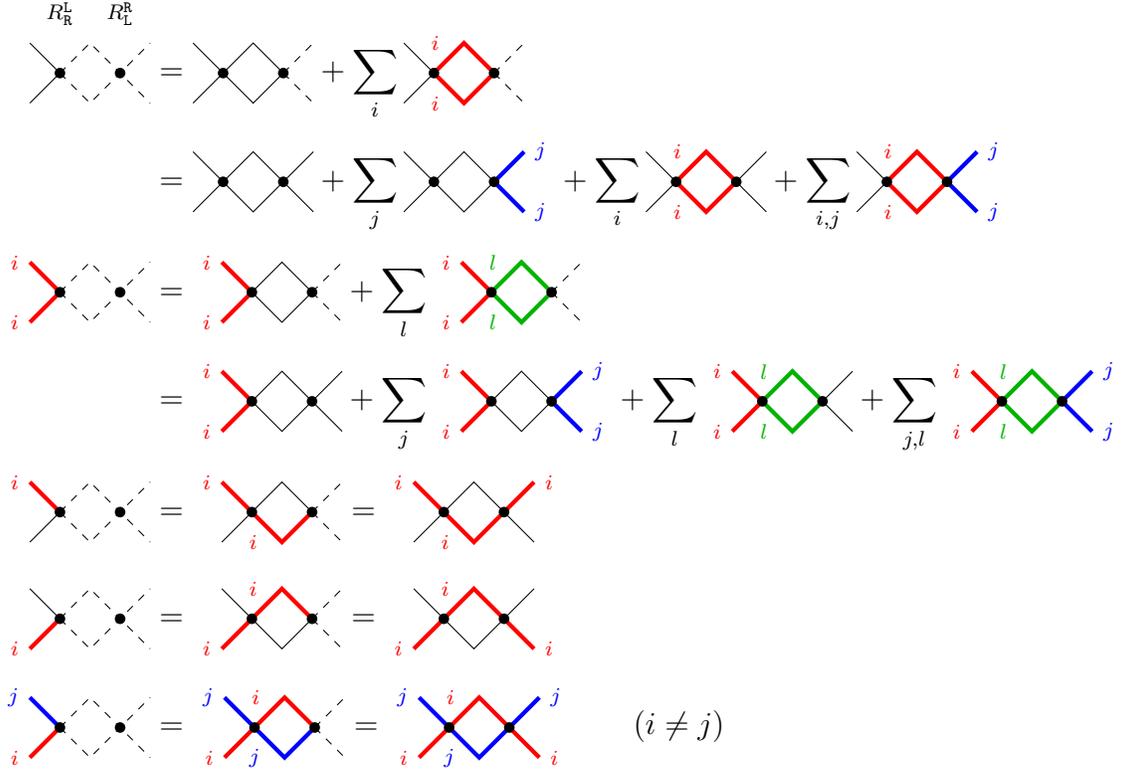
\newcommand\scale{0.4}
{%
  \allowdisplaybreaks
  \begin{align*}
  \begin{tikzpicture}[baseline={(0,-0.5ex)}, scale = \scale]
    \draw (-1,-1) -- (0,0);
    \draw (-1,1) -- (0,0);
    \draw[dashed] (0,0) -- (1,1) -- (3,-1);
    \draw[dashed] (0,0) -- (1,-1) -- (3,1);
    \node[dot, label={[label distance=1em]above:$\scriptstyle{\Rlr}$}] at (0,0) {};
    \node[dot, label={[label distance=1em]above:$\scriptstyle{\Rrl}$}] at (2,0) {};
  \end{tikzpicture}
  &=
  \begin{tikzpicture}[baseline={(0,-0.5ex)}, scale = \scale]
    \draw (-1,-1) -- (0,0);
    \draw (-1,1) -- (0,0); 
    \draw (0,0) -- (1,1) -- (2,0);
    \draw (0,0) -- (1,-1) -- (2,0);
    \draw[dashed] (2,0) -- (3,-1);
    \draw[dashed] (2,0) -- (3,1);
    \node[dot] at (0,0) {};
    \node[dot] at (2,0) {};
  \end{tikzpicture}
  +
  \sum_i
  \begin{tikzpicture}[baseline={(0,-0.5ex)}, scale = \scale]
    \draw (-1,-1) -- (0,0);
    \draw (-1,1) --  (0,0); 
    \draw[ultra thick, red] (0,0) -- (1,1) node[label=left:$\scriptstyle{i}$]{} -- (2,0);
    \draw[ultra thick, red] (0,0) -- (1,-1) node[label=left:$\scriptstyle{i}$]{}  -- (2,0);
    \draw[dashed] (2,0) -- (3,-1);
    \draw[dashed] (2,0) -- (3,1);
    \node[dot] at (0,0) {};
    \node[dot] at (2,0) {};
  \end{tikzpicture}
  \\
  &=
  \begin{tikzpicture}[baseline={(0,-0.5ex)}, scale = \scale]
    \draw (-1,-1) -- (0,0);
    \draw (-1,1) -- (0,0); 
    \draw (0,0) -- (1,1) -- (2,0);
    \draw (0,0) -- (1,-1) -- (2,0);
    \draw (2,0) -- (3,-1);
    \draw (2,0) -- (3,1);
    \node[dot] at (0,0) {};
    \node[dot] at (2,0) {};
  \end{tikzpicture}
  + 
  \sum_j
  \begin{tikzpicture}[baseline={(0,-0.5ex)}, scale = \scale]
    \draw (-1,-1) -- (0,0);
    \draw (-1,1) -- (0,0); 
    \draw (0,0) -- (1,1) -- (2,0);
    \draw (0,0) -- (1,-1) -- (2,0);
    \draw[ultra thick, blue] (2,0) -- (3,-1) node[label={[label distance=-1ex]right:$\scriptstyle{j}$}]{};
    \draw[ultra thick, blue] (2,0) -- (3,1) node[label={[label distance=-1ex]right:$\scriptstyle{j}$}]{};
    \node[dot] at (0,0) {};
    \node[dot] at (2,0) {};
  \end{tikzpicture}
  +
  \sum_i
  \begin{tikzpicture}[baseline={(0,-0.5ex)}, scale = \scale]
    \draw (-1,-1) -- (0,0);
    \draw (-1,1) --  (0,0); 
    \draw[ultra thick, red] (0,0) -- (1,1) node[label=left:$\scriptstyle{i}$]{} -- (2,0);
    \draw[ultra thick, red] (0,0) -- (1,-1) node[label=left:$\scriptstyle{i}$]{}  -- (2,0);
    \draw (2,0) -- (3,-1);
    \draw (2,0) -- (3,1);
    \node[dot] at (0,0) {};
    \node[dot] at (2,0) {};
  \end{tikzpicture}
  +
  \sum_{i,j}
  \begin{tikzpicture}[baseline={(0,-0.5ex)}, scale = \scale]
    \draw (-1,-1) -- (0,0);
    \draw (-1,1) --  (0,0); 
    \draw[ultra thick, red] (0,0) -- (1,1) node[label=left:$\scriptstyle{i}$]{} -- (2,0);
    \draw[ultra thick, red] (0,0) -- (1,-1) node[label=left:$\scriptstyle{i}$]{}  -- (2,0);
    \draw[ultra thick, blue] (2,0) -- (3,-1) node[label={[label distance=-1ex]right:$\scriptstyle{j}$}]{};
    \draw[ultra thick, blue] (2,0) -- (3,1) node[label={[label distance=-1ex]right:$\scriptstyle{j}$}]{};
    \node[dot] at (0,0) {};
    \node[dot] at (2,0) {};
  \end{tikzpicture}
  \\
  \begin{tikzpicture}[baseline={(0,-0.5ex)}, scale = \scale]
    \draw[ultra thick, red] (-1,-1) node[label={[label distance=-1ex]left:$\scriptstyle{i}$}]{} -- (0,0);
    \draw[ultra thick, red] (-1,1) node[label={[label distance=-1ex]left:$\scriptstyle{i}$}]{} -- (0,0);
    \draw[dashed] (0,0) -- (1,1) -- (3,-1);
    \draw[dashed] (0,0) -- (1,-1) -- (3,1);
    \node[dot] at (0,0) {};
    \node[dot] at (2,0) {};
  \end{tikzpicture}
  &=
  \begin{tikzpicture}[baseline={(0,-0.5ex)}, scale = \scale]
    \draw[ultra thick, red] (-1,-1) node[label={[label distance=-1ex]left:$\scriptstyle{i}$}]{} -- (0,0);
    \draw[ultra thick, red] (-1,1) node[label={[label distance=-1ex]left:$\scriptstyle{i}$}]{} -- (0,0);
    \draw (0,0) -- (1,1) -- (2,0);
    \draw (0,0) -- (1,-1) -- (2,0);
    \draw[dashed] (2,0) -- (3,-1);
    \draw[dashed] (2,0) -- (3,1);
    \node[dot] at (0,0) {};
    \node[dot] at (2,0) {};
  \end{tikzpicture}
  +
  \sum_l
  \begin{tikzpicture}[baseline={(0,-0.5ex)}, scale = \scale]
    \draw[ultra thick, red] (-1,-1) node[label={[label distance=-1ex]left:$\scriptstyle{i}$}]{} -- (0,0);
    \draw[ultra thick, red] (-1,1) node[label={[label distance=-1ex]left:$\scriptstyle{i}$}]{} -- (0,0);
    \draw[ultra thick, green] (0,0) -- (1,1) node[label=left:$\scriptstyle{l}$]{} -- (2,0);
    \draw[ultra thick, green] (0,0) -- (1,-1) node[label=left:$\scriptstyle{l}$]{}  -- (2,0);
    \draw[dashed] (2,0) -- (3,-1);
    \draw[dashed] (2,0) -- (3,1);
    \node[dot] at (0,0) {};
    \node[dot] at (2,0) {};
  \end{tikzpicture}
  \\
  &=
  \begin{tikzpicture}[baseline={(0,-0.5ex)}, scale = \scale]
    \draw[ultra thick, red] (-1,-1) node[label={[label distance=-1ex]left:$\scriptstyle{i}$}]{} -- (0,0);
    \draw[ultra thick, red] (-1,1) node[label={[label distance=-1ex]left:$\scriptstyle{i}$}]{} -- (0,0);
    \draw (0,0) -- (1,1) -- (2,0);
    \draw (0,0) -- (1,-1) -- (2,0);
    \draw (2,0) -- (3,-1);
    \draw (2,0) -- (3,1);
    \node[dot] at (0,0) {};
    \node[dot] at (2,0) {};
  \end{tikzpicture}
  +
  \sum_j
  \begin{tikzpicture}[baseline={(0,-0.5ex)}, scale = \scale]
    \draw[ultra thick, red] (-1,-1) node[label={[label distance=-1ex]left:$\scriptstyle{i}$}]{} -- (0,0);
    \draw[ultra thick, red] (-1,1) node[label={[label distance=-1ex]left:$\scriptstyle{i}$}]{} -- (0,0);
    \draw (0,0) -- (1,1) -- (2,0);
    \draw (0,0) -- (1,-1) -- (2,0);
    \draw[ultra thick, blue] (2,0) -- (3,-1) node[label={[label distance=-1ex]right:$\scriptstyle{j}$}]{};
    \draw[ultra thick, blue] (2,0) -- (3,1) node[label={[label distance=-1ex]right:$\scriptstyle{j}$}]{};
    \node[dot] at (0,0) {};
    \node[dot] at (2,0) {};
  \end{tikzpicture}
  +
  \sum_l
  \begin{tikzpicture}[baseline={(0,-0.5ex)}, scale = \scale]
    \draw[ultra thick, red] (-1,-1) node[label={[label distance=-1ex]left:$\scriptstyle{i}$}]{} -- (0,0);
    \draw[ultra thick, red] (-1,1) node[label={[label distance=-1ex]left:$\scriptstyle{i}$}]{} -- (0,0);
    \draw[ultra thick, green] (0,0) -- (1,1) node[label=left:$\scriptstyle{l}$]{} -- (2,0);
    \draw[ultra thick, green] (0,0) -- (1,-1) node[label=left:$\scriptstyle{l}$]{}  -- (2,0);
    \draw (2,0) -- (3,-1);
    \draw (2,0) -- (3,1);
    \node[dot] at (0,0) {};
    \node[dot] at (2,0) {};
  \end{tikzpicture}
  +
  \sum_{j,l}
  \begin{tikzpicture}[baseline={(0,-0.5ex)}, scale = \scale]
    \draw[ultra thick, red] (-1,-1) node[label={[label distance=-1ex]left:$\scriptstyle{i}$}]{} -- (0,0);
    \draw[ultra thick, red] (-1,1) node[label={[label distance=-1ex]left:$\scriptstyle{i}$}]{} -- (0,0);
    \draw[ultra thick, green] (0,0) -- (1,1) node[label=left:$\scriptstyle{l}$]{} -- (2,0);
    \draw[ultra thick, green] (0,0) -- (1,-1) node[label=left:$\scriptstyle{l}$]{}  -- (2,0);
    \draw[ultra thick, blue] (2,0) -- (3,-1) node[label={[label distance=-1ex]right:$\scriptstyle{j}$}]{};
    \draw[ultra thick, blue] (2,0) -- (3,1) node[label={[label distance=-1ex]right:$\scriptstyle{j}$}]{};
    \node[dot] at (0,0) {};
    \node[dot] at (2,0) {};
  \end{tikzpicture}
  \\
  \begin{tikzpicture}[baseline={(0,-0.5ex)}, scale = \scale]
    \draw (-1,-1) -- (0,0);
    \draw[ultra thick, red] (-1,1) node[label={[label distance=-1ex]left:$\scriptstyle{i}$}]{} -- (0,0);
    \draw[dashed] (0,0) -- (1,1) -- (3,-1);
    \draw[dashed] (0,0) -- (1,-1) -- (3,1);
    \node[dot] at (0,0) {};
    \node[dot] at (2,0) {};
  \end{tikzpicture}
  &= 
  \begin{tikzpicture}[baseline={(0,-0.5ex)}, scale = \scale]
    \draw (-1,-1) -- (0,0);
    \draw[ultra thick, red] (-1,1) node[label={[label distance=-1ex]left:$\scriptstyle{i}$}]{} -- (0,0);
    \draw (0,0) -- (1,1) -- (2,0);
    \draw[ultra thick, red] (0,0) -- (1,-1) node[label={[label distance=0ex]left:$\scriptstyle{i}$}]{} -- (2,0);
    \draw[dashed] (2,0) -- (3,-1);
    \draw[dashed] (2,0) -- (3,1); 
    \node[dot] at (0,0) {};
    \node[dot] at (2,0) {};
  \end{tikzpicture}
  =
  \begin{tikzpicture}[baseline={(0,-0.5ex)}, scale = \scale]
    \draw (-1,-1) -- (0,0);
    \draw[ultra thick, red] (-1,1) node[label={[label distance=-1ex]left:$\scriptstyle{i}$}]{} -- (0,0);
    \draw (0,0) -- (1,1) -- (2,0);
    \draw[ultra thick, red] (0,0) -- (1,-1) node[label={[label distance=0ex]left:$\scriptstyle{i}$}]{} -- (2,0);
    \draw[ultra thick, red] (2,0) -- (3,1) node[label={[label distance=-1ex]right:$\scriptstyle{i}$}]{};
    \draw (2,0) -- (3,-1); 
    \node[dot] at (0,0) {};
    \node[dot] at (2,0) {};
  \end{tikzpicture}
  \\
  \begin{tikzpicture}[baseline={(0,-0.5ex)}, scale = \scale, yscale=-1]
    \draw (-1,-1) -- (0,0);
    \draw[ultra thick, red] (-1,1) node[label={[label distance=-1ex]left:$\scriptstyle{i}$}]{} -- (0,0);
    \draw[dashed] (0,0) -- (1,1) -- (3,-1);
    \draw[dashed] (0,0) -- (1,-1) -- (3,1);
    \node[dot] at (0,0) {};
    \node[dot] at (2,0) {};
  \end{tikzpicture}
  &= 
  \begin{tikzpicture}[baseline={(0,-0.5ex)}, scale = \scale, yscale=-1]
    \draw (-1,-1) -- (0,0);
    \draw[ultra thick, red] (-1,1) node[label={[label distance=-1ex]left:$\scriptstyle{i}$}]{} -- (0,0);
    \draw (0,0) -- (1,1) -- (2,0);
    \draw[ultra thick, red] (0,0) -- (1,-1) node[label={[label distance=0ex]left:$\scriptstyle{i}$}]{} -- (2,0);
    \draw[dashed] (2,0) -- (3,-1);
    \draw[dashed] (2,0) -- (3,1); 
    \node[dot] at (0,0) {};
    \node[dot] at (2,0) {};
  \end{tikzpicture}
  =
  \begin{tikzpicture}[baseline={(0,-0.5ex)}, scale = \scale, yscale=-1]
    \draw (-1,-1) -- (0,0);
    \draw[ultra thick, red] (-1,1) node[label={[label distance=-1ex]left:$\scriptstyle{i}$}]{} -- (0,0);
    \draw (0,0) -- (1,1) -- (2,0);
    \draw[ultra thick, red] (0,0) -- (1,-1) node[label={[label distance=0ex]left:$\scriptstyle{i}$}]{} -- (2,0);
    \draw[ultra thick, red] (2,0) -- (3,1) node[label={[label distance=-1ex]right:$\scriptstyle{i}$}]{};
    \draw (2,0) -- (3,-1); 
    \node[dot] at (0,0) {};
    \node[dot] at (2,0) {};
  \end{tikzpicture}
  \\
  \begin{tikzpicture}[baseline={(0,-0.5ex)}, scale = \scale]
    \draw[ultra thick, red] (-1,-1) node[label={[label distance=-1ex]left:$\scriptstyle{i}$}]{} -- (0,0);
    \draw[ultra thick, blue] (-1,1) node[label={[label distance=-1ex]left:$\scriptstyle{j}$}]{} -- (0,0);
    \draw[dashed] (0,0) -- (1,1) -- (3,-1);
    \draw[dashed] (0,0) -- (1,-1) -- (3,1);
    \node[dot] at (0,0) {};
    \node[dot] at (2,0) {};
  \end{tikzpicture}
  &= 
  \begin{tikzpicture}[baseline={(0,-0.5ex)}, scale = \scale]
    \draw[ultra thick, red] (-1,-1) node[label={[label distance=-1ex]left:$\scriptstyle{i}$}]{} -- (0,0);
    \draw[ultra thick, blue] (-1,1) node[label={[label distance=-1ex]left:$\scriptstyle{j}$}]{} -- (0,0);
    \draw[ultra thick, red] (0,0) -- (1,1) node[label={[label distance=0ex]left:$\scriptstyle{i}$}]{} -- (2,0);
    \draw[ultra thick, blue] (0,0) -- (1,-1) node[label={[label distance=0ex]left:$\scriptstyle{j}$}]{} -- (2,0);
    \draw[dashed] (2,0) -- (3,-1);
    \draw[dashed] (2,0) -- (3,1); 
    \node[dot] at (0,0) {};
    \node[dot] at (2,0) {};
  \end{tikzpicture}
  =
  \begin{tikzpicture}[baseline={(0,-0.5ex)}, scale = \scale]
    \draw[ultra thick, red] (-1,-1) node[label={[label distance=-1ex]left:$\scriptstyle{i}$}]{} -- (0,0);
    \draw[ultra thick, blue] (-1,1) node[label={[label distance=-1ex]left:$\scriptstyle{j}$}]{} -- (0,0);
    \draw[ultra thick, red] (0,0) -- (1,1) node[label={[label distance=0ex]left:$\scriptstyle{i}$}]{} -- (2,0);
    \draw[ultra thick, blue] (0,0) -- (1,-1) node[label={[label distance=0ex]left:$\scriptstyle{j}$}]{} -- (2,0);
    \draw[ultra thick, red] (2,0) -- (3,-1) node[label={[label distance=-1ex]right:$\scriptstyle{i}$}]{};
    \draw[ultra thick, blue] (2,0) -- (3,1) node[label={[label distance=-1ex]right:$\scriptstyle{j}$}]{}; 
    \node[dot] at (0,0) {};
    \node[dot] at (2,0) {};
  \end{tikzpicture}
  \qquad (i \neq j)
\end{align*}
}

Let us now consider the weights.
We need to show that when the left and right boundary edges match then the corresponding partition function, summing over all admissible internal edge configurations, is equal to $(z_2 - q^2 z_1) (z_2 - q^{2 m} z_1)$ and that otherwise the partition function is zero.

Using the weights from \cref{tab:R-vertices-mixed} we immediately get that
\begin{align*}
  Z\left( 
    \begin{tikzpicture}[baseline={(0,-0.5ex)}, scale = \scale]
    \draw (-1,-1) -- (0,0);
    \draw[ultra thick, red] (-1,1) node[label={[label distance=-1ex]left:$\scriptstyle{i}$}]{} -- (0,0);
    \draw (0,0) -- (1,1) -- (2,0);
    \draw[ultra thick, red] (0,0) -- (1,-1) node[label={[label distance=0ex]left:$\scriptstyle{i}$}]{} -- (2,0);
    \draw[ultra thick, red] (2,0) -- (3,1) node[label={[label distance=-1ex]right:$\scriptstyle{i}$}]{};
    \draw (2,0) -- (3,-1); 
    \node[dot] at (0,0) {};
    \node[dot] at (2,0) {};
  \end{tikzpicture} 
\right) &=
  Z \left(
    \begin{tikzpicture}[baseline={(0,-0.5ex)}, scale = \scale, yscale=-1]
    \draw (-1,-1) -- (0,0);
    \draw[ultra thick, red] (-1,1) node[label={[label distance=-1ex]left:$\scriptstyle{i}$}]{} -- (0,0);
    \draw (0,0) -- (1,1) -- (2,0);
    \draw[ultra thick, red] (0,0) -- (1,-1) node[label={[label distance=0ex]left:$\scriptstyle{i}$}]{} -- (2,0);
    \draw[ultra thick, red] (2,0) -- (3,1) node[label={[label distance=-1ex]right:$\scriptstyle{i}$}]{};
    \draw (2,0) -- (3,-1); 
    \node[dot] at (0,0) {};
    \node[dot] at (2,0) {};
  \end{tikzpicture}
    \right)
    =
    \Phi(z_2 - q^2 z_1) \cdot (z_2 - q^{2m}z_1)/\Phi = C 
    \\
  Z\left(
    \begin{tikzpicture}[baseline={(0,-0.5ex)}, scale = \scale]
    \draw[ultra thick, red] (-1,-1) node[label={[label distance=-1ex]left:$\scriptstyle{i}$}]{} -- (0,0);
    \draw[ultra thick, blue] (-1,1) node[label={[label distance=-1ex]left:$\scriptstyle{j}$}]{} -- (0,0);
    \draw[ultra thick, red] (0,0) -- (1,1) node[label={[label distance=0ex]left:$\scriptstyle{i}$}]{} -- (2,0);
    \draw[ultra thick, blue] (0,0) -- (1,-1) node[label={[label distance=0ex]left:$\scriptstyle{j}$}]{} -- (2,0);
    \draw[ultra thick, red] (2,0) -- (3,-1) node[label={[label distance=-1ex]right:$\scriptstyle{i}$}]{};
    \draw[ultra thick, blue] (2,0) -- (3,1) node[label={[label distance=-1ex]right:$\scriptstyle{j}$}]{}; 
    \node[dot] at (0,0) {};
    \node[dot] at (2,0) {};
  \end{tikzpicture}
    \right)
    &= \Phi^2 \mathfrak{X}_{j,i} (z_2 - q^2 z_1) \cdot \frac{\mathfrak{X}_{i,j}}{(q\Phi)^2}(z_2 - q^{2m}z_1) = C \qquad(i \neq j)
\end{align*}
since $\mathfrak{X}_{i,j} \mathfrak{X}_{j,i} = q^2$ for $i\neq j$.

The remaining cases involve sums over internal colors, i.e.\, color loops, where the sum are over all $m$ colors.
That \eqref{appendix:eq:mixed_R-matrices_inverses}, and therefore also the Yang--Baxter equations, holds for all $m$ can be explained by the fact that these sums turn out to be telescopic due to an intricate interplay between the two vertex weights in each state.

\begin{align*}
  \MoveEqLeft
  Z\left(
    \begin{tikzpicture}[baseline={(0,-0.5ex)}, scale = \scale]
    \draw (-1,-1) -- (0,0);
    \draw (-1,1) -- (0,0); 
    \draw (0,0) -- (1,1) -- (2,0);
    \draw (0,0) -- (1,-1) -- (2,0);
    \draw (2,0) -- (3,-1);
    \draw (2,0) -- (3,1);
    \node[dot] at (0,0) {};
    \node[dot] at (2,0) {};
  \end{tikzpicture}
  + 
  \sum_i
  \begin{tikzpicture}[baseline={(0,-0.5ex)}, scale = \scale]
    \draw (-1,-1) -- (0,0);
    \draw (-1,1) --  (0,0); 
    \draw[ultra thick, red] (0,0) -- (1,1) node[label=left:$\scriptstyle{i}$]{} -- (2,0);
    \draw[ultra thick, red] (0,0) -- (1,-1) node[label=left:$\scriptstyle{i}$]{}  -- (2,0);
    \draw (2,0) -- (3,-1);
    \draw (2,0) -- (3,1);
    \node[dot] at (0,0) {};
    \node[dot] at (2,0) {};
  \end{tikzpicture}
    \right) = \\
    &= (z_2-z_1)(z_2 - q^{2m+2}z_1) + \sum_{i=1}^m \Phi(1-q^2)z_2 \cdot \frac{1}{\Phi}q^{2\res_m(i-1)}(1-q^2)z_1 \\
    &= (z_2-z_1)(z_2 - q^{2m+2}z_1) + z_1 z_2(1-q^2)\sum_{i=1}^m  (1-q^2) q^{2(i-1)} \\
    &= (z_2-z_1)(z_2 - q^{2m+2}z_1) + z_1 z_2(1-q^2)(1-q^{2m}) = C
\end{align*}
since $\sum_{i=1}^m  (1-q^2) q^{2(i-1)} = \sum_{i=1}^m (q^{2(i-1)} - q^{2i}) = 1 - q^{2m}$.

For the next computation we first separate the term where $i=j$ and then use the fact that $k=1$ to simplify the conditions appearing in the weights in \cref{tab:R-vertices-mixed}.
\begin{align*}
  \MoveEqLeft
  Z\left(
  \begin{tikzpicture}[baseline={(0,-0.5ex)}, scale = \scale]
    \draw (-1,-1) -- (0,0);
    \draw (-1,1) -- (0,0); 
    \draw (0,0) -- (1,1) -- (2,0);
    \draw (0,0) -- (1,-1) -- (2,0);
    \draw[ultra thick, blue] (2,0) -- (3,-1) node[label={[label distance=-1ex]right:$\scriptstyle{j}$}]{};
    \draw[ultra thick, blue] (2,0) -- (3,1) node[label={[label distance=-1ex]right:$\scriptstyle{j}$}]{};
    \node[dot] at (0,0) {};
    \node[dot] at (2,0) {};
  \end{tikzpicture}
  +
  \sum_{i}
  \begin{tikzpicture}[baseline={(0,-0.5ex)}, scale = \scale]
    \draw (-1,-1) -- (0,0);
    \draw (-1,1) --  (0,0); 
    \draw[ultra thick, red] (0,0) -- (1,1) node[label=left:$\scriptstyle{i}$]{} -- (2,0);
    \draw[ultra thick, red] (0,0) -- (1,-1) node[label=left:$\scriptstyle{i}$]{}  -- (2,0);
    \draw[ultra thick, blue] (2,0) -- (3,-1) node[label={[label distance=-1ex]right:$\scriptstyle{j}$}]{};
    \draw[ultra thick, blue] (2,0) -- (3,1) node[label={[label distance=-1ex]right:$\scriptstyle{j}$}]{};
    \node[dot] at (0,0) {};
    \node[dot] at (2,0) {};
  \end{tikzpicture}
  \right) =   
   Z\left(
  \begin{tikzpicture}[baseline={(0,-0.5ex)}, scale = \scale]
    \draw (-1,-1) -- (0,0);
    \draw (-1,1) -- (0,0); 
    \draw (0,0) -- (1,1) -- (2,0);
    \draw (0,0) -- (1,-1) -- (2,0);
    \draw[ultra thick, blue] (2,0) -- (3,-1) node[label={[label distance=-1ex]right:$\scriptstyle{j}$}]{};
    \draw[ultra thick, blue] (2,0) -- (3,1) node[label={[label distance=-1ex]right:$\scriptstyle{j}$}]{};
    \node[dot] at (0,0) {};
    \node[dot] at (2,0) {};
  \end{tikzpicture}
  +
  \begin{tikzpicture}[baseline={(0,-0.5ex)}, scale = \scale]
    \draw (-1,-1) -- (0,0);
    \draw (-1,1) --  (0,0); 
    \draw[ultra thick, blue] (0,0) -- (1,1) node[label=left:$\scriptstyle{j}$]{} -- (2,0);
    \draw[ultra thick, blue] (0,0) -- (1,-1) node[label=left:$\scriptstyle{j}$]{}  -- (2,0);
    \draw[ultra thick, blue] (2,0) -- (3,-1) node[label={[label distance=-1ex]right:$\scriptstyle{j}$}]{};
    \draw[ultra thick, blue] (2,0) -- (3,1) node[label={[label distance=-1ex]right:$\scriptstyle{j}$}]{};
    \node[dot] at (0,0) {};
    \node[dot] at (2,0) {};
  \end{tikzpicture}
  +
  \sum_{i\neq j}
  \begin{tikzpicture}[baseline={(0,-0.5ex)}, scale = \scale]
    \draw (-1,-1) -- (0,0);
    \draw (-1,1) --  (0,0); 
    \draw[ultra thick, red] (0,0) -- (1,1) node[label=left:$\scriptstyle{i}$]{} -- (2,0);
    \draw[ultra thick, red] (0,0) -- (1,-1) node[label=left:$\scriptstyle{i}$]{}  -- (2,0);
    \draw[ultra thick, blue] (2,0) -- (3,-1) node[label={[label distance=-1ex]right:$\scriptstyle{j}$}]{};
    \draw[ultra thick, blue] (2,0) -- (3,1) node[label={[label distance=-1ex]right:$\scriptstyle{j}$}]{};
    \node[dot] at (0,0) {};
    \node[dot] at (2,0) {};
  \end{tikzpicture}
  \right) = \\
  &= (z_2-z_1)\cdot\frac{1}{\Phi}q^{2\res^m(1-j)-2}(1-q^2)z_2 + \Phi(1-q^2)z_2 \cdot (q^{2m-2} z_1 - z_2)/\Phi^2  + {}\\
  &\quad +
  \sum_{i\neq j} \Phi(1-q^2)z_2 \cdot \frac{1}{\Phi^2} q^{2\res_m(i-j)-2}(1-q^2) \Bigl\{ \begin{smallmatrix} z_1 & \text{if } i < j \\ z_2 & \text{if }j < i\end{smallmatrix}  .
\end{align*}
Since
\begin{multline*}
  \sum_{i\neq j} q^{2\res_m(i-j)-2}(1-q^2) \Bigl\{ \begin{smallmatrix} z_1 & \text{if } i < j \\ z_2 & \text{if }j < i\end{smallmatrix} = z_1 \sum_{i<j} q^{2(m+i-j)-2}(1-q^2) + z_2 \sum_{i>j} q^{2(i-j)-2}(1-q^2) = \\
  = z_1 \bigl( q^{2(m+1-j)-2} - q^{2(m+j-1-j)} \bigr) + z_2 \bigl( q^{2(j+1-j)-2} - q^{2(m-j)} \bigr)
\end{multline*}
and $\res^m(1-j) = m+1-j$ for $j \in \{1,\ldots,m\}$ we get that the partition function equals
\begin{multline*}
  \frac{1}{\Phi}(1-q^2)z_2 \Bigl( (z_2-z_1)q^{2(m-j)} + (q^{2m-2} z_1 - z_2) + z_1 \bigl( q^{2(m-j)} - q^{2m-2} \bigr) + z_2 \bigl( 1 - q^{2(m-j)} \bigr) \Bigr) = 0.
\end{multline*}

The mirrored case is similar.
\begin{align*}
  \MoveEqLeft
  Z\left(
  \begin{tikzpicture}[baseline={(0,-0.5ex)}, scale = \scale]
    \draw[ultra thick, red] (-1,-1) node[label={[label distance=-1ex]left:$\scriptstyle{i}$}]{} -- (0,0);
    \draw[ultra thick, red] (-1,1) node[label={[label distance=-1ex]left:$\scriptstyle{i}$}]{} -- (0,0);
    \draw (0,0) -- (1,1) -- (2,0);
    \draw (0,0) -- (1,-1) -- (2,0);
    \draw (2,0) -- (3,-1);
    \draw (2,0) -- (3,1);
    \node[dot] at (0,0) {};
    \node[dot] at (2,0) {};
  \end{tikzpicture}
  +
  \sum_l
  \begin{tikzpicture}[baseline={(0,-0.5ex)}, scale = \scale]
    \draw[ultra thick, red] (-1,-1) node[label={[label distance=-1ex]left:$\scriptstyle{i}$}]{} -- (0,0);
    \draw[ultra thick, red] (-1,1) node[label={[label distance=-1ex]left:$\scriptstyle{i}$}]{} -- (0,0);
    \draw[ultra thick, green] (0,0) -- (1,1) node[label=left:$\scriptstyle{l}$]{} -- (2,0);
    \draw[ultra thick, green] (0,0) -- (1,-1) node[label=left:$\scriptstyle{l}$]{}  -- (2,0);
    \draw (2,0) -- (3,-1);
    \draw (2,0) -- (3,1);
    \node[dot] at (0,0) {};
    \node[dot] at (2,0) {};
  \end{tikzpicture}
    \right) = 
  Z\left(
  \begin{tikzpicture}[baseline={(0,-0.5ex)}, scale = \scale]
    \draw[ultra thick, red] (-1,-1) node[label={[label distance=-1ex]left:$\scriptstyle{i}$}]{} -- (0,0);
    \draw[ultra thick, red] (-1,1) node[label={[label distance=-1ex]left:$\scriptstyle{i}$}]{} -- (0,0);
    \draw (0,0) -- (1,1) -- (2,0);
    \draw (0,0) -- (1,-1) -- (2,0);
    \draw (2,0) -- (3,-1);
    \draw (2,0) -- (3,1);
    \node[dot] at (0,0) {};
    \node[dot] at (2,0) {};
  \end{tikzpicture}
  +
  \begin{tikzpicture}[baseline={(0,-0.5ex)}, scale = \scale]
    \draw[ultra thick, red] (-1,-1) node[label={[label distance=-1ex]left:$\scriptstyle{i}$}]{} -- (0,0);
    \draw[ultra thick, red] (-1,1) node[label={[label distance=-1ex]left:$\scriptstyle{i}$}]{} -- (0,0);
    \draw[ultra thick, red] (0,0) -- (1,1) node[label=left:$\scriptstyle{i}$]{} -- (2,0);
    \draw[ultra thick, red] (0,0) -- (1,-1) node[label=left:$\scriptstyle{i}$]{}  -- (2,0);
    \draw (2,0) -- (3,-1);
    \draw (2,0) -- (3,1);
    \node[dot] at (0,0) {};
    \node[dot] at (2,0) {};
  \end{tikzpicture}
  +
  \sum_{l \neq i}
  \begin{tikzpicture}[baseline={(0,-0.5ex)}, scale = \scale]
    \draw[ultra thick, red] (-1,-1) node[label={[label distance=-1ex]left:$\scriptstyle{i}$}]{} -- (0,0);
    \draw[ultra thick, red] (-1,1) node[label={[label distance=-1ex]left:$\scriptstyle{i}$}]{} -- (0,0);
    \draw[ultra thick, green] (0,0) -- (1,1) node[label=left:$\scriptstyle{l}$]{} -- (2,0);
    \draw[ultra thick, green] (0,0) -- (1,-1) node[label=left:$\scriptstyle{l}$]{}  -- (2,0);
    \draw (2,0) -- (3,-1);
    \draw (2,0) -- (3,1);
    \node[dot] at (0,0) {};
    \node[dot] at (2,0) {};
  \end{tikzpicture}
  \right) = \\
  &= \Phi(1-q^2)z_1 \cdot (z_2 - q^{2m+2} z_1) + \Phi^2(q^4 z_1 - z_2) \cdot \frac{1}{\Phi} q^{2\res_m(i-1)} (1-q^2) z_1 + {} \\
  & \quad + \sum_{l\neq i} -\Phi^2(1-q^2) \left\{\begin{smallmatrix*}[l] q^2 z_1 & \text{if } i < l \\ z_2 & \text{if } l < i \end{smallmatrix*} \right\} \cdot \frac{1}{\Phi} q^{2\res_m(l-1)}(1-q^2)z_1 \\
  &= \Phi(1-q^2) z_1 \Biggl(z_2 - q^{2m+2} z_1 + (q^4 z_1 - z_2)q^{2(i-1)} + {} \\[-0.5em]
  & \hspace{3cm} - z_2 \sum_{l<i} (1-q^2) q^{2(l-1)} - q^2 z_1 \sum_{l>i} (1-q^2) q^{2(l-1)} \Biggr) \\
  &= \Phi(1-q^2) z_1 \Bigl(z_2 - q^{2m+2} z_1 + (q^4 z_1 - z_2)q^{2i-2} - z_2(1-q^{2i-2}) - q^2 z_1 (q^{2i} - q^{2m}) \Bigr) = 0.
\end{align*}

Finally, the last remaining partition function is
\begin{align*}
  Z\left(
    \begin{tikzpicture}[baseline={(0,-0.5ex)}, scale = \scale]
    \draw[ultra thick, red] (-1,-1) node[label={[label distance=-1ex]left:$\scriptstyle{i}$}]{} -- (0,0);
    \draw[ultra thick, red] (-1,1) node[label={[label distance=-1ex]left:$\scriptstyle{i}$}]{} -- (0,0);
    \draw (0,0) -- (1,1) -- (2,0);
    \draw (0,0) -- (1,-1) -- (2,0);
    \draw[ultra thick, blue] (2,0) -- (3,-1) node[label={[label distance=-1ex]right:$\scriptstyle{j}$}]{};
    \draw[ultra thick, blue] (2,0) -- (3,1) node[label={[label distance=-1ex]right:$\scriptstyle{j}$}]{};
    \node[dot] at (0,0) {};
    \node[dot] at (2,0) {};
  \end{tikzpicture}
  +
  \sum_{l}
  \begin{tikzpicture}[baseline={(0,-0.5ex)}, scale = \scale]
    \draw[ultra thick, red] (-1,-1) node[label={[label distance=-1ex]left:$\scriptstyle{i}$}]{} -- (0,0);
    \draw[ultra thick, red] (-1,1) node[label={[label distance=-1ex]left:$\scriptstyle{i}$}]{} -- (0,0);
    \draw[ultra thick, green] (0,0) -- (1,1) node[label=left:$\scriptstyle{l}$]{} -- (2,0);
    \draw[ultra thick, green] (0,0) -- (1,-1) node[label=left:$\scriptstyle{l}$]{}  -- (2,0);
    \draw[ultra thick, blue] (2,0) -- (3,-1) node[label={[label distance=-1ex]right:$\scriptstyle{j}$}]{};
    \draw[ultra thick, blue] (2,0) -- (3,1) node[label={[label distance=-1ex]right:$\scriptstyle{j}$}]{};    \node[dot] at (0,0) {};
    \node[dot] at (2,0) {};
  \end{tikzpicture}
  \right).
\end{align*}
We will split this computation into two cases: $i=j$ and $i\neq j$.
For $i = j$ we get that the partition function equals
\begin{align*}
  \MoveEqLeft
  Z\left(
    \begin{tikzpicture}[baseline={(0,-0.5ex)}, scale = \scale]
    \draw[ultra thick, red] (-1,-1) node[label={[label distance=-1ex]left:$\scriptstyle{i}$}]{} -- (0,0);
    \draw[ultra thick, red] (-1,1) node[label={[label distance=-1ex]left:$\scriptstyle{i}$}]{} -- (0,0);
    \draw (0,0) -- (1,1) -- (2,0);
    \draw (0,0) -- (1,-1) -- (2,0);
    \draw[ultra thick, red] (2,0) -- (3,-1) node[label={[label distance=-1ex]right:$\scriptstyle{i}$}]{};
    \draw[ultra thick, red] (2,0) -- (3,1) node[label={[label distance=-1ex]right:$\scriptstyle{i}$}]{};
    \node[dot] at (0,0) {};
    \node[dot] at (2,0) {};
  \end{tikzpicture}
  +
  \begin{tikzpicture}[baseline={(0,-0.5ex)}, scale = \scale]
    \draw[ultra thick, red] (-1,-1) node[label={[label distance=-1ex]left:$\scriptstyle{i}$}]{} -- (0,0);
    \draw[ultra thick, red] (-1,1) node[label={[label distance=-1ex]left:$\scriptstyle{i}$}]{} -- (0,0);
    \draw[ultra thick, red] (0,0) -- (1,1) node[label=left:$\scriptstyle{i}$]{} -- (2,0);
    \draw[ultra thick, red] (0,0) -- (1,-1) node[label=left:$\scriptstyle{i}$]{}  -- (2,0);
    \draw[ultra thick, red] (2,0) -- (3,-1) node[label={[label distance=-1ex]right:$\scriptstyle{i}$}]{};
    \draw[ultra thick, red] (2,0) -- (3,1) node[label={[label distance=-1ex]right:$\scriptstyle{i}$}]{};
    \node[dot] at (0,0) {};
    \node[dot] at (2,0) {};
  \end{tikzpicture}
  +
  \sum_{l\neq i}
  \begin{tikzpicture}[baseline={(0,-0.5ex)}, scale = \scale]
    \draw[ultra thick, red] (-1,-1) node[label={[label distance=-1ex]left:$\scriptstyle{i}$}]{} -- (0,0);
    \draw[ultra thick, red] (-1,1) node[label={[label distance=-1ex]left:$\scriptstyle{i}$}]{} -- (0,0);
    \draw[ultra thick, green] (0,0) -- (1,1) node[label=left:$\scriptstyle{l}$]{} -- (2,0);
    \draw[ultra thick, green] (0,0) -- (1,-1) node[label=left:$\scriptstyle{l}$]{}  -- (2,0);
    \draw[ultra thick, red] (2,0) -- (3,-1) node[label={[label distance=-1ex]right:$\scriptstyle{i}$}]{};
    \draw[ultra thick, red] (2,0) -- (3,1) node[label={[label distance=-1ex]right:$\scriptstyle{i}$}]{};
    \node[dot] at (0,0) {};
    \node[dot] at (2,0) {};
  \end{tikzpicture}
  \right) = {} 
  \\
  &= \Phi(1-q^2)z_1 \cdot \frac{1}{\Phi} q^{2\res^m(1-i)-2}(1-q^2)z_2 + \Phi^2(q^4z_1 - z_2) \cdot (q^{2m-2}z_1 - z_2)/\Phi^2 + {} \\
  & \quad + \sum_{l\neq i} -\Phi^2(1-q^2) \Bigl\{ \begin{smallmatrix*}[l] q^2 z_1 & \text{if } i < l \\ z_2 & \text{if } l < i \end{smallmatrix*} \Bigr\} \cdot \frac{1}{\Phi^2}q^{2\res_m(l-i)-2} (1-q^2) \Bigl\{ \begin{smallmatrix*}[l] z_1 & \text{if } l < i \\ z_2 & \text{if } i < l \end{smallmatrix*} 
  \\
  &= (1-q^2)z_1 \cdot q^{2(m+1-i)-2}(1-q^2)z_2 + (q^4z_1 - z_2) \cdot (q^{2m-2}z_1 - z_2) + {} \\
  & \quad - (1-q^2) z_1 z_2 \sum_{l\neq i} (1-q^2) q^{2\res_m(l-i)-2} \Bigl\{ \begin{smallmatrix*}[l] q^2 & \text{if } i < l \\ 1 & \text{if } l < i \end{smallmatrix*} 
\end{align*}
where we have used that $\res^m(1-i) = m+1-i$ for $i \in \{1, \ldots m\}$.
We split up the sum into two parts to obtain
\begin{multline*}
  \sum_{l\neq i} (1-q^2) q^{2\res_m(l-i)-2} \Bigl\{ \begin{smallmatrix*}[l] q^2 & \text{if } i < l \\ 1 & \text{if } l < i \end{smallmatrix*} = \sum_{l<i} (1-q^2) q^{2(m+l-i)-2} + \sum_{l>i} (1-q^2) q^{2(l-i)} = \\
    = q^{2(m+1-i)-2} - q^{2(m+i-1-i)} + q^{2(i+1-i)} - q^{2(m-i)+2}.
\end{multline*}
Thus we obtain the partition function
\begin{align*}
&= (1-q^2)z_1 \cdot q^{2(m-i)}(1-q^2)z_2 + (q^4z_1 - z_2) \cdot (q^{2m-2}z_1 - z_2) + {} \\
  & \quad - (1-q^2) z_1 z_2 (q^{2(m-i)} - q^{2(m-1)} + q^{2} - q^{2(m-i)+2}) = C.
\end{align*}

For $i \neq j$ we get that the partition function equals
\begin{align*}
  \MoveEqLeft
  Z\left(
    \begin{tikzpicture}[baseline={(0,-0.5ex)}, scale = \scale]
    \draw[ultra thick, red] (-1,-1) node[label={[label distance=-1ex]left:$\scriptstyle{i}$}]{} -- (0,0);
    \draw[ultra thick, red] (-1,1) node[label={[label distance=-1ex]left:$\scriptstyle{i}$}]{} -- (0,0);
    \draw (0,0) -- (1,1) -- (2,0);
    \draw (0,0) -- (1,-1) -- (2,0);
    \draw[ultra thick, blue] (2,0) -- (3,-1) node[label={[label distance=-1ex]right:$\scriptstyle{j}$}]{};
    \draw[ultra thick, blue] (2,0) -- (3,1) node[label={[label distance=-1ex]right:$\scriptstyle{j}$}]{};
    \node[dot] at (0,0) {};
    \node[dot] at (2,0) {};
  \end{tikzpicture}
  +
  \begin{tikzpicture}[baseline={(0,-0.5ex)}, scale = \scale]
    \draw[ultra thick, red] (-1,-1) node[label={[label distance=-1ex]left:$\scriptstyle{i}$}]{} -- (0,0);
    \draw[ultra thick, red] (-1,1) node[label={[label distance=-1ex]left:$\scriptstyle{i}$}]{} -- (0,0);
    \draw[ultra thick, red] (0,0) -- (1,1) node[label=left:$\scriptstyle{i}$]{} -- (2,0);
    \draw[ultra thick, red] (0,0) -- (1,-1) node[label=left:$\scriptstyle{i}$]{}  -- (2,0);
    \draw[ultra thick, blue] (2,0) -- (3,-1) node[label={[label distance=-1ex]right:$\scriptstyle{j}$}]{};
    \draw[ultra thick, blue] (2,0) -- (3,1) node[label={[label distance=-1ex]right:$\scriptstyle{j}$}]{};
    \node[dot] at (0,0) {};
    \node[dot] at (2,0) {};
  \end{tikzpicture}
  +
  \begin{tikzpicture}[baseline={(0,-0.5ex)}, scale = \scale]
    \draw[ultra thick, red] (-1,-1) node[label={[label distance=-1ex]left:$\scriptstyle{i}$}]{} -- (0,0);
    \draw[ultra thick, red] (-1,1) node[label={[label distance=-1ex]left:$\scriptstyle{i}$}]{} -- (0,0);
    \draw[ultra thick, blue] (0,0) -- (1,1) node[label=left:$\scriptstyle{j}$]{} -- (2,0);
    \draw[ultra thick, blue] (0,0) -- (1,-1) node[label=left:$\scriptstyle{j}$]{}  -- (2,0);
    \draw[ultra thick, blue] (2,0) -- (3,-1) node[label={[label distance=-1ex]right:$\scriptstyle{j}$}]{};
    \draw[ultra thick, blue] (2,0) -- (3,1) node[label={[label distance=-1ex]right:$\scriptstyle{j}$}]{};
    \node[dot] at (0,0) {};
    \node[dot] at (2,0) {};
  \end{tikzpicture}
  +
  \sum_{l\neq i,j}
  \begin{tikzpicture}[baseline={(0,-0.5ex)}, scale = \scale]
    \draw[ultra thick, red] (-1,-1) node[label={[label distance=-1ex]left:$\scriptstyle{i}$}]{} -- (0,0);
    \draw[ultra thick, red] (-1,1) node[label={[label distance=-1ex]left:$\scriptstyle{i}$}]{} -- (0,0);
    \draw[ultra thick, green] (0,0) -- (1,1) node[label=left:$\scriptstyle{l}$]{} -- (2,0);
    \draw[ultra thick, green] (0,0) -- (1,-1) node[label=left:$\scriptstyle{l}$]{}  -- (2,0);
    \draw[ultra thick, blue] (2,0) -- (3,-1) node[label={[label distance=-1ex]right:$\scriptstyle{j}$}]{};
    \draw[ultra thick, blue] (2,0) -- (3,1) node[label={[label distance=-1ex]right:$\scriptstyle{j}$}]{};
    \node[dot] at (0,0) {};
    \node[dot] at (2,0) {};
  \end{tikzpicture}
  \right) = {} 
  \\
  &= \Phi(1-q^2)z_1 \cdot \frac{1}{\Phi} q^{2\res^m(1-j)-2}(1-q^2)z_2 + {} \\
  & \quad + \Phi^2(q^4z_1 - z_2) \cdot \frac{1}{\Phi^2} q^{2\res_m(i-j)-2}(1-q^2) \Bigl\{ \begin{smallmatrix*}[l] z_1 & \text{if } i < j \\ z_2 & \text{if } j < i \end{smallmatrix*} \Bigr\} + {} \\
  & \quad - \Phi^2(1-q^2) \Bigl\{ \begin{smallmatrix*}[l] q^2 z_1 & \text{if } i < j \\ z_2 & \text{if } j < i \end{smallmatrix*} \Bigr\} \cdot (q^{2m-2} z_1 - z_2)/\Phi^2 + {} \\
  & \quad + \sum_{l\neq i,j} -\Phi^2(1-q^2) \Bigl\{ \begin{smallmatrix*}[l] q^2 z_1 & \text{if } i < l \\ z_2 & \text{if } l < i \end{smallmatrix*} \Bigr\} \cdot \frac{1}{\Phi^2}q^{2\res_m(l-j)-2} (1-q^2) \Bigl\{ \begin{smallmatrix*}[l] z_1 & \text{if } l < j \\ z_2 & \text{if } j < l \end{smallmatrix*} 
  \\
  &= (1-q^2) \Biggl( z_1  q^{2(m-j)}(1-q^2)z_2 + (q^4z_1 - z_2)  q^{2(i-j)-2} \Bigl\{ \begin{smallmatrix*}[l] q^{2m} z_1 & \text{if } i < j \\ z_2 & \text{if } j < i \end{smallmatrix*} \Bigr\} + {} \\
  & \quad -  \Bigl\{ \begin{smallmatrix*}[l] q^2 z_1 & \text{if } i < j \\ z_2 & \text{if } j < i \end{smallmatrix*} \Bigr\} (q^{2m-2} z_1 - z_2) - 
\underbrace{\sum_{l\neq i,j} \Bigl\{ \begin{smallmatrix*}[l] q^2 z_1 & \text{if } i < l \\ z_2 & \text{if } l < i \end{smallmatrix*} \Bigr\} q^{2(l-j)-2} (1-q^2) \Bigl\{ \begin{smallmatrix*}[l] q^{2m} z_1 & \text{if } l < j \\ z_2 & \text{if } j < l \end{smallmatrix*}\Bigr\}}_{A} \Biggr).
\end{align*}
Let us simplify the last sum $A$.
For $i<j$ we split the sum over $l$ into the cases: $l<i<j$, $i<l<j$ and $i<j<l$.
We then get that
\begin{align*}
  A &\stackrel{\mathclap{i<j}}{=} q^{2m-2j-2} z_1 z_2 \sum_{l<i<j} q^{2l}(1-q^2) + q^{2m-2j} z_1^2 \sum_{i<l<j} q^{2l}(1-q^2) + q^{-2j} z_1 z_2 \sum_{i<j<l} q^{2l}(1-q^2) \\
  &= q^{2m-2j-2} z_1 z_2 (q^2 - q^{2i}) + q^{2m-2j} z_1^2 (q^{2(i+1)} - q^{2j}) + q^{-2j} z_1 z_2 (q^{2(j+1)} - q^{2m+2}). 
\end{align*}
For $j<i$ we split $A$ into the cases: $l<j<i$, $j<l<i$ and $j<i<l$, and get that
\begin{align*}
  A &\stackrel{\mathclap{j<i}}{=} q^{2m-2j-2}z_1z_2  \sum_{l<j<i} q^{2l}(1-q^2) +  q^{-2j-2} z_2^2 \sum_{j<l<i} q^{2l}(1-q^2) + q^{-2j}z_1z_2 \sum_{j<i<l} q^{2l}(1-q^2) \\
  &= q^{2m-2j-2}z_1z_2 (q^2 - q^{2j}) +  q^{-2j-2} z_2^2 (q^{2(j+1)} - q^{2i}) + q^{-2j}z_1z_2 (q^{2(i+1)} - q^{2m+2}).
\end{align*}
After collecting all the terms this shows that the partition function is zero for both $i<j$ and $j<i$, and this completes the proof.

\bibliographystyle{alphaurl}
\bibliography{gammadelta}

\end{document}